\documentclass[10pt,reqno]{amsart}
\usepackage{amssymb,mathrsfs,color}
\usepackage{enumerate}
\usepackage{graphicx}
\usepackage{xcolor} 
\usepackage{geometry}
\usepackage{amsfonts,amssymb,amsmath}
\usepackage{slashed}
\usepackage{ mathrsfs }
\usepackage[titletoc,title]{appendix}
\usepackage{wrapfig}

\usepackage{cancel}

\usepackage{cite}

\usepackage{nth}

\usepackage{hyperref}

\usepackage{marginnote}

\usepackage{pgfplots}

\usepackage{slashed}

\definecolor{green}{rgb}{0,0.8,0} 

\definecolor{deepgreen}{cmyk}{1,0,1,0.5}

\newcommand{\pmat}[1]{\begin{pmatrix} #1 \end{pmatrix}}

\newcommand{\Del}[1]{}

\numberwithin{equation}{section}

\newtheorem{theorem}{Theorem}[section]
\newtheorem{lemma}[theorem]{Lemma}
\newtheorem{proposition}[theorem]{Proposition}
\newtheorem{remark}[theorem]{Remark}

\newcommand{\angles}[2]{\langle #1,#2\rangle}

\renewcommand{\div}{\mathrm{div}\,}

\newcommand{\tr}{\textrm{tr}}

\newcommand{\barg}{{\overline g}}

\newcommand{\barA}{{\overline A}}

\newcommand{\barE}{{\overline E}}

\newcommand{\barH}{{\overline H}}

\newcommand{\barR}{{\overline R}}

\newcommand{\barGamma}{\overline{\Gamma}}

\renewcommand{\hbar}{{\underline h}}

\newcommand{\bbR}{\mathbb R}

\newcommand{\calA}{\mathcal A}
\newcommand{\calB}{\mathcal B}
\newcommand{\calC}{\mathcal C}

\newcommand{\calE}{\mathcal E}
\newcommand{\calF}{\mathcal F}
\newcommand{\calG}{\mathcal G}
\newcommand{\calH}{\mathcal H}
\newcommand{\calI}{\mathcal I}
\newcommand{\calJ}{\mathcal J}
\newcommand{\calK}{\mathcal K}
\newcommand{\calL}{\mathcal L}

\newcommand{\calP}{\mathcal P}

\newcommand{\calT}{\mathcal T}

\newcommand{\frakn}{\mathfrak n}

\newcommand{\tila}{{\tilde{a}}}

\newcommand{\tiln}{{\tilde{n}}}

\newcommand{\tilp}{{\tilde{p}}}

\newcommand{\tilB}{{\tilde{B}}}

\newcommand{\tilD}{{\tilde{D}}}

\newcommand{\tilF}{{\tilde{F}}}

\newcommand{\tilG}{{\tilde{G}}}

\newcommand{\tilI}{{\tilde{I}}}

\newcommand{\tilJ}{{\tilde{J}}}

\newcommand{\tilK}{{\tilde{K}}}

\newcommand{\tilP}{{\tilde{P}}}

\newcommand{\tilR}{{\tilde{R}}}

\newcommand{\tilS}{{\tilde{S}}}

\newcommand{\tilV}{{\tilde{V}}}

\newcommand{\tilX}{{\tilde{X}}}

\newcommand{\tilY}{{\tilde{Y}}}

\newcommand{\hate}{{\hat{e}}}

\newcommand{\hatP}{{\hat{P}}}

\newcommand{\hatT}{{\hat{T}}}

\newcommand{\hatV}{{\hat{V}}}

\newcommand{\scE}{{\mathscr{E}}}

\newcommand{\curl}{\mathrm{curl}\,}

\newcommand{\snabla}{{\slashed{\nabla}}}

\newcommand{\ud}{\mathrm{d}}

\newcommand{\tiltheta}{{\tilde{\theta}}}

\makeatletter
\newsavebox{\@brx}
\newcommand{\llangle}[1][]{\savebox{\@brx}{\(\m@th{#1\langle}\)}%
  \mathopen{\copy\@brx\kern-0.5\wd\@brx\usebox{\@brx}}}
\newcommand{\rrangle}[1][]{\savebox{\@brx}{\(\m@th{#1\rangle}\)}%
  \mathclose{\copy\@brx\kern-0.5\wd\@brx\usebox{\@brx}}}
\makeatother
\newcommand{\bangles}[2]{\llangle #1,#2\rrangle}

\newcommand{\Ric}{{\mathrm{Ric}}}

\newcommand{\barnabla}{{\overline{\nabla}}}

\usepackage{nth}

\newcommand{\hatTheta}{{\hat{\Theta}}}
\newcommand{\tilcalC}{{\tilde{\calC}}}
\newcommand{\tiltilcalC}{{\tilde{\tilde{\calC}}}}

\newcommand{\chD}{{\check{D}}}

\newcommand{\chE}{{\check{E}}}
\newcommand{\chF}{{\check{F}}}
\newcommand{\chH}{{\check{H}}}

\newcommand{\chcalI}{{\check{\calI}}}
\newcommand{\chcalJ}{{\check{\calJ}}}
\newcommand{\chGamma}{{\check{\Gamma}}}
\newcommand{\chTheta}{{\check{\Theta}}}
\newcommand{\chX}{{\check{X}}}
\newcommand{\chcurl}{{\check{\curl}}}
\newcommand{\chdiv}{{\check{\div}}}
\newcommand{\chR}{{\check{R}}}
\newcommand{\chK}{{\check{K}}}
\newcommand{\chW}{{\check{W}}}
\newcommand{\chcalK}{{\check{\calK}}}
\newcommand{\chcalB}{{\check{\calB}}}

\newcommand{\che}{{\check{e}}}

\newcommand{\chgamma}{{\check{\gamma}}}
\newcommand{\cha}{{\check{a}}}
\newcommand{\chnabla}{\check{\nabla}}

\newcommand{\tilOmega}{{\tilde{\Omega}}}
\newcommand{\tilDelta}{{\tilde{\Delta}}}

\begin{document}
\title[Free Boundary Hard Phase Model]{Well-posedness for the free boundary hard phase model in general relativity}
\author{Shuang Miao
	\and  Sohrab Shahshahani}
\thanks{S. Miao was supported by the Fundamental Research Funds for the Central Universities in China and NSFC grant 12071360.  S. Shahshahani was supported by the Simons Foundation grant 639284.}
\maketitle
\begin{abstract}
	The hard phase model describes a relativistic barotropic and irrotational fluid with sound speed equal to the speed of light. In the framework of general relativity, the fluid, as a matter field, affects the geometry of the background spacetime. Therefore the motion of the fluid must be coupled to the Einstein equations which describe the structure of the underlying spacetime. In this work we prove a priori estimates and well-posedness in Sobolev spaces for this model with free boundary. Estimates for the curvature are derived using the Bianchi equations in a frame that is parallel transported by the fluid velocity. The fluid velocity is also decomposed with respect to this parallel frame, and its components are estimated using a coupled interior-boundary system of wave equations.
\end{abstract}
\section{Introduction}
Einstein's equations coupled to an ideal fluid are
\begin{align}\label{eq:Einstein1}
\begin{split}
G_{\mu\nu}=T_{\mu\nu},
\end{split}
\end{align}
where $G$ is the Einstein tensor
\begin{align*}
\begin{split}
G=\Ric -\frac{1}{2}S g,
\end{split}
\end{align*}
with $g$ the metric, $S$ the scalar curvature, and $\Ric$ the Ricci curvature. $T$ is the energy momentum tensor, which in the fluid domain $\Omega$ is defined as
\begin{align*}
\begin{split}
T^{\mu\nu}=(\rho+p)u^\mu u^\nu+p(g^{-1})^{\mu\nu},\qquad \mathrm{in~}\Omega,
\end{split}
\end{align*}
and in the exterior $\Omega^c$ as
\begin{align*}
\begin{split}
T^{\mu\nu}=0,\qquad \mathrm{in~}\Omega^c.
\end{split}
\end{align*}
Here $u$ is a future directed unit timelike vector\footnote{That is, $g_{\alpha\beta}u^\alpha u^\beta =-1$, $u^0>0$.}, $\rho$ is the energy density, and $p$ is the fluid pressure. The motion of the fluid, which occupies the domain $\Omega$, is governed by the conservation laws
\begin{align}
&\nabla_\mu T^{\mu\nu}=0,\label{eq:relEuler1}\\
&\nabla_\mu I^\mu =0,\label{eq:relEuler2}
\end{align}
where $\nabla$ is the Levi-Civita connection of the metric $g$ and $I$ is the particle current
\begin{align*}
\begin{split}
I^\mu = \frakn u^\mu.
\end{split}
\end{align*}
Here $\frakn$ denotes the number density of particles. With $s$ denoting the entropy per particle and $\theta$ the temperature, the laws of thermodynamics state that $p$ and $\rho$ are non-negative functions of $\frakn$ and $s$, and 
\begin{align*}
\begin{split}
\theta=\frac{1}{\frakn}\frac{\partial \rho}{\partial s},\qquad p= \frakn \frac{\partial\rho}{\partial\frakn}-\rho.
\end{split}
\end{align*}
The sound speed $\eta$, which is assumed to satisfy $0\leq \eta \leq 1$, is defined by
\begin{align*}
\begin{split}
\eta^2=\big(\frac{\partial p }{\partial \rho}\big)_s.
\end{split}
\end{align*}
Here the speed of light is normalized to be $1$. The fluid is said to be barotropic if the pressure is a function of the density
\begin{align}\label{eq:pfrho}
\begin{split}
p=f(\rho).
\end{split}
\end{align}
For a barotropic fluid the equations \eqref{eq:relEuler1} and \eqref{eq:relEuler2} decouple, and $p$ and $\rho$ can be written as functions of a single variable $\sigma$ defined by
\begin{align*}
\begin{split}
p+\rho=\sigma \frac{\ud \rho}{\ud \sigma}.
\end{split}
\end{align*}
Assuming further that $f$ appearing in \eqref{eq:pfrho} is strictly increasing and that 
\begin{align*}
\begin{split}
F(p)=\int_0^p \frac{\ud p'}{p'+\rho}
\end{split}
\end{align*}
is finite, we define the (renormalized)  fluid velocity $V$ as
\begin{align*}
\begin{split}
V=\|V\| u
\end{split}
\end{align*}
where $\|V\|:=e^F$. Then, with
\begin{align*}
\begin{split}
G=\frac{p+\rho}{\|V\|^2},
\end{split}
\end{align*}
equation \eqref{eq:relEuler1} reduces to the following equations in the fluid domain $\Omega$ (see \cite{Ch-hp1} for a derivation)
\begin{align}\label{eq:relbar1}
\begin{split}
&\nabla_V V+\frac{1}{2}\nabla (\|V\|^2)=0,\\
&\nabla_\mu (G(\|V\|)V^\mu)=0.
\end{split}
\end{align}
The \emph{hard phase} model refers to a barotropic fluid which is in addition irrotational and satisfies $\eta=1$. Here irrotational means that $dV=0$, or equivalently that $V$ can be written as the gradient of a scalar function,
\begin{align*}
\begin{split}
V=\nabla\phi.
\end{split}
\end{align*}
The hard phase model is an idealized model for the physical situation where, during the gravitational collapse of the degenerate core of a massive star, the mass-energy density exceeds the nuclear saturation density. See \cite{Ch-video,Ch-hp1, F-P, Lich-book, Rez-book, Walecka, Zeldovich}. Under the normalization that the nuclear saturation density is 1, as derived in \cite{Ch-hp1}, the equations of state relating $p$ and $\rho$ to $\sigma$ in the hard phase model are
\begin{align*}
\begin{split}
p=\frac{1}{2}(\sigma^2-1)
\end{split}
\end{align*}
and
\begin{align*}
\begin{split}
\rho=\frac{1}{2}(\sigma^2+1),
\end{split}
\end{align*}
and $G$ and $\|V\|$ are given by
\begin{align*}
\begin{split}
\|V\|=\sigma,\qquad G=1.
\end{split}
\end{align*}
In this case, $\sigma$ is the enthalpy and in the fluid domain $\Omega$ satisfies $\sigma^2>1$. The equations \eqref{eq:relbar1} then reduce to
\begin{align*}
\begin{split}
\nabla_\mu V^\mu=0,\qquad -V_\mu V^\mu = \sigma^2,\qquad \mathrm{in~}\Omega.
\end{split}
\end{align*}
In terms of the velocity potential $\phi$, this is equivalent to
\begin{align*}
\begin{split}
\Box \phi =0,\qquad -\nabla_\mu\phi \nabla^\mu\phi=\sigma^2,\qquad \mathrm{in~}\Omega.
\end{split}
\end{align*}
Moreover, the energy momentum tensor $T$ and the particle current $I$ are given by
\begin{align}\label{eq:hpemtensor1}
\begin{split}
T^{\mu\nu}= V^\mu V^\nu-\frac{1}{2}(V_\alpha V^\alpha+1)(g^{-1})^{\mu\nu},\qquad I^\mu = V^\mu,\qquad \mathrm{in~}\Omega.
\end{split}
\end{align}
In terms of the velocity potential $\phi$, we get
\begin{align}\label{eq:Thp1}
T^{\mu\nu}=
\begin{cases}
\nabla^\mu\phi \nabla^\nu\phi-\frac{1}{2}(\nabla_\alpha\phi\nabla^\alpha\phi)(g^{-1})^{\mu\nu}\qquad&\mathrm{in~}\Omega\\
0\qquad&\mathrm{in~}\Omega^c
\end{cases}.
\end{align}

In this paper we study Einstein's equations coupled to a hard phase fluid with free boundary surrounded by vacuum. That the boundary is free refers to the fact that the fluid domain is a priori unknown and has to be determined by the boundary conditions. These boundary conditions are (here $\calT\partial\Omega$ denotes the tangent space to the boundary)
\begin{align*}
\begin{split}
&\sigma^2=1\qquad \mathrm{on~}\partial\Omega,\\
&V\vert_{\partial\Omega} \in \calT \partial\Omega.
\end{split}
\end{align*}
The first condition is equivalent to $p=0$ on $\partial\Omega$ and is the analogue of the vanishing of surface tension in the Newtonian setting. The second condition states that a fluid particle on the boundary will remain on the boundary for all time (recall that $V$ is parallel to the fluid velocity $u$). The restriction to the hard phase state is both because of the independent interest of this model as discussed above, and because, as described in \cite{MSW1}, this model captures the important mathematical features of more general barotropic equations of state. We refer the reader to \cite{MSW1}, in particular Appendix A, and \cite{YWang1} for more details about general barotropic fluids in the case of relativistic fluids in Minkowski spacetime. Using the notation
\begin{align*}
\begin{split}
\sigma^2=-V^\mu V_\mu,
\end{split}
\end{align*}
the hard phase fluid equations can be summarized as
\begin{align}\label{eq:relEuler3}
\begin{cases}
\nabla_\mu V^\mu=0,\quad dV=0\quad&\mathrm{in~}\Omega\\
\sigma^2=1\quad&\mathrm{on~}\partial\Omega\\
V\vert_{\partial\Omega}\in\calT\partial\Omega
\end{cases}.
\end{align}

Equation \eqref{eq:relEuler3} describe the motion of the fluid for a given metric. Next, we use \eqref{eq:Einstein1} to derive working equations for the evolution of the metric which is coupled to the fluid.
As usual for the Einstein equations, in view of the diffeomorphism invariance of the equations, we need to fix coordinates or a gauge in order to study well-posedness. In this work we choose Lagrangian coordinates, $(t,x)=(x^0,x^1,x^2,x^3)$ so that $\frac{V}{\|V\|}=\partial_t$ in the fluid domain. Since $V$ is required to be tangential to $\partial\Omega$, by definition the fluid domain (restricted to $\{0\leq x^0\leq T\}$) in our coordinates is given by the product
\begin{align}\label{eq:Omegalag1}
\begin{split}
\Omega_0^T=[0,T]\times \Omega_0,
\end{split}
\end{align}
where in general
\begin{align*}
\begin{split}
\Omega_t=\Omega\cap \{x^0=t\},
\end{split}
\end{align*}
and $\Omega_0\subseteq\{x^0=0\}$ is the initial domain, which is assumed to be diffeomorphic to the unit ball. We also use the notation
\begin{align*}
\begin{split}
\partial\Omega_0^T=[0,T]\times \partial\Omega_t,
\end{split}
\end{align*}
and when it is clear from the context, write $\Omega$ and $\partial\Omega$ for $\Omega_0^T$ and $\partial\Omega_0^T$, respectively. The fluid velocity $V$ is extended to the exterior $\Omega^c$ using Sobolev extensions such that $V$ is always parallel to $\partial_t$, and $V=\partial_t$ outside a compact set. Note that since our fluid domain is in product form this can be achieved by Sobolev extension on each slice, and that $\partial_t$ and the extension commute. We work with an orthonormal frame $\{e_I\}_{I=0,1,2,3}$ with 
\begin{align*}
\begin{split}
g(e_I,e_I)=\epsilon_I,\qquad \epsilon_0=-1,~\epsilon_\tilI=1,~\tilI=1,2,3.
\end{split}
\end{align*}
We choose $\{e_I\}$ such that initially
\begin{align*}
\begin{split}
e_\tilI\vert_{t=0}\mathrm{~tangent~to~} \{t=0\},\quad \tilI=1,2,3,
\end{split}
\end{align*}
and require that the frame elements be parallel transported by the fluid velocity, that is, $\nabla_Ve_I=0$. The connection coefficients for the frame will be denoted by $\Gamma_{IJ}^K$, that is,
\begin{align*}
\begin{split}
\nabla_{e_I}e_J=\sum_K\Gamma_{IJ}^Ke_K.
\end{split}
\end{align*}
We write the fluid velocity $V$ as
\begin{align*}
\begin{split}
V=\sum_{I}\Theta^I e_I,
\end{split}
\end{align*}
and introduce the renormalized version $\hatV= \frac{1}{\sqrt{\sigma^2}}V=\partial_t$, with
\begin{align*}
\begin{split}
\hatV= \hatTheta^I e_I,\qquad \hatTheta^I=\frac{1}{\sqrt{\sigma^2}}\Theta^I.
\end{split}
\end{align*}
In terms of $\Theta^I$, equations \eqref{eq:relEuler3} become (the tangency condition being already incorporated into the definition \eqref{eq:Omegalag1} of $\Omega$)
\begin{equation}\label{eq:relEuler4}
\begin{cases}
\nabla_I \Theta^I=0,\quad&\mathrm{in~}\Omega\\
\nabla_I\Theta_J-\nabla_J\Theta_I=0\quad&\mathrm{in~}\Omega\\
\Theta_I\Theta^I=-1\quad&\mathrm{on~}\partial\Omega
\end{cases}.
\end{equation}
Here we have used the notation
\begin{align*}
\begin{split}
\nabla_J\Theta^I = D_J\Theta^I+\Gamma_{JK}^I\Theta^K,
\end{split}
\end{align*}
and
\begin{align*}
\begin{split}
\Theta_0=-\Theta^0,\qquad \Theta_\tilI=\Theta^\tilI,~\tilI=1,2,3.
\end{split}
\end{align*}
The frame components of the curvature tensor are denoted by
\begin{align*}
\begin{split}
R_{IJKL}=R(e_{I},e_{J},e_{K},e_{L}),
\end{split}
\end{align*}
and the Ricci tensor and scalar curvature are given, respectively, by
\begin{align*}
\begin{split}
R_{IJ}= \sum_K \epsilon_K R_{IKJK},
\end{split}
\end{align*}
and
\begin{align*}
\begin{split}
S=\sum_I \epsilon_I R_{II}.
\end{split}
\end{align*}
Note that in view of the discontinuity of the energy momentum tensor $T$ across the fluid boundary $\partial\Omega$, the curvature is an $L^2$ function, which, under suitable initial regularity assumptions, will be more regular in both $\Omega$ and $\Omega^c$. The condition $\nabla_V e_I=0$ yields transport equations for $e_I$ in terms of  $\Gamma_{IJ}^K$ and $\Theta^I$, and for $\Gamma_{IJ}^K$ in terms of $R_{IJKL}$ and $\Theta^I$. Indeed, as we will show below in Section \ref{subsec:frameeqs}, $e_I$ and $\Gamma_{IJ}^K$ satisfy
\begin{align}\label{eq:frame1}
\begin{split}
&\partial_t e_\tilI=-(D_{\tilI}\hatTheta^J) e_J - \hatTheta^J\Gamma_{\tilI J}^K e_K,\quad \tilI = 1,2,3,\\
&e_0= \frac{1}{\hatTheta^0}(\hatV-\hatTheta^\tilI e_\tilI),
\end{split}
\end{align}
and
\begin{align}\label{eq:connection1}
\begin{split}
&\partial_t \Gamma_{\tilI J}^K = \hatTheta^I(R^K_{\phantom{K}JI\tilI}-\Gamma_{MJ}^K\Gamma^M_{\tilI I})-\Gamma^K_{IJ}D_\tilI\hatTheta^I,\quad \tilI= 1,2,3,\\
&\Gamma^K_{0J}=-\frac{\hatTheta^\tilI}{\hatTheta^0}\Gamma^K_{\tilI J}.
\end{split}
\end{align}
Here and in what follows, $D_I$  denotes the derivative in the direction of $e_I$ applied to a scalar function,
\begin{align*}
\begin{split}
D_I f=e_{I}(f):= e_I^\mu \partial_\mu f.
\end{split}
\end{align*}
Finally, to write the Einstein equations \eqref{eq:Einstein1} in the frame, note that taking the trace of \eqref{eq:Einstein1} gives
\begin{align*}
\begin{split}
S= \Theta_K\Theta^K +2\quad \mathrm{in~}\Omega,\qquad S=0\quad \mathrm{in~}\Omega^c.
\end{split}
\end{align*}
The frame formulation of \eqref{eq:Einstein1} then becomes
\begin{align}\label{eq:Einstein2}
\begin{split}
R_{IJ}=\begin{cases}\Theta_I\Theta_I+\frac{1}{2}\delta_{IJ}\qquad&\mathrm{in~}\Omega\\ 0\qquad&\mathrm{in~}\Omega^c\end{cases}.
\end{split}
\end{align}
The Einstein-Euler equations in the frame as given in  \eqref{eq:relEuler4}, \eqref{eq:frame1}, \eqref{eq:connection1}, and \eqref{eq:Einstein2}, are the working equations in the remainder of this article.  In fact, as described in Section~\ref{subsec:frameeqs}, we will work with derived quasilinear equations based on the contracted differential Bianchi equations for the curvature, and based on a material derivative applied to \eqref{eq:relEuler4} for the fluid quantities.  As is common with Einstein's equations, we will use the data for Einstein's equations \eqref{eq:Einstein1}, which are assumed to satisfy the usual constraint equations, to construct compatible data for \eqref{eq:relEuler4}, \eqref{eq:frame1}, \eqref{eq:connection1}, \eqref{eq:Einstein2}. We then prove that our our solution to \eqref{eq:relEuler4}, \eqref{eq:frame1}, \eqref{eq:connection1}, \eqref{eq:Einstein2} satisfies \eqref{eq:Einstein1} for this choice of data.

Before stating the main result we need to discuss one more condition, which is the relativistic analogue of the Taylor sign condition. Since $\sigma^2$ is constant on the boundary, its gradient is normal to $\partial\Omega$. The relativistic Taylor sign condition is that
\begin{align}\label{eq:Taylor1}
\begin{split}
\nabla_\mu \sigma^2 \nabla^\mu \sigma^2\geq c_0^2 >0.
\end{split}
\end{align}
This condition, which will be used crucially in the analysis, is the analogue of the Newtonian Taylor sign condition, in which case its failure is known to lead to instabilities, cf. \cite{GTay1}. It will follow from our bootstrap assumptions that if \eqref{eq:Taylor1} holds initially, then it will hold during the evolution with $c_0^2$ replaced a slightly smaller constant. We refer the reader to \cite{MSW1} for more discussion on this point.

Returning to \eqref{eq:relEuler4}, \eqref{eq:frame1}, \eqref{eq:connection1}, \eqref{eq:Einstein2}, we  assume that initially, at $t=0$, the following regularity and compatibility conditions are satisfied:
\begin{align}\label{eq:regularity1}
\begin{split}
\partial_x^p\partial_t^k\Theta^I,~\partial_x^p\partial_t^{k+1}\sigma^2\in L^2(\Omega_t),\quad &k\leq M+1, ~2p+k\leq M+2,\\
\partial_t^k\sigma^2\in H^1_0(\Omega_t),\quad&k\leq M+1,\\
\partial_t^{M}R_{IJKL},~\partial\partial_t^{M-1}R_{IJKL},~\partial^p\partial_t^kR_{IJKL}\in L^2(\Omega_t)\cap L^2(\Omega^c_t),\quad&2p+k\leq M,\\
\partial_t^k R_{IJKL}\in L^2(\bbR^3),\quad&k\leq M,\\
\partial^p\partial_t^k e_I\in L^2(\Omega_t)\cap L^2(\Omega_t^c),\quad&2p+k\leq M+1,\\
e_{I}-\partial_{I}\in L^2(\Omega_t)\cap L^2(\Omega_t^c),\quad& I=0,1,2,3,\\
\partial^p\partial_t^k \Gamma_{IJ}^K\in L^2(\Omega_t)\cap L^2(\Omega_t^c),\quad&2p+k\leq M+1 .
\end{split}
\end{align}
and, with $\Theta_0^I=\Theta^I(0)$, 
\begin{align}\label{eq:compatibility1}
\begin{cases}
\sigma^2_0:=-\sum_I\epsilon_I(\Theta^I_0)^2\geq 1,\quad \Theta^0_0>0,\quad &\mathrm{in~} \Omega_0\\
\sigma_0^2=1,\quad&\mathrm{on~}\partial\Omega_0\\
\nabla_I \sigma_0^2 \nabla^I\sigma_0^2\geq c_0^2>0,\quad&\mathrm{on~}\partial\Omega_0
\end{cases}.
\end{align}
The following is the main theorem of this work. Besides the compatibility and regularity assumptions \eqref{eq:regularity1} and \eqref{eq:compatibility1}, we have also assumed extra initial compatibility conditions on the data given by \eqref{def vanishing fluid int} and \eqref{def vanishing geom int}. As will be explained in Section~\ref{subsec:frameeqs} below, these conditions, which we have stated in Section~\ref{subsec:frameeqs} below to avoid cumbersome notation at this point, are simply the assumption that the equations and the geometric properties of the curvature tensor (such as the symmetries) hold initially.
\begin{theorem}\label{thm:main1}
Suppose initially the regularity and compatibility conditions \eqref{eq:regularity1}, \eqref{eq:compatibility1}, \eqref{def vanishing fluid int}, and \eqref{def vanishing geom int} are satisfied with $M$ sufficiently large. Then there is $T>0$ and a unique solution to \eqref{eq:relEuler4}, \eqref{eq:frame1}, \eqref{eq:connection1}, \eqref{eq:Einstein2}, in $\Omega=[0,T]\times \Omega_0$ satisfying \eqref{eq:regularity1} for all $t\in[0,T]$, and $$u:=\frac{1}{\|V\|}V = \frac{1}{\sqrt{-\sum_J\epsilon_J(\Theta^J)^2}}\Theta^I e_I$$ and $g$ defined by $$(g^{-1})^{\mu\nu}:= -e_0^\mu e_0^\nu+\sum_{\tilI=1}^3e_\tilI^\mu e_\tilI^\nu,$$ satisfy equations \eqref{eq:Einstein1}, \eqref{eq:relEuler1}, \eqref{eq:relEuler2}.
\end{theorem}
 
 Following Choquet-Bruhat's existence theorem in \cite{C-Bpaper1} for Einstein's equations in vacuum, well-posedness results have been established for the Einstein equations coupled to many different matter fields. Indeed, much progress has been made beyond the local theory in the past few decades.  Despite all the progress in this direction, in the presence of isolated bodies, especially with free boundaries, our understanding of even the local theory is very limited. A well-posedness theory for isolated bodies is of central importance as it is naturally the first step in any further analysis of the motion and interaction of gravitating bodies. With the exception of \cite{CB-Fri1,Ehlers,Rendall92}, which consider special solutions under symmetries or where the motion of the boundary is not tracked, the only work on well-posedness for isolated bodies in general relativity that we are aware of is \cite{AOS1}, which considers the very different case of solid elastic bodies. The free boundary problem for fluid bodies is, however, very different and includes many new analytical challenges even in the Newtonian setting. To the best of our knowledge the current work is the first to prove well-posedness for a free-boundary fluid equation in the setting of general relativity.  The type of fluid model considered in this work, where the energy density has a jump across the fluid boundary, is sometimes referred to as a liquid model. Well-posedness for such free boundary fluid models is already subtle in the Newtonian case, and the first satisfactory local theory in Sobolev spaces was developed only in the mid 1990s in \cite{Wu97,Wu99} for the case of water waves.\footnote{The interested reader is referred to these works and for instance \cite{MSW1, Wu16} for more on the history of the well-posedness theory in the Newtonian case.} More recently, well-posedness was established for relativistic liquids with free boundary in \cite{Oliynyk4,Oliynyk1,Oliynyk3,MSW1}, but in the case of a fixed background, that is, without coupling the relativistic fluid equations \eqref{eq:relEuler1}-\eqref{eq:relEuler2} to Einstein's equations. See also \cite{Oliynyk2} for the case of two spatial dimensions, \cite{Ginsberg1} for a priori estimates under smallness assumptions, \cite{Trakhinin} for an existence results using Nash-Moser iteration, and \cite{GinLin1}. For related developments in the gaseous case (where the energy density vanishes at the fluid boundary) on a fixed background see \cite{HSS1,JLM1, DIT1}.
 
In view of the remarks in the previous paragraph, this work should be considered as a natural continuation of our earlier work \cite{MSW1}. Indeed, the idea is to use the general setup developed in \cite{MSW1} to treat the relativistic fluid. However, a main difficulty is that in the presence of Einstein's equations one has to guarantee that geometric quantities do not break the regularity structure necessary to close the estimates for the fluid quantities. As discussed earlier in the introduction, this is achieved by working in a frame that is parallel transported along the fluid velocity. This frame formulation plays a crucial role in our analysis, and is discussed more carefully in Section~\ref{subsec:frameeqs} below. The idea of using frames to derive estimates in the study of Einstein's equations has a long history, and has proved especially useful in long time analysis of the dynamics. See for instance \cite{CK1}. Some of the choices we have made are inspired by the work \cite{Frie1} which studies various formulations of the Bianchi equations as a first order hyperbolic system. However, our final formulation is different from the ones considered in \cite{Frie1}, especially because the free boundary fluid is analyzed differently based on \cite{MSW1}.

In the remainder of the introduction we will describe the derived equations to which we alluded earlier, and then discuss the main ideas of the proof of Theorem~\ref{thm:main1}, in particular the a priori estimates. This is the content of Section~\ref{subsec:frameeqs}. The initial data are discussed in Section~\ref{subseq:data}. We will use lowercase roman indices $i,j,k,\dots\in{1,2,3}$ for the spatial coordinates, capital letters $I,J,K,\dots\in\{0,\dots,3\}$ for all frame indices, and $\tilI, \tilJ,\tilK,\dots\in\{1,2,3\}$ for the spatial frame components. We will use $x^0$ and $t$, and similarly $\partial_0$ and $\partial_t$, interchangeably. When the exact structure is not important, we will use $e$, $\Theta$, $\Gamma$, and $R$ to denote arbitrary frame components of the corresponding quantities. Scalar differentiation with respect to $e_I$ will be denoted by $D_I:=e_I^\mu \partial_\mu$, and when the exact structure is not important we simply write $D$ or $D^k$ for higher order derivatives.
\subsection{Equations in the frame and the outline of the proof}\label{subsec:frameeqs}
We start by deriving the transport equations  \eqref{eq:frame1} and \eqref{eq:connection1} for $\Gamma_{IJ}^K$ and the components $e_I^\mu$ of the frame elements in the Lagrangian coordinates $\{x^\mu\}$. These transport equations will be coupled to the main evolution equations which are the Bianchi equations for the curvature and the wave-boundary system for the fluid components $\Theta^I$.   Since $\hatV$ is timelike (recall that $\hatV=\frac{1}{\|V\|}V$), for $e_0$ to be unit, future directed, and orthogonal to $e_\tilI$, we must have 
\begin{align*}
\begin{split}
e_0=\frac{1}{\hatTheta^0}(\hatV-\hatTheta^{\tilI}e_\tilI).
\end{split}
\end{align*}
In coordinates, the requirement that $\hatV=\partial_t$ gives
\begin{align}\label{eq:e0mu1}
\begin{split}
e_0^\mu=\frac{1}{\hatTheta^0}(\delta_0^\mu-\hatTheta^{\tilI}e_{\tilI}^\mu),
\end{split}
\end{align}
so the components of $e_0$ are determined algebraically in terms of $\hatTheta$ and  the components of $e_\tilI$. To derive a transport equation for the components of $e_\tilI$ we use the fact that, with $\calL$ denoting the Lie derivative,
\begin{align*}
\begin{split}
\calL_{e_{J}}e_I=\nabla_{e_J}e_{I}-\nabla_{e_I}e_J=(\Gamma_{JI}^K-\Gamma_{IJ}^K)e_K.
\end{split}
\end{align*}
Choosing $J=0$ gives
\begin{align*}
\begin{split}
(\Gamma_{0\tilI}^K-\Gamma_{\tilI0}^K)e_K^\mu& = \frac{1}{\hatTheta^0}\partial_te_\tilI^\mu-\frac{\hatTheta^\tilJ}{\hatTheta^0}D_{\tilJ}(e_\tilI^\mu)-D_{\tilI}(\frac{1}{\hatTheta^0}(\delta_0^\mu-\hatTheta^\tilJ e_\tilJ^\mu))\\
&=\frac{1}{\hatTheta^0}\partial_te_\tilI^\mu-\frac{\hatTheta^\tilJ}{\hatTheta^0}\calL_{e_{\tilJ}}e_{\tilI}^\mu+\frac{D_\tilI\hatTheta^0}{\hatTheta^0}e_0^\mu+\frac{D_\tilI\hatTheta^\tilJ}{\hatTheta^0}e_{\tilJ}^\mu.
\end{split}
\end{align*}
Rearranging, and using $\hatTheta^J\Gamma_{J\tilI}^Ke_K=\nabla_{\hatV}e_\tilI=0$ we get
\begin{align}\label{eq:etransport1}
\begin{split}
\partial_te_\tilI^\mu=-e_J^\mu D_\tilI\hatTheta^J-\hatTheta^J\Gamma_{\tilI J}^Ke_K^\mu.
\end{split}
\end{align}
Equations \eqref{eq:e0mu1} and \eqref{eq:etransport1} describe how to recover $e_I^\mu$ from $\hatTheta^I$ and $\Gamma_{IJ}^K$. 
\begin{remark}\label{rem:transport e0mu}
Equations \eqref{eq:e0mu1},\eqref{eq:etransport1} and \eqref{eq:Gamma01} imply the following transport equation for $e_{0}^{\mu}$:
\begin{align*}
\partial_{t}e_{0}^{\mu}=&-D_{0}\hatTheta^{J}\cdot e_{J}^{\mu}-\hatTheta^{J}\Gamma_{0J}^{K}e_{K}^{\mu}.
\end{align*}
Indeed, in view of \eqref{eq:e0mu1} and \eqref{eq:etransport1}, 
\begin{align*}
\partial_{t}e_{0}^{\mu}=&-\frac{\partial_{t}\hatTheta^{0}}{(\hatTheta^{0})^{2}}(\delta_{0}^{\mu}-\hatTheta^{\tilI}e_{\tilI}^{\mu})-\frac{1}{\hatTheta^{0}}\partial_{t}\hatTheta^{\tilI}e_{\tilI}^{\mu}-\frac{1}{\hatTheta^{0}}\hatTheta^{\tilI}\partial_{t}e_{\tilI}^{\mu}\\
=&-\frac{\partial_{t}\hatTheta^{J}}{\hatTheta^{0}}e_{J}^{\mu}-\frac{\hatTheta^{\tilI}}{\hatTheta^{0}}\left(-e_{J}^{\mu}D_{\tilI}\hatTheta^{J}-\hatTheta^{J}\Gamma_{\tilI J}^{K}e_{K}^{\mu}\right).
\end{align*}
The desired result follows from  \eqref{eq:Gamma01} and the fact $\partial_{t}=\hatTheta^{0}D_{0}+\hatTheta^{\tilI}D_{\tilI}$ (which is a restatement of \eqref{eq:e0mu1}) .
\end{remark}
Next we derive a transport and algebraic equations for $\Gamma_{IJ}^K$ in terms of $e_I^\mu$, $\Theta^I$ and the curvature. For the algebraic equation observe that since $\nabla_\hatV e_J=0$,
\begin{align*}
\begin{split}
\Gamma_{0J}^Ke_K=\nabla_{e_0}e_J=\frac{1}{\hatTheta^0}\nabla_{\hatV-\hatTheta^\tilI e_\tilI}e_J=-\frac{\hatTheta^\tilI}{\hatTheta^0}\Gamma_{\tilI J}^K e_K.
\end{split}
\end{align*}
This shows that
\begin{align}\label{eq:Gamma01}
\begin{split}
\Gamma_{0J}^K=-\frac{\hatTheta^\tilI}{\hatTheta^0}\Gamma_{\tilI J}^K,
\end{split}
\end{align}
so that $\Gamma_{0J}^K$ is determined algebraically in terms of $\Theta^I$ and $\Gamma^I_{\tilI J}$. For the transport equations for the remaining components, first note that the condition $\nabla_\hatV e_I=0$ implies $\hatTheta^J\Gamma_{JI}^L=0$, so
\begin{align*}
\begin{split}
\hatTheta^I(D_I\Gamma_{\tilI J}^K-D_\tilI \Gamma_{IJ}^K)=\partial_t\Gamma_{\tilI J}^K+\Gamma_{I J}^KD_\tilI\hatTheta^I.
\end{split}
\end{align*}
On the other hand, by the Ricci identity 
\begin{align*}
\begin{split}
R^I_{\phantom{I}JKL}=D_K\Gamma_{LJ}^I-D_L\Gamma_{KJ}^I-\Gamma_{MJ}^I(\Gamma_{KL}^M-\Gamma_{LK}^M)+\Gamma_{KM}^I\Gamma_{LJ}^M-\Gamma_{LM}^I\Gamma_{KJ}^M.
\end{split}
\end{align*}
Putting these identities together, and again using $\hatTheta^J\Gamma_{JI}^L=0$, we arrive at
\begin{align}\label{eq:Gammatransport1}
\begin{split}
\partial_t\Gamma_{\tilI J}^K&=\hatTheta^I\big(R^K_{\phantom{K}JI\tilI}+\Gamma_{MJ}^K(\Gamma_{I\tilI}^M-\Gamma_{\tilI I}^M)+\Gamma^K_{\tilI M}\Gamma^M_{IJ}-\Gamma^K_{IM}\Gamma^M_{\tilI J}\big)-\Gamma_{I J}^KD_\tilI\hatTheta^I\\
&=\hatTheta^{I}\left(R^{K}{}_{JI\tilI}-\Gamma_{MJ}^{K}\Gamma_{\tilI I}^{M}\right)-\Gamma_{IJ}^{K}D_{\tilI}\hatTheta^{I}.
\end{split}
\end{align}
Equations \eqref{eq:e0mu1}, \eqref{eq:etransport1}, \eqref{eq:Gamma01}, and \eqref{eq:Gammatransport1} are precisely equations \eqref{eq:frame1} and \eqref{eq:connection1} mentioned in the previous subsection. They constitute our working equations for the connection and frame coefficients.
\begin{remark}\label{rem:transport Gamma0}  
Equations \eqref{eq:Gamma01} and \eqref{eq:Gammatransport1} imply the following transport equation for $\Gamma_{0J}^{K}$:
\begin{align*}
\partial_{t}\Gamma_{0J}^{K}=&\hatTheta^{I}R^{K}{}_{JI0}-\hatTheta^{I}\Gamma^{K}_{MJ}\Gamma_{0I}^{M}-D_{0}\hat\Theta^{I}\Gamma_{IJ}^{K}.
\end{align*}
Indeed, in view of \eqref{eq:Gamma01},
\begin{align*}
\partial_{t}\Gamma_{0J}^{K}=&-\frac{\partial_{t}\hatTheta^{\tilI}}{\hatTheta^{0}}\Gamma_{\tilI J}^{K}+\frac{\partial_{t}\hatTheta^{0}}{(\hatTheta^{0})^{2}}\hatTheta^{\tilI}\Gamma_{\tilI J}^{K}-\frac{\hatTheta^{\tilI}}{\hatTheta^{0}}\partial_{t}\Gamma_{\tilI J}^{K}=:A+B+C.
\end{align*}
By \eqref{eq:Gammatransport1},
\begin{align*}
C=&-\frac{\hatTheta^{\tilI}}{\hatTheta^{0}}\cdot\hatTheta^{I}\left(R^{K}{}_{JI\tilI}+\Gamma^{K}_{MJ}(\Gamma_{I\tilI}^{M}-\Gamma^{M}_{\tilI I})+\Gamma^{K}_{\tilI M}\Gamma^{M}_{IJ}-\Gamma_{IM}^{K}\Gamma^{M}_{\tilI J}\right)+\Gamma_{IJ}^{K}D_{\tilI}\hatTheta^{I}\cdot \frac{\hatTheta^{\tilI}}{\hatTheta^{0}}.
\end{align*}
For the first term on the right-hand side above, we have, using the symmetry properties of the curvature tensor,
\begin{align*}
-\frac{\hatTheta^{\tilI}}{\hatTheta^{0}}\cdot \hatTheta^{I}R^{K}{}_{JI\tilI}=\hatTheta^{I}R^{K}{}_{JI0}.
\end{align*}
For the second term, we have
\begin{align*}
-\frac{\hatTheta^{\tilI}}{\hatTheta^{0}}\cdot \hatTheta^{I}\Gamma_{MJ}^{K}\Gamma^{M}_{I\tilI}=	-\frac{\hatTheta^{\tilI}}{\hatTheta^{0}}\cdot \hatTheta^{\tilJ}\Gamma_{MJ}^{K}\Gamma^{M}_{\tilJ\tilI}-\hatTheta^{\tilI}\Gamma_{MJ}^{K}\Gamma^{M}_{0\tilI}=\hatTheta^{\tilI}\Gamma_{MJ}^{K}\Gamma^{M}_{0\tilI}-\hatTheta^{\tilI}\Gamma_{MJ}^{K}\Gamma^{M}_{0\tilI}=0.
\end{align*}
Therefore $C$ can be rewritten as
\begin{align*}
C&=\hatTheta^{I}R^{K}{}_{JI0}+\hatTheta^{I}\left(-\Gamma^{K}_{MJ}\Gamma_{0I}^{M}+\Gamma_{0M}^{K}\Gamma_{IJ}^{M}-\Gamma_{IM}^{K}\Gamma^{M}_{0J}\right)+\Gamma_{IJ}^{K}D_{\tilI}\hatTheta^{I}\cdot \frac{\hatTheta^{\tilI}}{\hatTheta^{0}}\\
&=\hatTheta^{I}R^{K}{}_{JI0}-\hatTheta^{I}\Gamma_{MJ}^{K}\Gamma_{0I}^{M}+\Gamma_{IJ}^{K}D_{\tilI}\hatTheta^{I}\cdot\frac{\hatTheta^{\tilI}}{\hatTheta^{0}}.
\end{align*}
For $A$ and $V$, using $\partial_{t}=\hatTheta^{0}D_{0}+\hatTheta^{\tilI}D_{\tilI}$ and \eqref{eq:Gamma01},
\begin{align*}
A=&-D_{0}\hatTheta^{\tilI}\Gamma^{K}_{\tilI J}-\frac{\hatTheta^{\tilJ}D_{\tilJ}\hatTheta^{\tilI}}{\hatTheta^{0}}\Gamma_{\tilI J}^{K},\\
B=&\frac{D_{0}\hatTheta^{0}}{\hatTheta^{0}}\hatTheta^{\tilI}\Gamma_{\tilI J}^{K}+\frac{\hatTheta^{\tilJ}D_{\tilJ}\hatTheta^{0}}{(\hatTheta^{0})^{2}}\hatTheta^{\tilI}\Gamma_{\tilI J}^{K}
=-D_{0}\hatTheta^{0}\Gamma_{0J}^{K}-\frac{\hatTheta^{\tilJ}D_{\tilJ}\hatTheta^{0}}{\hatTheta^{0}}\Gamma_{0J}^{K}.
\end{align*}
On the other hand, note that
\begin{align*}
\Gamma_{IJ}^{K}D_{\tilI}\hatTheta^{I}\cdot \frac{\hatTheta^{\tilI}}{\hatTheta^{0}}-\frac{\hatTheta^{\tilJ}D_{\tilJ}\hatTheta^{\tilI}}{\hatTheta^{0}}\Gamma_{\tilI J}^{K}-\frac{\hatTheta^{\tilJ}D_{\tilJ}\hatTheta^{0}}{\hatTheta^{0}}\Gamma_{0J}^{K}=0.
\end{align*}
It follows that
\begin{align*}
A+B+C=\hatTheta^{I}R^{K}{}_{JI0}-\hatTheta^{I}\Gamma^{K}_{MJ}\Gamma_{0I}^{M}-D_{0}\hat\Theta^{I}\Gamma_{IJ}^{K},
\end{align*}
as desired.
\end{remark}
We also note that the metric is determined by the requirement that $\{e_I\}$ is an orthonormal frame. That is, with $m_{IJ}$ and $(m^{-1})^{IJ}$ denoting the components of the Minkowski metric \emph{in rectangular coordinates}, we have
\begin{align}\label{eq:g1}
\begin{split}
m_{IJ}= g_{\mu\nu}e_{I}^\mu e_J^\nu,\qquad (g^{-1})^{\mu\nu}=(m^{-1})^{IJ}e_I^\mu e_J^\nu.
\end{split}
\end{align}

Next we turn to the fluid equations. To motivate the calculations, we briefly recall some of the relevant ideas in a Newtonian analogue of our problem, namely the water wave equations, and the relativistic problem on the fixed Minkowski background. For the Newtonian problem, let $\tilV$, $\tilp$, and $\tilOmega_t$ denote the fluid velocity, pressure, and fluid domain respectively, and let $\tiln$ denote the unit exterior normal of $\partial\Omega_t$. In \cite{Wu97,Wu99} Wu applied a material derivative $\tilD_t = \partial_t+\tilV\cdot\nabla$ to the Euler equations, to derive the following quasilinear equation for the components of $\tilV$ 
\begin{align}\label{eq:water1}
\begin{cases}
(D_t^2+\tila \nabla_\tiln)\tilV = -\nabla  D_t\tilp\qquad &\mathrm{on~}\partial\Omega_t\\
\Delta \tilV=0\qquad &\mathrm{in~}\tilOmega_t
\end{cases}.
\end{align}
Here $\tila:=-\frac{\partial \tilp}{\partial\tiln}$ is positive on the boundary\footnote{This is the Taylor sign condition which is proved to hold in \cite{Wu97,Wu99}.}, and $\nabla_\tiln$ is the Dirichlet-Neumann map. See also \cite{ZZ}. The second equation in \eqref{eq:water1} holds by the assumptions that the fluid is irrotational and incompressible, and the positivity of the Dirichlet-Neumann map is used to prove energy estimates for this equation. To show that \eqref{eq:water1} is quasilinear Wu expressed $\nabla \tilD_t\tilp$ in terms of $\tilV$ using boundary integrals. That this is possible is ultimately related to the fact that $\tilD_t\tilp$ and $\tilp$ satisfy elliptic equations of the form
\begin{align}\label{eq:pressure1}
\begin{cases}
\Delta\tilp= \tilF(\partial\tilV)\\
\Delta \tilD_t\tilp= \tilG(\partial \tilV,\partial^2 \tilp)
\end{cases}\quad\mathrm{in~}\Omega_t,\qquad \tilp=\tilD_t\tilp=0\qquad\mathrm{on~}\partial\Omega_t.
\end{align} 
We refer the reader to the introduction of \cite{MSW1} for a more detailed discussion. Motivated by these earlier works, in \cite{MSW1} we applied a material derivative $D_V=V^\mu\partial_\mu$ to the relativistic Euler equations on the fixed Minkowski background to obtain
\begin{align}\label{eq:Minkowski1}
\begin{cases}
(D_V^2+\frac{a}{2}D_n) V^\nu = -\frac{1}{2}\nabla^\nu D_V\sigma^2\qquad &\mathrm{on~}\partial\Omega\\
\Box (V^\nu)=0\qquad &\mathrm{in~}\Omega
\end{cases}.
\end{align}
Here $\Omega$ is the fluid domain, $n$ the (spacetime) exterior unit normal, $a=|\nabla\sigma^2|$, and the rectangular components $V^\nu$ of the fluid velocity are interpreted as scalar functions. Similarly, $\sigma^2$ and $D_V\sigma^2$ were shown to satisfy equations of the form
\begin{align}\label{eq:Minkowski2}
\begin{cases}
\Box \sigma^2= F(\partial V)\\
\Box D_V\sigma^2= G(\partial V, \partial^2\sigma^2)
\end{cases}\quad \mathrm{in~}\Omega,\qquad \sigma^2=1,~D_V\sigma^2=0\quad \mathrm{on~}\partial\Omega.
\end{align}
Despite the formal similarity between \eqref{eq:Minkowski1}, \eqref{eq:Minkowski2} and \eqref{eq:water1}, \eqref{eq:pressure1}, these equations have the important difference that while $\tilV$, $\tilp$, and $\tilD_t\tilp$ satisfy elliptic equations in the interior of the fluid domain, $V^\nu$, $\sigma^2$, and $D_V\sigma^2$ satisfy a hyperbolic equation, and the normal derivative appearing in \eqref{eq:Minkowski1} is the hyperbolic Dirichlet-Neumann map. Nevertheless, in \cite{MSW1} we were able to use \eqref{eq:Minkowski1} and \eqref{eq:Minkowski2} to derive a priori estimates and set up an iteration. 

Motivated by this discussion we now demonstrate how to derive analogous equations for $\Theta^I$ and $\sigma^2$. The main advantage of considering the frame components $\Theta^I$ of $V$ rather than arbitrary coordinate components $V^\nu$ is that, as we will see, only the connection coefficients $\Gamma$ and not their derivatives appear in the boundary equation for the fluid velocity. Our starting point is the first equation in \eqref{eq:relbar1} which we rewrite as
\begin{align*}
\begin{split}
\nabla_VV+\frac{1}{2}\nabla\sigma^2=0.
\end{split}
\end{align*}
Applying the covariant derivative $\nabla_V$ to this equation yields
\begin{align*}
\begin{split}
\nabla_V^2V+\frac{a}{2}\nabla_nV+\frac{1}{2}\nabla D_V\sigma^2=0.
\end{split}
\end{align*}
Here
\begin{align*}
\begin{split}
a:= \sqrt{\nabla_\mu\sigma^2\nabla^\mu \sigma^2},
\end{split}
\end{align*}
which in terms of the frame becomes
\begin{align}\label{eq:a1}
\begin{split}
a=\sqrt{\sum_I\epsilon_I(D_I\sigma^2)^2}.
\end{split}
\end{align}
By definition $n$ is (note that since $\sigma^2$ is constant on the boundary $\nabla\sigma^2$ is normal to $\partial\Omega$)
\begin{align*}
\begin{split}
n=\frac{1}{a}\nabla\sigma^2.
\end{split}
\end{align*}
 Taking the inner product with $e_I$ we get 
\begin{align*}
\begin{split}
0&=\epsilon_ID_V^2\Theta^I+\frac{a\epsilon_I}{2}D_n\Theta^I+\frac{\epsilon_{I}\epsilon_K}{2}\Gamma_{KJ}^I\Theta^JD_K\sigma^2+\frac{\epsilon_I}{2}D_ID_V\sigma^2.
\end{split}
\end{align*}
Rearranging gives
\begin{align*}
\begin{split}
(D_V^2+\frac{a}{2}D_n)\Theta^I=-\frac{\epsilon_K}{2}\Gamma_{KI}^J\Theta^JD_K\sigma^2-\frac{1}{2}D_ID_V\sigma^2.
\end{split}
\end{align*}
Using \eqref{eq:e0mu1} to express $D_V$ in terms of $\partial_t$, and introducing the notation
\begin{align}\label{eq:gamma1}
\begin{split}
\gamma:=\frac{a}{2\sigma^2}=\frac{\sqrt{\epsilon_I(D_I\sigma^2)^2}}{2\sigma^2},
\end{split}
\end{align}
we arrive at
\begin{align}\label{eq:Thetabdry1}
\begin{split}
(\partial_t^2+\gamma D_n)\Theta^I=-\frac{\epsilon_K}{2\sigma^2}\Gamma_{KI}^J\Theta^JD_K\sigma^2-\frac{1}{2\sigma^2}D_ID_V\sigma^2-\frac{1}{2\sigma^2}(\partial_t\sigma^2)\partial_t\Theta^I.
\end{split}
\end{align}
Equation \eqref{eq:Thetabdry1} is the analogue of the first equation in \eqref{eq:Minkowski1} in our setting. We next turn to the interior wave equation for $\Theta^I$, that is, the analogue of the second equation in \eqref{eq:Minkowski1}. Applying $\nabla_\nu$ to the vanishing divergence equation $\nabla_\mu V^\mu=0$, see \eqref{eq:relEuler3}, commuting $\nabla_\nu$ and $\nabla_\mu$, and using $dV=0$, gives
\begin{align}\label{eq:boxVintrotemp1}
\begin{split}
0=\nabla_\mu \nabla^\mu V_\nu-R_{\nu\lambda}V^\lambda.
\end{split}
\end{align}
Now we use \eqref{eq:Einstein1} to replace $R_{\nu\lambda}$ on the right. Taking trace of \eqref{eq:Einstein1} we get
\begin{align*}
\begin{split}
S= V_\mu V^\mu +2\quad \mathrm{in~}\Omega,\qquad S=0\quad \mathrm{in~}\Omega^c,
\end{split}
\end{align*}
and therefore in view of \eqref{eq:Einstein1} and \eqref{eq:hpemtensor1}
\begin{align*}R_{\nu\lambda}=
\begin{cases}
 V_\nu V_\lambda+\frac{1}{2}g_{\nu\lambda} \quad&\mathrm{in~}\Omega\\
 0\quad&\mathrm{in~}\Omega^c
\end{cases}.
\end{align*}
Plugging back into \eqref{eq:boxVintrotemp1} gives
\begin{align}\label{eq:boxVintrotemp2}
\nabla_\mu \nabla^\mu V^\nu=
\begin{cases}
(\frac{1}{2}-\sigma^2)V^\nu\qquad&\mathrm{in~}\Omega\\
0\qquad&\mathrm{in~}\Omega^c
\end{cases}.
\end{align}
Taking the inner product of \eqref{eq:boxVintrotemp2} with $e_I$ and using $\Box$ to denote the wave operator of $g$ we get
\begin{align*}
\begin{split}
\epsilon_I\big(\frac{1}{2}-\sigma^2\big)\Theta^I&=e_I^\mu \nabla_\nu \nabla^\nu V_\mu = \nabla_\nu (e_I^\mu \nabla^\nu V_\mu) -(\nabla^\nu V_\mu)(\nabla_\nu e_I^\mu)\\
&=\epsilon_I\Box \Theta^I-2(\nabla^\nu V_\mu)(\nabla_\nu e_I^\mu)-V_\mu \nabla_\nu\nabla^\nu e_I^\mu.
\end{split}
\end{align*}
The last two terms can be written in terms of the frame and connection coefficients as
\begin{align*}
\begin{split}
&(\nabla^\nu V_\mu)(\nabla_\nu e_I^\mu)=\epsilon_J\epsilon_Kg(\nabla_{e_K}(\Theta^M e_M),e_J)g(\nabla_{e_K}e_I,e_J)=\epsilon_J\epsilon_K\Gamma_{KI}^JD_K\Theta^J+\epsilon_J\epsilon_K\Gamma_{KL}^J\Gamma_{KI}^J\Theta^L,\\
&V_\mu \nabla_\nu\nabla^\nu e_I^\mu=\epsilon_J\epsilon_K\Theta^Jg(e_J,\nabla_{e_K}\nabla_{e_K}e_I)=\epsilon_K\Theta^J D_K\Gamma_{KI}^J-\epsilon_{L}\epsilon_K\Theta^J\Gamma_{KI}^L\Gamma_{KJ}^L.
\end{split}
\end{align*}
Combining these identities we get
\begin{align}\label{eq:BoxTheta1}
\begin{split}
\Box \Theta^I=(\frac{1}{2}-\sigma^2)\Theta^I+\epsilon_I\epsilon_J\epsilon_K\big(2\Gamma_{KI}^JD_K\Theta^J+2\Gamma_{KL}^J\Gamma_{KI}^J\Theta^L+\epsilon_J\Theta^J D_K\Gamma_{KI}^J-\epsilon_J\epsilon_{L}\Theta^J\Gamma_{KI}^L\Gamma_{KJ}^L).
\end{split}
\end{align}

To complete the set of fluid equations we need to derive wave equations for $\sigma^2$ and $D_V\sigma^2$ with Dirichlet boundary conditions (where as in \cite{MSW1} for $\sigma^2$ this will be treated as an elliptic equation with $\partial D_V\sigma^2$ as the source term). A direct computation using $\nabla_\mu V^\mu=0$ and $\nabla_\mu V_\nu= \nabla_\nu V_\mu$ shows that
\begin{align*}
\begin{split}
\Box \sigma^2 = \sigma^2-2\sigma^4-2(\nabla_\mu V_\nu)(\nabla^\mu V^\nu),\qquad \sigma^2\vert_{\partial\Omega}\equiv1.
\end{split}
\end{align*}
In the frame, this takes the form
\begin{align}\label{eq:Sigma1}
\begin{split}
\Box \sigma^2 = \sigma^2-2\sigma^4-2\epsilon_I\epsilon_J(D_I\Theta^J+\Gamma_{IK}^J\Theta^K)^2,\qquad \sigma^2\vert_{\partial\Omega}\equiv1.
\end{split}
\end{align}
For $D_V\sigma^2$, commuting the original equation for $\sigma^2$ with $D_V$ and using the fact that $\nabla_VV=-\frac{1}{2}\nabla\sigma^2$ we get
\begin{align*}
\begin{split}
\Box D_V\sigma^2=2D_V\sigma^2-6\sigma^2 D_V\sigma^2+6(\nabla^\mu V^\nu)(\nabla_\mu \nabla_\nu\sigma^2)+4(\nabla^\mu V^\nu)(\nabla_\mu V^\lambda)(\nabla_\lambda V_\nu)-4 R_{\lambda\mu\nu\kappa}(\nabla^\mu V^\nu)V^\kappa V^\lambda,
\end{split}
\end{align*}
with $D_V\sigma^2\vert_{\partial\Omega}\equiv0$. 
Recalling that $\partial_t\sigma^2=\frac{1}{\sqrt{\sigma^2}}D_V\sigma^2$, in terms of the frame this becomes
\begin{align}\label{eq:Lambda1}
\begin{split}
\Box \partial_t\sigma^2&=\frac{2}{\sqrt{\sigma^2}}D_{V}\sigma^2-\frac{6}{\sqrt{\sigma^2}}\sigma^2 D_V\sigma^2+\frac{6}{\sqrt{\sigma^2}}\epsilon_I\epsilon_J(D_I\Theta^J+\Gamma_{IK}^J\Theta^K)(D_ID_J\sigma^2-\Gamma_{IJ}^K D_K\sigma^2)\\
&\quad+\frac{4}{\sqrt{\sigma^2}}\epsilon_I\epsilon_J(D_I\Theta^J+\Gamma_{IM}^{J}\Theta^M)(D_I\Theta^K+\Gamma_{IN}^K\Theta^N)(D_K\Theta^J+\Gamma_{KP}^J\Theta^P)\\
&\quad-\frac{4}{\sqrt{\sigma^2}} \epsilon_I\epsilon_K\epsilon_LR_{L I J K}(D_I\Theta^J+\Gamma_{IM}^J\Theta^M)\Theta^{K}\Theta^{L}+\frac{\epsilon_I}{\sigma^2} D_I\sigma^2 D_I\partial_t\sigma^2\\
&\quad+\frac{\epsilon_I\partial_t\sigma^2}{4\sigma^4}D_I\sigma^2 D_I\sigma^2+\frac{\partial_t\sigma^2}{2}-\sigma^2\partial_t\sigma^2-\frac{\partial_t\sigma^2}{\sigma^2}\epsilon_I\epsilon_J(D_I\Theta^J+\Gamma_{IK}^J\Theta^K)^2).
\end{split}
\end{align}
Equations \eqref{eq:Thetabdry1}, \eqref{eq:BoxTheta1}, \eqref{eq:Sigma1}, \eqref{eq:Lambda1} are the fluid equations in the frame.

It remains to derive the equations satisfied by the curvature components. Our goal is to use the differential Bianchi equations, viewed as a first order hyperbolic system. However, to do this we need to take the discontinuity at the boundary into account. For this, observe that if $X,Y,Z,W$ are arbitrary vectors with $Z$ and $W$ tangential to the boundary, then $R(X,Y,Z,W)$ contains at most one transversal (relative to the fluid boundary) derivative of the metric. This follows directly from the coordinate expression of the curvature
\begin{align*}
\begin{split}
R_{\kappa\lambda\mu\nu}=g_{\lambda\beta}(\partial_\nu\Gamma^\beta_{\mu\kappa}-\partial_\mu\Gamma_{\nu\kappa}^\beta+\Gamma_{\mu\kappa}^\alpha\Gamma_{\nu\alpha}^\beta-\Gamma_{\nu\kappa}^\alpha\Gamma_{\mu\alpha}^\beta).
\end{split}
\end{align*}
Therefore, if $X$ and $Y$ are tangential on $\partial\Omega$, we can view the contracted Bianchi identity 
\begin{align}\label{eq:cBian1}
\begin{split}
\sum_I \epsilon_I \nabla_{I}R_{IJXY}&=\nabla_XR_{YJ}-\nabla_YR_{XJ},
\end{split}
\end{align}
as a weak equation on 
\begin{align}\label{eq:Sigmaleaf1}
\begin{split}
\Sigma:=\Omega\cup\Omega^c.
\end{split}
\end{align} 
Note that the right-hand side of \eqref{eq:cBian1} contains only tangential derivatives of the Ricci curvature.  
To derive a hyperbolic system which can be used to estimate the curvature, we combine \eqref{eq:cBian1} with the differential Bianchi equation
\begin{align}\label{eq:dBian1}
\begin{split}
\nabla_{[I}R_{JK]XY}=0,
\end{split}
\end{align}
which can again be viewed as a weak equation when $X$ and $Y$ are tangential to the fluid boundary. We now describe these points in more detail. 
Since $\{e_I:I=0,1,2,3\}$ is a frame and the fluid boundary is timelike, at any point the projections of $e_0$ and two of $\{e_1,e_2,e_3\}$ to the tangent space of the boundary span this tangent space. At a given point, without loss of generality, assume that the projections of $\{e_0,e_1,e_2\}$ span the tangent space to the boundary.  Let
\begin{align*}
\begin{split}
\tilX_A:=\epsilon_A\frac{e_A-g(e_A,n)n}{\|e_A-g(e_A,n)n\|},\qquad A=0,1,2,
\end{split}
\end{align*}
where we recall that $n=\frac{\nabla\sigma^2}{\|\nabla\sigma^2\|}$. The frame components of $\tilX_A$ are denoted by $\tilX_A^I$, that is,
\begin{align}\label{eq:XAI1}
\begin{split}
\tilX_A= \tilX_A^I e_I,\quad \tilX_A^I:=\epsilon_Ig(\tilX_A,e_I)=\epsilon_I\frac{\delta_A^I-(\sum_K\epsilon_K(D_K\sigma^2)^{2})^{-1}\epsilon_{A}D_A\sigma^2 D_I\sigma^2}{\sqrt{1-(\sum_J \epsilon_J (D_J\sigma^2)^2)^{-1}(D_A\sigma^2)^2}}.
\end{split}
\end{align}
Let $X_A$ be obtained from $\tilX_A$ by Gram-Schmidt, that is
\begin{align*}
\begin{split}
X_0= \tilX_0,\quad X_1=\frac{\tilX_1+g(\tilX_1,X_0)X_0}{\|\tilX_1+g(\tilX_1,X_0)X_0\|}, \quad X_2=\frac{\tilX_2+g(\tilX_2,X_0)X_0-g(\tilX_2,X_1)X_1}{\|\tilX_2+g(\tilX_2,X_0)X_0-g(\tilX_2,X_1)X_1\|}.
\end{split}
\end{align*}
The frame components of $X_A$ are denoted by $X_A^I$. Then, for each pair $A\neq B$ we define the two form $F^{AB}$ whose components in the frame $\{e_I\}$ are
\begin{align*}
\begin{split}
F^{AB}_{IJ}=R(e_I,e_J,X_A,X_B).
\end{split}
\end{align*}
Often the particular choice of $A,B$ is not important and we simply write $F_{IJ}$ for $F^{AB}_{IJ}$. As noted above, \eqref{eq:Einstein2} and the symmetries of the curvature tensor, allow us to recover all the components of the curvature tensor algebraically from $F^{AB}_{IJ}$, $A,B=0,1,2$, $I,J=0,1,2,3$. Indeed, consider $R(Y,Z,n,X_{A})$, $A=0,1,2$ where $Y,Z$ are arbitrary vectorfields. If $Y=X_{B},Z=X_{C}$, then
\begin{align}\label{curvature recover 1}
	R(X_{B},X_{C},n,X_{A})=R(n,X_{A},X_{B},X_{C}).
\end{align}
If $Y=n,Z=X_{B}$, then
\begin{align}\label{curvature recover 2}
	R(n,X_{B},n,X_{A})=-\sum_{C}\epsilon_{C}R(X_{C},X_{B},X_{C},X_{A})+X_B^IX_A^JR_{IJ},
\end{align}
where $R_{IJ}$ is as in \eqref{eq:Einstein2}, and if $Y=X_{B},Z=n$, then
\begin{align}\label{curvature recover 3}
	R(X_{B},n,n,X_{A})=-R(n,X_{B},n,X_{A}).
\end{align}
Now for general $Y,Z$, $R(Y,Z,n,X_A)$ can be written as a linear combination of \eqref{curvature recover 1}, \eqref{curvature recover 2}, \eqref{curvature recover 3}. To derive the analogue of \eqref{eq:cBian1} and \eqref{eq:dBian1} for $F$, we divide the six components $F_{IJ}\equiv F^{AB}_{IJ}$ of $F$  into the electric parts $E_\tilI\equiv E^{AB}_\tilI$ and the magnetic parts $H^\tilI\equiv H_{AB}^\tilI$, $\tilI=1,2,3$, defined by 
\begin{align*}
\begin{split}
E_\tilI=F_{\tilI 0},\quad H^1=-F_{23},\quad H^2=-F_{31},\quad H^3=-F_{12}.
\end{split}
\end{align*}
The Bianchi equations \eqref{eq:cBian1} and \eqref{eq:dBian1} can be used to derive a Maxwell system for $(E,H)$ in the usual way. Indeed, contracting $\nabla^\mu R_{\mu\nu\lambda\kappa}=\nabla_\lambda R_{\kappa\nu}-\nabla_\kappa R_{\lambda \nu}$ with $X_A^\lambda$ and $X^\kappa_B$ we get
\begin{align}\label{Maxwell-inhomo1}
\begin{split}
\nabla^\mu F^{AB}_{\mu\nu}&= X_A^\lambda X_B^\kappa(\nabla_\lambda R_{\kappa\nu}-\nabla_\kappa R_{\lambda \nu})+R_{\mu\nu\lambda\kappa}(X_B^\kappa\nabla^\mu X_A^\lambda+X_A^\lambda \nabla^\mu X_B^\kappa)\\
&=\begin{cases}g(X_B,V)\nabla_{X_A}V_\nu-g(X_A,V)\nabla_{X_B}V_\nu+R_{\mu\nu\lambda\kappa}(X_B^\kappa\nabla^\mu X_A^\lambda+X_A^\lambda \nabla^\mu X_B^\kappa),\quad\mathrm{in~}\Omega\\ R_{\mu\nu\lambda\kappa}(X_B^\kappa\nabla^\mu X_A^\lambda+X_A^\lambda \nabla^\mu X_B^\kappa),\phantom{g(X_A,V)\nabla_{X_B}V_\nu-g(X_B,V)\nabla_{X_A}V_\nu+}\quad\mathrm{in~~}\Omega^c\end{cases}.
\end{split}
\end{align}
In terms of the frame components we get
\begin{align}\label{eq:Maxwell1}
\begin{split}
\epsilon_I D_I F^{AB}_{IK}&=\calI_{K}^{AB}
\end{split}
\end{align}
where
\begin{align}\label{eq:calIint1}
\begin{split}
\calI^{AB}_K&=\epsilon_K\epsilon_I (X_B^IX_A^J-X_A^IX_B^J)\Theta^I(D_J\Theta^K+\Theta^L\Gamma_{JL}^K)+\epsilon_I\Gamma_{II}^JF^{AB}_{JK}+\epsilon_I\Gamma_{IK}^JF^{AB}_{IJ}\\
&\quad+\epsilon_J R_{JKLI}\big(X_A^ID_JX_B^L+X_B^ID_JX_A^L+\Gamma_{JM}^L(X_B^MX_A^I+X_A^MX_B^I)\big),\qquad \mathrm{in~}\Omega,
\end{split}
\end{align}
and
\begin{align}\label{eq:calIext1}
\begin{split}
\calI^{AB}_K&=\epsilon_I\Gamma_{II}^JF^{AB}_{JK}+\epsilon_I\Gamma_{IK}^JF^{AB}_{IJ}\\\
&\quad+\epsilon_J R_{JKLI}\big(X_A^ID_JX_B^L+X_B^ID_JX_A^L+\Gamma_{JM}^L(X_B^MX_A^I+X_A^MX_B^I)\big),\qquad \mathrm{in~}\Omega^c.
\end{split}
\end{align}
Similarly, contracting the identity $\nabla_{[\mu}R_{\nu\lambda]\kappa\rho}=0$ with $X_A^\kappa$ and $X_B^\rho$ gives
\begin{align}\label{Maxwell-inhomo2}
\begin{split}
\nabla_{[\mu}F_{\nu\lambda]}^{AB}=X_B^\rho(\nabla_{[\mu}X_A^\kappa )R_{\nu\lambda]\kappa\rho}+X_A^\kappa(\nabla_{[\mu}X_B^\rho) R_{\nu\lambda]\kappa\rho}.
\end{split}
\end{align}
In terms of the frame components this becomes
\begin{align}\label{eq:Maxwell2}
\begin{split}
 D_{[I}F^{AB}_{JK]}=\calJ_{IJK}^{AB}
 \end{split}
\end{align}
where 
\begin{align}\label{eq:calJ1}
\begin{split}
\calJ_{IJK}^{AB}&= 2\Gamma_{[IK}^MF^{AB}_{J]M}+X_B^LR_{ML[JK}D_{I]}X_A^M+X_B^LR_{ML[JK}\Gamma_{I]N}^MX_A^N\\
 &\quad+X_A^MR_{ML[JK}D_{I]}X_B^L+X_A^MR_{ML[JK}\Gamma_{I]N}^LX_B^N.
\end{split}
\end{align}
Equations \eqref{eq:Maxwell1} and \eqref{eq:Maxwell2} can be written in terms of $E$ and $H$ as 
\begin{align}\label{eq:EH1}
\begin{split}
D_0E+\curl H= \calI,\qquad D_0H-\curl E=\calJ^\ast,
\end{split}
\end{align}
where $(\calJ^\ast)^I=\epsilon^{IJKL}\calJ_{JKL}$ and $\curl$ denotes the usual curl operator in dimension three with respect to $D_\tilI$, $\tilI=1,2,3$:
\begin{align*}
\begin{split}
\curl E = \pmat{D_2E_3-D_3E_2\\D_3E_1-D_1E_3\\D_1E_2-D_2E_1},
\end{split}
\end{align*}
and similarly for $\curl H$. Defining $W\equiv W^{AB}=(W_1,\dots,W_6)$ where $W_m=E_m$ for $m=1,2,3$ and $W_m=H^{m-3}$ for $m=4,5,6$, and with
\begin{align}\label{eq:Amatdef1}
\begin{split}
\calA^1=\tiny{\pmat{0&0&0&0&0&0\\0&0&0&0&0&-1\\0&0&0&0&1&0\\0&0&0&0&0&0\\0&0&1&0&0&0\\0&-1&0&0&0&0}},\quad \calA^2=\pmat{0&0&0&0&0&1\\0&0&0&0&0&0\\0&0&0&-1&0&0\\0&0&-1&0&0&0\\0&0&0&0&0&0\\1&0&0&0&0&0},\quad \calA^3=\pmat{0&0&0&0&-1&0\\0&0&0&1&0&0\\0&0&0&0&0&0\\0&1&0&0&0&0\\-1&0&0&0&0&0\\0&0&0&0&0&0},
\end{split}
\end{align}
equation~\eqref{eq:EH1} becomes a first order symmetric hyperbolic system for $W$ of the form 
\begin{align}\label{eq:W1}
\begin{split}
D_0 W+\sum_{\tilI=1}^3\calA^\tilI D_\tilI W= \calK, \qquad \calK=(\calI_1,\calI_2,\calI_3,\calJ^\ast_1,\calJ^\ast_2,\calJ^\ast_3).
\end{split}
\end{align}
Since $e_0$ does not coincide with $\partial_t$, and is not adapted to the foliation by constant $t$ slices, we further decompose this as
\begin{align}\label{eq:W2}
\begin{split}
\sum_{\mu=0}^3 \calB^\mu \partial_\mu W=\calK, \quad \calB^0:=e_0^0+\sum_{\tilI=1}^3 e_\tilI^0\calA^\tilI, \quad \calB^j:=e_0^j+\sum_{\tilI=1}^3e_\tilI^j \calA^\tilI,~j=1,2,3.
\end{split}
\end{align}
This is the weakly defined first order symmetric system for $W$ which can be used to prove a priori estimates and set up an iteration.  Note that the positivity of the coefficient for $\partial_t$ follows from the fact that $e_\tilI^0$ vanish on the initial slice, and are therefore small by continuity during the local in time evolution.

To summarize, equations \eqref{eq:e0mu1}, \eqref{eq:etransport1}, \eqref{eq:Gamma01}, \eqref{eq:Gammatransport1}, \eqref{eq:g1}, \eqref{eq:BoxTheta1}, \eqref{eq:Thetabdry1}, \eqref{eq:Sigma1}, \eqref{eq:Lambda1}, \eqref{curvature recover 1}, \eqref{curvature recover 2}, \eqref{curvature recover 3}, \eqref{eq:Maxwell1}, \eqref{eq:Maxwell2}, 
and \eqref{eq:W2} are the equations we use to prove a priori estimates and set up an iteration. To record the structure of these equations, we collect them in compact form here again, where the precise expressions for the source terms $\calF$ on the right-hand side are as in  \eqref{eq:e0mu1}, \eqref{eq:etransport1}, \eqref{eq:Gamma01}, \eqref{eq:Gammatransport1}, \eqref{eq:g1}, \eqref{eq:BoxTheta1}, \eqref{eq:Thetabdry1}, \eqref{eq:Sigma1}, \eqref{eq:Lambda1}, \eqref{curvature recover 1}, \eqref{curvature recover 2}, \eqref{curvature recover 3},
\eqref{eq:Maxwell1}, \eqref{eq:Maxwell2}, \eqref{eq:W2}:
\begin{align}\label{eq:collected1}
\begin{cases}
(\partial_t^2+\gamma D_n)\Theta^I = \calF_{\Theta^I,\partial\Omega}\qquad &\mathrm{on~}\partial\Omega\\
\Box \Theta^I = \calF_{\Theta^I,\Omega}\qquad &\mathrm{in~}\Omega\\
\Box \sigma^2=\calF_{\sigma^2}\qquad &\mathrm{in~}\Omega\\
\Box D_V\sigma^2 = \calF_{D_V\sigma^2}\qquad &\mathrm{in~}\Omega\\
\sigma^2\equiv 1,\quad D_V\sigma^2\equiv 0\qquad &\mathrm{on~}\partial\Omega\\
\partial_te_\tilI=\calF_{e_\tilI},\quad e_0=\frac{1}{\hatTheta^0}(\hatV-\hatTheta^{\tilI}e_\tilI)\qquad &\mathrm{in~}\Sigma\\
\partial_t\Gamma_{\tilI J}^K=\calF_{\Gamma_{\tilI J}^K},\quad \Gamma_{0J}^K=-\frac{\hatTheta^\tilI}{\hatTheta^0}\Gamma_{\tilI J}^K\qquad&\mathrm{in~}\Sigma\\
\sum_{\mu=0}^3 \calB^\mu \partial_\mu W^{AB}=\calF_{W^{AB}}\qquad&\mathrm{in~}\Sigma
\end{cases}.
\end{align}
Here recall from \eqref{eq:Sigmaleaf1} that $\Sigma=\Omega\cup\Omega^c$. Below we will also use the notation
\begin{align*}
\begin{split}
\Sigma_t=\Omega_t\cup\Omega_t^c= \Sigma\cap\{x^0=t\}.
\end{split}
\end{align*}
Note that in arriving at equation \eqref{eq:collected1} we have already assumed that we have a solution  and that $R$ is the Riemann curvature tensor of the metric $g$. Therefore, to recover the original equations \eqref{eq:Einstein2} and \eqref{eq:relEuler4} from \eqref{eq:collected1}, we naturally need the initial data to satisfy the equations we have used in the derivation. As stated in Theorem~\ref{thm:main1}, this is captured by requiring that the following vanishing requirements hold initially (corresponding to both the original fluid equations and various geometric identities such as the Ricci identity and the symmetries of the curvature tensor):
\begin{align}
&\nabla_{I}\Theta^{I}\vert_{t=0},\quad (\nabla_{I}\Theta_{J}-\nabla_{J}\Theta_{I})\vert_{t=0},\quad (\sigma^{2}+\Theta^{I}\Theta_{I})\vert_{t=0}, \quad(\Theta^{J}\nabla_{J}\Theta_{I}+\frac12\nabla_{I}\sigma^{2})\vert_{t=0},\label{def vanishing fluid int}\\
 &(\Theta^{J}\nabla_{J}\left(\Theta^{K}\nabla_{K}\Theta_{I}\right)-\frac12\left(\nabla^{J}\sigma^{2}\right)\nabla_{J}\Theta_{I}+\frac12\nabla_{I}(\Theta^{J}\nabla_{J}\sigma^{2}))\vert_{t=0},\quad (\Box\sigma^{2}+2(\nabla^{I}\Theta^{J})(\nabla_{I}\Theta_{J})+(2\sigma^{2}-1)\sigma^{2})\vert_{t=0}=0,\nonumber
\end{align}
for the fluid and 
\begin{align}\label{def vanishing geom int}
\begin{split}
&(\Gamma_{JI}^{K}e_K-\Gamma_{IJ}^{K}e_K+[e_{I},e_{J}],\quad R_{[IJK]L})\vert_{t=0},\quad (R_{IJKL}-R_{KLIJ})\vert_{t=0}, \nabla_{[I}R_{JK]LM}\vert_{t=0}\\
&(\nabla^{J}R_{JIKL}-\chi_{\Omega}\left(\Theta_{L}\nabla_{K}\Theta_{I}-\Theta_{K}\nabla_{L}\Theta_{I}\right))\vert_{t=0},\quad (R_{IJ}-\chi_{\Omega}(\Theta_{I}\Theta_{J}+\frac12g_{IJ}))\vert_{t=0},\\
&(R^{K}{}_{MIJ}-\left(D_{I}\Gamma_{JM}^{K}-D_{J}\Gamma_{IM}^{K}\right)-\left(\Gamma_{IL}^{K}\Gamma_{JM}^{L}-\Gamma_{JL}^{K}\Gamma_{IM}^{L}\right)-\left(\Gamma_{JI}^{L}-\Gamma_{IJ}^{L}\right)\Gamma_{LM}^{K})\vert_{t=0}=0,
\end{split}
\end{align} 
for the geometric quantities.

We end this section by briefly describing some of the main ideas used to study \eqref{eq:collected1}. The first two equations for the fluid velocity contain the hyperbolic Dirichlet-Neumann map. They are used to get control of
\begin{align*}
\begin{split}
\sup_t\|\partial_{t,x}\Theta\|_{L^2(\Omega_t)}+ \sup_t\|\partial_t\Theta\|_{L^2(\partial\Omega_t)}.
\end{split}
\end{align*}
This is achieved by multiplying both equations by $\partial_t\Theta$ and observing the crucial cancellations of the boundary terms with unfavorable signs. The wave equation for $D_V\sigma^2$ with Dirichlet boundary conditions is used to get control of
\begin{align*}
\begin{split}
\sup_t\|\partial_{t,x}D_V\sigma^2\|_{L^2(\Omega_t)}+\|\partial_{t,x}D_V\sigma^2\|_{L^2_tL^2(\Omega_t)}.
\end{split}
\end{align*}
This is achieved by the choice of a suitable multiplier consisting of an appropriate linear combination of $\partial_t$ and the normal $n$.  This choice is necessary as $\|\partial_{t,x}D_V\sigma^2\|_{L^2_tL^2(\Omega_t)}$ comes up as a source term in the energy estimate for $\Theta$. In treating the nonlinearity in the equation for $D_V\sigma^2$ we use the wave equation for $\sigma^2$ with Dirichlet conditions as an elliptic equations to estimate two derivatives of $\sigma^2$ in terms of one derivative of $D_V\sigma^2$. The second and third equations from the bottom in \eqref{eq:collected1} are used as transport and algebraic equations for the frame and connection coefficients in terms of the other unknowns. The last equation, which is a symmetric first order hyperbolic system, and the algebraic relations \eqref{curvature recover 1}, \eqref{curvature recover 2}, \eqref{curvature recover 3} are used to estimate $\|R_{IJKL}\|_{L^2}$. The curvature itself appears both on the right-hand side of the transport equation for $\Gamma$ and in all other equations after commuting derivatives. 

To control higher order derivatives of the unknowns, we commute $\partial_t^k$ with \eqref{eq:collected1}. This is important because $\partial_t$ is tangential to the boundary and commuting $\partial_t$ preserves the main structure of the left-hand side of \eqref{eq:collected1}. We then need to use elliptic estimates to turn control of $\partial_t^k$ of the unknowns into control of arbitrary derivatives. For the fluid velocity $\Theta$, this is done by using the boundary equation to express the normal derivative $D_n\Theta$ in terms of $\partial_t^2\Theta$, and then using the interior equation as an elliptic equation with Neumann boundary data. Consequently $L^2$ control of an arbitrary spatial derivative requires $L^2$ control of two time derivatives. Since $\Theta$ appears as a source term in all other equations, this relation between space and time derivatives holds for the other unknowns as well. For $D_V\sigma^2$, since $D_V\sigma^2\equiv0$ on $\partial\Omega$, elliptic estimates are obtained by viewing the wave equation as an elliptic equation with Dirichlet boundary conditions, and with $\partial D_V\sigma^2$ as source term. For the curvature, we return to the Maxwell system \eqref{eq:EH1} satisfied by the electric and magnetic decomposition. Combined with the divergence equations for $E$ and $H$, as well as the algebraic relations \eqref{curvature recover 1},  \eqref{curvature recover 2}, \eqref{curvature recover 3}, this yields a curl-divergence system which allows us to estimate spatial derivatives of $R$. The proof of a priori estimates described in this and the previous paragraphs is contained in Section~\ref{sec:apriori}. 

The iteration, which is the content of Section~\ref{sec:iteration}, is also based on equations \eqref{eq:collected1}. Since the curvature equation satisfies a first order symmetric hyperbolic system (in the weak sense), the linear existence theory for the last equation in \eqref{eq:collected1} is known. However, the linear existence theory for the fluid variables is not standard due to the unusual boundary equation given in the first equation of \eqref{eq:collected1}. For this we borrow the existence theory which was developed in our earlier work \cite{MSW1} using Galerkian approximations based on a weak formulation of the fluid equations in \eqref{eq:collected1}. The convergence of the iteration follows from the similar arguments to those used to prove a priori estimates. A main technical difficulty, however, is that since the relations \eqref{curvature recover 1},  \eqref{curvature recover 2}, \eqref{curvature recover 3} are used to define arbitrary curvature components algebraically in terms of those with tangential components, and the metric is defined by the requirement that $\{e_I\}$ be an orthonormal frame, it is not immediate that $R$ is the curvature tensor of the metric and that equation \eqref{eq:Einstein2} is satisfied. This requires a careful analysis of the choices made in setting up the iteration and deriving homogeneous equations for what we refer to as \emph{vanishing quantities}, that is quantities that vanish if $R$ is the curvature tensor of the constructed metric and \eqref{eq:Einstein2} holds. Similarly, since the fluid equations are at the level of a derivative of the original equations \eqref{eq:relEuler4}, we need to derive homogeneous equations for the vanishing quantities (such as the divergence and vorticity of the velocity) corresponding to the fluid to show that \eqref{eq:relEuler4} is satisfied. We show that the fluid and geometric vanishing quantities satisfy a coupled system with similar structure to \eqref{eq:collected1}, and therefore vanish due to the vanishing on the initial slice and in view of the a priori estimates established in Section~\ref{sec:apriori}.
\subsection{The initial data}\label{subseq:data}
Here we describe how to construct initial data for the system \eqref{eq:collected1} based on the initial data for \eqref{eq:Einstein1}. Recall that the initial data for \eqref{eq:Einstein1}, with the choice \eqref{eq:Thp1}, consist of the following: A symmetric positive definite two form $\barg$ (the metric), a symmetric two form $k$ (the second fundamental form), and two scalar functions $\phi_0$ and $\phi_1$ (initial data for $\phi$). They are assumed to satisfy the constraint equations\footnote{We refer the reader to, for instance, the monograph \cite{CB} by Choquet-Bruhat for discussion on the constraint equations and methods of constructions of initial data for various matter models.}
\begin{align*}
	\begin{split}
		&\barR-k_{ij}\,k^{ij}+(\tr_{\barg}k)^{2}=\chi_{\Omega_{0}}\big((\phi_1)^2+(\barg^{-1})^{ij}\partial_{i}\phi_0\partial_{j}\phi_0+1\big),\\
		&\barnabla^{j}k_{ji}-\barnabla_{i}\big(\tr_{\barg}k\big)=\chi_{\Omega_{0}}\,\phi_1\,\partial_{i}\phi_0.
	\end{split}
\end{align*}
Here quantities with an overline, such as $\barR$ and $\barnabla$, are associated with the metric $\barg$ on the initial slice $\Sigma_0$ (which is given and diffeomorphic to $\bbR^3$), and roman indices take values in $\{1,2,3\}$ and correspond to a fixed coordinate choice on the initial slice. Indices are raised and lowered with respect to $\barg$ and $\barg^{-1}$. We also extend $\phi_0$ and $\phi_1$ to the entire slice $\Sigma_0=\{t=0\}$, using a Sobolev extension, such that $\phi_0\equiv0$ and $\phi_1\equiv 1$ outside a compact set. Once the metric is given we extend $V$ to the exterior by $V=\nabla\phi$. Now recall that our convention in Lagrangian coordinates is that $\partial_t=\hatV= \frac{V}{\|V\|}$. This leads to
\begin{align}\label{g00 data}
	g_{00}|_{t=0}=-1, \qquad \partial_tg_{00}\vert_{t=0}=0,
\end{align}
and 
\begin{align*}
	g_{0i}\vert_{t=0}=g(\partial_{t},\partial_{i})=\frac{1}{V^{0}}\partial_{i}\phi_0.
\end{align*}
To determine $V^0$, note that since $V=V^0\partial_t=\nabla\phi$, and $g_{00}=-1$, initially
\begin{align*}
	\begin{split}
		-(V^0)^2=g(V,V)=g(\nabla\phi,\nabla\phi)=-\phi_1^2+(\barg^{-1})^{ij}\partial_i\phi_0\partial_j\phi_0,
	\end{split}
\end{align*}
so
\begin{align*}
	\begin{split}
		V^0\vert_{t=0}=(\phi_1^2-(\barg^{-1})^{ij}\partial_i\phi_0\partial_j\phi_0)^{\frac{1}{2}},
	\end{split}
\end{align*}
and
\begin{align}\label{eq:g0idata}
	\begin{split}
		g_{0i}\vert_{t=0}=\frac{\partial_i\phi_0}{(\phi_1^2-(\barg^{-1})^{ij}\partial_i\phi_0\partial_j\phi_0)^{1/2}}.
	\end{split}
\end{align}
We also note that $V^0$ also determines $\partial_t\phi\vert_{t=0}$ by
\begin{align}\label{eq:dtphidata}
	\begin{split}
		\partial_{t}\phi|_{t=0}=V_{0}|_{t=0}=g_{00}V^{0}|_{t=0}=-(\phi_1^2-(\barg^{-1})^{ij}\partial_i\phi_0\partial_j\phi_0)^{\frac{1}{2}},
	\end{split}
\end{align}
and that knowledge of $\barg$ and the choices \eqref{g00 data} and \eqref{eq:g0idata} fully determine $g(0)=g\vert_{t=0}$ and hence $g^{-1}(0)=g^{-1}\vert_{t=0}$. To determine the initial values of $\partial_tg_{ij}$, let $T$ be the future-directed unit normal vectorfield to $\Sigma_{0}$, that is,
\begin{align}\label{def T}
	T=-\frac{(g^{-1})^{00}}{\sqrt{-(g^{-1})^{00}}}\partial_{t}-\frac{(g^{-1})^{0i}}{\sqrt{-(g^{-1})^{00}}}\partial_{i}.
\end{align}
By the definition of $k$ we have
\begin{align}\label{def k}
	2k_{ij}=\left(\calL_{T}g\right)_{ij}=T(g_{ij})-g([T,\partial_{i}],\partial_{j})-g(\partial_{i},[T,\partial_{j}])
\end{align}
Since 
\begin{align*}
	\left[T,\partial_{k}\right]=&\partial_{k}\Big(\frac{(g^{-1})^{00}}{\sqrt{-(g^{-1})^{00}}}\Big)\partial_{t}+\partial_{k}\Big(\frac{(g^{-1})^{0i}}{\sqrt{-(g^{-1})^{00}}}\Big)\partial_{i},
\end{align*}
and $(g^{-1})^{00}\slashed{=}0$, the initial data for $k_{ij}$ and the initial data for $g_{\alpha\beta}$ as well as their spatial derivatives give the initial data for $\partial_{t}g_{ij}$:
\begin{align}\label{eq:dtgijdata}
	\begin{split}
		\sqrt{-(g^{-1})^{00}}\partial_tg_{ij}=&2k_{ij}+\partial_i\left(\frac{(g^{-1})^{0\alpha}}{\sqrt{-(g^{-1})^{00}}}\right)g_{\alpha j}+\partial_j\left(\frac{(g^{-1})^{0\alpha}}{\sqrt{-(g^{-1})^{00}}}\right)g_{\alpha i}\\
		&+\frac{(g^{-1})^{0k}}{\sqrt{-(g^{-1})^{00}}}\partial_{k}g_{ij}\qquad \mathrm{at~} t=0.
	\end{split}
\end{align}
Next we consider the initial data of the Christoffel symbols. We start with the components $\Gamma_{ij}^{\mu}$. By definition,
\begin{align*}
	\Gamma_{ij}^{\mu}=(g^{-1})^{\mu k}g(\nabla_{\partial_{i}}\partial_{j},\partial_{k})+(g^{-1})^{\mu0}g(\nabla_{\partial_{i}}\partial_{j},\partial_{t}).
\end{align*}
The first term on the right above can be written as 
\begin{align*}
	(g^{-1})^{\mu k}g_{k\ell}\barGamma^{\ell}_{ij},
\end{align*}
whose initial data is determined. For the second term, we note that
\begin{align*}
	\partial_{t}=\frac{1}{\sqrt{-(g^{-1})^{00}}}T+\frac{(g^{-1})^{0i}}{(g^{-1})^{00}}\partial_{i},
\end{align*}
so that
\begin{align}\label{eq:datatemp1}
	g(\nabla_{\partial_{i}}\partial_{j},\partial_{t})
	=&-\frac{1}{\sqrt{-(g^{-1})^{00}}}k_{ij}+\frac{(g^{-1})^{0k}}{(g^{-1})^{00}}g_{k\ell}\barGamma_{ij}^{\ell}.
\end{align}
Therefore the initial data for $\Gamma_{ij}^{\mu}$ is
\begin{align}\label{eq:Gammaijmudata}
	\begin{split}
		\Gamma_{ij}^\mu=(g^{-1})^{\mu k}g_{k\ell}\barGamma^{\ell}_{ij}+(g^{-1})^{\mu0}\Big(-\frac{1}{\sqrt{-(g^{-1})^{00}}}k_{ij}+\frac{(g^{-1})^{0k}}{(g^{-1})^{00}}g_{k\ell}\barGamma_{ij}^{\ell}\Big)\qquad \mathrm{at~}t=0.
	\end{split}
\end{align}
For the components $\Gamma_{ij}^{0}$ we have
\begin{align*}
	\Gamma_{ij}^{0}=(g^{-1})^{00}g(\nabla_{\partial_{i}}\partial_{j},\partial_{t})+(g^{-1})^{0k}g(\nabla_{\partial_{i}}\partial_{j},\partial_{k}).
\end{align*}
For the components $\Gamma_{i0}^{\mu}$, since $g(\partial_{t},\partial_{t})=-1$,
\begin{align*}
	\Gamma_{i0}^{\mu}
	=(g^{-1})^{\mu j}g(\nabla_{\partial_{i}}\partial_{t},\partial_{j})=(g^{-1})^{\mu j}\partial_i g_{0j}-(g^{-1})^{\mu j}g(\nabla_{\partial_i}\partial_j, \partial_t).
\end{align*}
It follows from \eqref{eq:datatemp1} that
\begin{align}\label{eq:Gammai0mudata}
	\begin{split}
		\Gamma_{i0}^{\mu}=(g^{-1})^{\mu j}\partial_i g_{0j}-(g^{-1})^{\mu j}\Big(-\frac{1}{\sqrt{-(g^{-1})^{00}}}k_{ij}+\frac{(g^{-1})^{0k}}{(g^{-1})^{00}}g_{k\ell}\barGamma_{ij}^{\ell}\Big)\qquad \mathrm{at~}t=0.
	\end{split}
\end{align}
For the components $\Gamma_{00}^{\mu}$, arguing as above, using \eqref{g00 data} and $\nabla_{\partial_t}\partial_k=\nabla_{\partial_k}\partial_t$,
\begin{align}\label{eq:datatemp2}
	\Gamma_{00}^{\mu}=(g^{-1})^{\mu k}g(\nabla_{\partial_{t}}\partial_{t},\partial_{k})=(g^{-1})^{\mu k}\partial_tg_{0k}-(g^{-1})^{\mu k }g(\partial_t,\nabla_{\partial_k}\partial_t)=(g^{-1})^{\mu k}\partial_tg_{0k}.
\end{align}
Therefore, the initial data of the Christoffel symbols will be determined once we know the initial data for $\partial_t g_{0k}$. For this, recall that $g_{0k}=\frac{1}{V^0}g(V,\partial_k)=\frac{1}{V^0}\partial_k\phi$. It follows that
\begin{align*}
	\begin{split}
		\partial_tg_{0k}=\frac{\partial_{tk}^2\phi}{V^0}+\frac{\partial_k\phi}{(V^0)^2}\partial_tV^0.
	\end{split}
\end{align*}
The first term is already determined initially in view of \eqref{eq:dtphidata}. For the second term, since $V$ is divergence free (equivalently, since $\Box\phi=0$), by \eqref{eq:datatemp2}
\begin{align*}
	\begin{split}
		\partial_tV^0&= -\Gamma^0_{00}V^0-\Gamma^{i}_{i0}V^0=\Gamma_{i0}^iV^0-(g^{-1})^{0k}V^0\partial_tg_{0k}\\
		&=\Gamma_{i0}^iV^0-(g^{-1})^{0k}\partial_{tk}^2\phi+(g^{-1})^{0k}g_{0k}\partial_tV^0=\Gamma_{i0}^iV^0-(g^{-1})^{0k}\partial_{tk}^2\phi+(1-g_{00}(g^{-1})^{00})\partial_tV^0.
	\end{split}
\end{align*}
Since $g_{00}=-1$ we get $\partial_tV^0=\frac{1}{(g^{-1})^{00}}((g^{-1})^{0k}\partial^2_{tk}\phi-\Gamma_{j0}^0V^0)$ and hence
\begin{align}\label{eq:dtg0kdata}
	\begin{split}
		\partial_tg_{0k}=\frac{\partial_{tk}^2\phi}{V^0}+\frac{\partial_k\phi}{(g^{-1})^{00}(V^0)^2}((g^{-1})^{0k}\partial^2_{tk}\phi-\Gamma_{j0}^jV^0)\qquad \mathrm{at~}t=0,
	\end{split}
\end{align}
and by \eqref{eq:datatemp2}
\begin{align}\label{eq:Gamma00mudata}
	\begin{split}
		\Gamma^\mu_{00}=(g^{-1})^{\mu k}\partial_t g_{0k}\qquad \mathrm{at~}t=0.
	\end{split}
\end{align}
Having determined the data for $g$, $\partial_tg$ and $\Gamma_{\alpha\beta}^\gamma$, the data of the curvature are also determined in the usual way. The data for $R_{ij\ell m}$ and $R_{\ell Tij}$ are given by the Gauss and Codazzi equations
\begin{align}\label{eq:GaussCodazzi}
	\begin{split}
		&R_{ij\ell m}\vert_{t=0}=\barR_{ij\ell m}-k_{ij}k_{\ell m}+k_{i\ell}k_{jm},\\
		&R_{\ell Tij}\vert_{t=0}=\barnabla_{i}k_{j\ell}-\barnabla_{j}k_{i\ell}.
	\end{split}
\end{align}
Similarly to determine $R_{iT\ell T}$ we note that
\begin{align*}
	R_{ij}=-R_{iTj T}+(\barg^{-1})^{\ell m}R_{i\ell j m}=-R_{iTj T}+\barR_{ij}-k_{i\ell}k_{j}^{\ell}+k_{ij}\tr_{\barg}k.
\end{align*}
Rearranging and using the Einstein equations we get
\begin{align}\label{eq:RiTjTdata}
	\begin{split}
		R_{iTj T}=k_{ij}\tr_{\barg}k-k_{i\ell}k_{j}^{\ell}-(\partial_i\phi\partial_j\phi+\frac{1}{2}g_{ij})\chi_{\Omega_0}+\barR_{ij}\qquad\mathrm{at~}t=0.
	\end{split}
\end{align}
Equations \eqref{eq:RiTjTdata}, \eqref{eq:GaussCodazzi}, \eqref{eq:Gamma00mudata}, \eqref{eq:dtg0kdata}, \eqref{eq:Gammai0mudata}, \eqref{eq:Gammaijmudata}, \eqref{eq:dtgijdata}, \eqref{eq:dtphidata}, \eqref{eq:g0idata}, \eqref{g00 data} now fully determine the initial data for $g_{\alpha\beta}$, $\partial_\gamma g_{\alpha\beta}$, $\Gamma_{\alpha\beta}^\gamma$, and $R_{\alpha\beta\gamma\delta}$. Finally, we remark that this information, together with a choice of initial frame $\{e_I\}$ (which can be chosen since $g$ is determined initially), determines the initial values of $\Gamma_{IJ}^K$ and $R_{IJKL}$ as well. For the curvature, this simply follows from the knowledge of $R_{\alpha\beta\gamma\delta}$. For the connection coefficients, first note that the condition $\nabla_{\partial_t}e_I=0$ implies
\begin{align}\label{eq:dtedata}
	\begin{split}
		\partial_te_I^\mu=-\Gamma_{0\nu}^\mu e_I^\nu\qquad \mathrm{at~}t=0.
	\end{split}
\end{align}
The initial values of $\Gamma_{IJ}^K$ are then determined by the relation
\begin{align}\label{eq:GammaIJKdata}
	\begin{split}
		\Gamma_{IJ}^K=\epsilon_K e_K^{\beta}e_I^\gamma e_J^\alpha g_{\beta\delta}\Gamma^\delta_{\gamma\alpha}+\epsilon_K e_I^\gamma e_K^\beta g_{\alpha\beta}\partial_\gamma e_J^\alpha\qquad \mathrm{at~}t=0.
	\end{split}
\end{align}
\section{A Priori Estimates}\label{sec:apriori}
In this section we prove a priori estimates for the derived system \eqref{eq:collected1}. Besides being of independent interest in controlling the evolution of the system, the a priori estimates provide the framework for setting up the iteration for the proof of Theorem~\ref{thm:main1}. Our starting point is the following definition of the energies: 
\begin{align*}
	\begin{split}
		E_k(T)&:=\sup_{0\leq t \leq T}\Big(\|\partial \partial_t^k\Theta\|_{L^2(\Omega_t)}^2+\|\partial_t^{k+1}\Theta\|_{L^2(\partial\Omega_t)}^2+\|\partial \partial_t^{k+1}\sigma^2\|_{L^2(\Omega_t)}^2+\|\partial_t^{k}R\|_{L^2(\Sigma_t)}^2\Big) +\|\partial \partial_t^{k+1}\sigma^2\|_{L^2(\partial\Omega_0^T)}^2,
	\end{split}
\end{align*}
and
\begin{align*}
	\begin{split}
		\calE_\ell(T)=\sum_{k\leq \ell}E_k(T).
	\end{split}
\end{align*}
The following proposition contains the main a priori estimate for this work. It is proved in Subsection~\ref{subsec:apriori}. 
\begin{proposition}\label{prop:apriori}
	Suppose $(\Theta,R,\partial_t\sigma^2,\sigma^2)$ is a solution to the system \eqref{eq:collected1}
		and $\calE_\ell(T)\leq C_1$ for some constant $C_1$ with $\ell$ sufficiently large. Let
	\begin{align*}
		\begin{split}
			\scE_\ell(T)&=\calE_\ell(T)+\sup_{0\leq t\leq T}\sum_{2p+k\leq \ell+2}\Big(\|\partial^p\partial_t^k\Theta\|_{L^2(\Omega_t)}^2+\|\partial^p\partial_t^{k+1}\sigma^2\|_{L^2(\Omega_t)}^2\Big)\\
			&\quad+\sup_{0\leq t \leq T}\sum_{2p+k\leq \ell+1}\Big(\|\partial^p\partial_t^{k}R\|_{L^2(\Omega_t)}^2+\|\partial^p\partial_t^{k}R\|_{L^2(\Omega_t^c)}^2\Big).
		\end{split}
	\end{align*}
	If $T>0$ is sufficiently small depending on $\scE_\ell(0)$, $C_1$, $c_0$ (recall that $c_{0}$ is defined in \eqref{eq:Taylor1}) and $\ell$, then for some polynomial function $\calP_\ell$ (independent of $C_1$),
	\begin{align*}
		\begin{split}
			\calE_\ell(T)\leq \calP_{\ell}(\scE_\ell(0)).
		\end{split}
	\end{align*}
\end{proposition}

\subsection{Commutator identities and higher order equations}\label{subsec:commutators}
We first compute commutator identities with the main linear operators, and use these to derive higher order equations.
\begin{lemma}\label{lem:commutators}
	For any scalar function $u$, 
	\begin{align*}
		\begin{split}
			[\partial_t,D_I]u=&- (D_I\hatTheta^J+\hatTheta^K\Gamma_{I K}^J)D_Ju.\\
			\left[\partial_{t},\Box\right]u=&-\sum_{I}\epsilon_{I}\left(D_{I}\hatTheta^{J}+\hatTheta^{K}\Gamma_{IK}^{J}\right)\left(D_{I}D_{J}u+D_{J}D_{I}u\right)-\Box\hatTheta^{J}D_{J}u\\
			&+\sum_{J}\epsilon_{J}\left(\Gamma_{IJ}^{K}-\Gamma_{JI}^{K}\right)D_{J}\hatTheta^{I}D_{K}u-\sum_{I}\epsilon_{I}\hatTheta^{K}D_{I}\Gamma_{IK}^{J}D_{J}u+\sum_{I}\epsilon_{I}\Gamma_{II}^{K}\Gamma_{KM}^{J}\hatTheta^{M}D_{J}u\\
			&-\hatTheta^{I}\textrm{Ric}^{K}{}_{I}D_{K}u-\sum_{J}\epsilon_{J}\hatTheta^{I}\Gamma^{K}_{MJ}\Gamma_{JI}^{M}D_{K}u+\sum_{\tilI}\hatTheta^{I}\Gamma_{M\tilI}\Gamma_{\tilI I}^{M}D_{K}u\\
			&-\sum_{J}\epsilon_{J}\hatTheta^{I}\left(\Gamma_{JM}^{K}\Gamma^{M}_{IJ}-\Gamma^{K}_{IM}\Gamma^{M}_{JJ}\right)D_{K}u\\
			[\partial_t,\partial_t^2+\gamma D_n]u=&\frac{\epsilon_I}{2\sigma^2}(D_I\partial_t\sigma^2)D_Iu-\frac{\partial_t\sigma^2}{\sigma^2}\gamma D_nu  -\frac{\epsilon_I}{2\sigma^2}(D_I\hatTheta^J+\Gamma_{IK}^J\hatTheta^K)(D_J\sigma^2)D_Iu\\
			&\phantom{[\partial_t,\partial_t^2+\gamma D_n]=}-\frac{\epsilon_I}{2\sigma^2}(D_I\hatTheta^J+\Gamma_{IK}^J\hatTheta^K)(D_I\sigma^2)D_Ju.
		\end{split}
	\end{align*}
\end{lemma}
\begin{proof}
	The commutator $[\partial_t,D_I]$ can be calculated using \eqref{eq:etransport1} and Remark \ref{rem:transport e0mu}.
	For the commutator with $\Box$, we use the representation
	\begin{align}\label{eq:Boxframe1}
		\begin{split}
			\Box u = \sum_{I}\epsilon_I D_ID_Iu-\sum_{I}\epsilon_I\Gamma^K_{II}D_Ku.
		\end{split}
	\end{align}
	Then the commutator in consideration is given by
		\begin{align*}
			\left[\partial_{t},\Box\right]u=&\sum_{I}\epsilon_{I}\left[\partial_{t},D_{I}\right]D_{I}u+\sum_{I}\epsilon_{I}D_{I}\left(\left[\partial_{t},D_{I}\right]u\right)-\sum_{I}\epsilon_{I}\partial_{t}\Gamma_{II}^{K} D_{K}u-\sum_{I}\epsilon_{I}\Gamma_{II}^{K}\left[\partial_{t},D_{K}\right]u.
		\end{align*}
		The third term on the right-hand side above is the most complicated and we compute it first:
		\begin{align*}
			-\sum_{I}\epsilon_{I}\partial_{t}\Gamma_{II}^{K} D_{K}u=&\partial_{t}\Gamma_{00}^{K}D_{K}u-\sum_{\tilI}\partial_{t}\Gamma_{\tilI\tilI}^{K}D_{K}u.
		\end{align*}
		For the first term we have
		\begin{align*}
			\partial_{t}\Gamma_{00}^{K}D_{K}u=&\hatTheta^{I}R^{K}{}_{0I0}D_{K}u+\hatTheta^{I}\left(\Gamma_{M0}^{K}\Gamma_{0I}^{M}+\Gamma^{K}_{0M}\Gamma^{M}_{I0}-\Gamma_{IM}^{K}\Gamma_{00}^{M}\right)D_{K}u-D_{0}\hatTheta^{I}\Gamma_{I0}^{K}D_{K}u.
		\end{align*}
		For the second term we have
		\begin{align*}
			-\sum_{\tilI}\partial_{t}\Gamma_{\tilI\tilI}^{K}D_{K}u=&-\sum_{\tilI}\hatTheta^{I}\left(R^{K}{}_{\tilI I\tilI}+\Gamma^{K}_{M\tilI}(\Gamma^{M}_{\tilI I}-\Gamma^{M}_{\tilI I})+\Gamma_{\tilI M}^{K}\Gamma^{M}_{I\tilI}-\Gamma^{K}_{IM}\Gamma^{M}_{\tilI\tilI}\right)D_{K}u+\sum_{\tilI}\Gamma_{I\tilI}^{K}D_{\tilI}\hatTheta^{I}D_{K}u
		\end{align*}
		Therefore we have
		\begin{align}\label{Commutator Box 1}
			\begin{split}
				-\sum_{I}\epsilon_{I}\partial_{t}\Gamma_{II}^{K} D_{K}u=&-\hatTheta^{I}\textrm{Ric}^{K}{}_{I}D_{K}u-\sum_{J}\epsilon_{J}\hatTheta^{I}\Gamma^{K}_{MJ}\Gamma_{JI}^{M}D_{K}u+\sum_{\tilI}\hatTheta^{I}\Gamma_{M\tilI}\Gamma_{\tilI I}^{M}D_{K}u\\
				&-\sum_{J}\epsilon_{J}\hatTheta^{I}\left(\Gamma_{JM}^{K}\Gamma^{M}_{IJ}-\Gamma^{K}_{IM}\Gamma^{M}_{JJ}\right)D_{K}u+\sum_{J}\epsilon_{J}\Gamma_{IJ}^{K}D_{J}\hatTheta^{I}D_{K}u.
			\end{split}
		\end{align}
		For the other terms in the commutator, we directly use the formula for the commutator $\left[\partial_{t},D_{I}\right]$:
		\begin{align}\label{Commutator Box 2}
			-\sum_{I}\epsilon_{I}\Gamma_{II}^{K}\left[\partial_{t},D_{K}\right]u=&\sum_{I}\epsilon_{I}\Gamma_{II}^{K}\left(D_{K}\hatTheta^{J}+\hatTheta^{M}\Gamma_{KM}^{J}\right)D_{J}u,
		\end{align}
		\begin{align}\label{Commutator Box 3}
			\begin{split}
				\sum_{I}\epsilon_{I}\left[\partial_{t},D_{I}\right]D_{I}u=&-\sum_{I}\epsilon_{I}\left(D_{I}\hatTheta^{J}+\hatTheta^{K}\Gamma_{IK}^{J}\right)D_{J}D_{I}u,
			\end{split}
		\end{align}
		and
		\begin{align}\label{Commutator Box 4}
			\begin{split}
				\sum_{I}\epsilon_{I}D_{I}\left(\left[\partial_{t},D_{I}\right]u\right)=&-\sum_{I}\epsilon_{I}D_{I}\left(\left(D_{I}\hatTheta^{J}+\hatTheta^{K}\Gamma_{IK}^{J}\right)D_{J}u\right)\\
				=&-\sum_{I}\epsilon_{I}\left(D_{I}D_{I}\hatTheta^{J}+D_{I}\hatTheta^{K}\Gamma^{J}_{IK}+\hatTheta^{K}D_{I}\Gamma_{IK}^{J}\right)D_{J}u\\
				&-\sum_{I}\epsilon_{I}\left(D_{I}\hatTheta^{J}+\hatTheta^{K}\Gamma_{IK}^{J}\right)D_{I}D_{J}u.
			\end{split}
		\end{align}
		Adding \eqref{Commutator Box 1}-\eqref{Commutator Box 4} we obtain
		\begin{align*}
			\left[\partial_{t},\Box\right]u=&-\sum_{I}\epsilon_{I}\left(D_{I}\hatTheta^{J}+\hatTheta^{K}\Gamma_{IK}^{J}\right)\left(D_{I}D_{J}u+D_{J}D_{I}u\right)-\Box\hatTheta^{J}D_{J}u\\
			&+\sum_{J}\epsilon_{J}\left(\Gamma_{IJ}^{K}-\Gamma_{JI}^{K}\right)D_{J}\hatTheta^{I}D_{K}u-\sum_{I}\epsilon_{I}\hatTheta^{K}D_{I}\Gamma_{IK}^{J}D_{J}u+\sum_{I}\epsilon_{I}\Gamma_{II}^{K}\Gamma_{KM}^{J}\hatTheta^{M}D_{J}u\\
			&-\hatTheta^{I}\textrm{Ric}^{K}{}_{I}D_{K}u-\sum_{J}\epsilon_{J}\hatTheta^{I}\Gamma^{K}_{MJ}\Gamma_{JI}^{M}D_{K}u+\sum_{\tilI}\hatTheta^{I}\Gamma_{M\tilI}\Gamma_{\tilI I}^{M}D_{K}u\\
			&-\sum_{J}\epsilon_{J}\hatTheta^{I}\left(\Gamma_{JM}^{K}\Gamma^{M}_{IJ}-\Gamma^{K}_{IM}\Gamma^{M}_{JJ}\right)D_{K}u,
		\end{align*}
		as desired.
	Similarly the commutator $[\partial_t,\gamma D_n]$ follows by writing $\gamma D_n=\frac{\epsilon_I}{2\sigma^2}(D_I\sigma^2)D_I$ and using the commutator identity for $[\partial_t,D_I]$.
\end{proof}
Applying Lemma~\ref{lem:commutators} yields higher order equations for the fluid variables which we record in the form of a few lemmas for future reference. 
\begin{lemma}\label{lem:Vinthigh}
	$\partial_t^k\Theta^I$ satisfies
	\begin{align*}
		\begin{split}
			\Box \partial_t^k\Theta^I= F,
		\end{split}
	\end{align*}
	in $\Omega$, where $F$ is a linear combination of contractions of $\Gamma$, $D\Gamma$, $\partial_t^{j_1}\sigma^2$, $\partial_t^{j_2}\Theta$, $D\partial_t^{j_3}\Theta$, $D^2\partial_t^{j_4}\Theta$, $\partial_t^{j_5}R$, and $D \partial_t^{j_6}R$, with $\sum j_i\leq k$, and $j_1,j_2, j_3\leq k$, $j_4,j_5,j_{6} \leq k-1$.
\end{lemma}
The next lemma contains the equation for $\partial_t^k\Theta^I$ on $\partial\Omega$.
\begin{lemma}\label{lem:Vbdryhigh}
	$\partial_t^k\Theta^I$ satisfies
	\begin{align*}
		\begin{split}
			(\partial_t^2+\gamma D_n)\partial_t^k\Theta^I= f,
		\end{split}
	\end{align*}
	on $\partial\Omega$, where $f$ is a linear combination of $\Gamma$, $D \partial_t^{j_1+1}\sigma^2$, $D \partial_t^{j_2}\sigma^2$, $\partial_t^{j_3}\Theta$, $\partial_t^{j_4+1}\Theta$, $D\partial_t^{j_5}\Theta$, and $\partial_t^{j_6}R$ with $\sum j_i\leq k$, and $j_1,j_{2}\leq k$ and $j_3,j_4,j_5, j_6\leq k-1$.
\end{lemma}
Next we record the equation for $\partial_t^{k+1}\sigma^2$.
\begin{lemma}\label{lem:sigmahigh}
	$\partial_t^{k+1}\sigma^2$ satisfies
	\begin{align*}
		\begin{split}
			\Box \partial_t^{k+1}\sigma^2=H,
		\end{split}
	\end{align*}
	in $\Omega$, where $H$ is a linear combination of contractions of $\Gamma$, $D\Gamma$, $\partial_t^{j_1}\sigma^2$, $\partial_t^{j_2}D_V\sigma^2$, $D\partial_t^{j_3}\sigma^2$, $D\partial_t^{j_4+1}\sigma^2$, $D^2\partial_t^{j_5}\sigma^2$, $\partial_t^{j_6}\Theta$, $D\partial_t^{j_7}\Theta$, $\partial_t^{j_8}R$, $D^2\partial_t^{j_9}\Theta$, $D\partial_t^{j_{10}}R$, with $\sum j_i\leq k$, and $j_1,\dots,j_8\leq k$ and $j_9, j_{10}\leq k-1$. 
\end{lemma}
Finally, the higher order equations for $W$ can be calculated using \eqref{eq:XAI1}, \eqref{eq:calIint1}, \eqref{eq:calIext1} \eqref{eq:calJ1}, and \eqref{eq:W2} and are given in the next lemma.
\begin{lemma}\label{lem:Whigh}
	$\partial_t^kW$ satisfies
	\begin{align*}
		\begin{split}
			\sum_{\mu=0}^3 B^\mu \partial_\mu \partial_t^kW = \calH,
		\end{split}
	\end{align*}
	where $\calH$ is a linear combination of contractions of $\Gamma$, $D\Gamma$, $e$, $D\partial_t^{j_1}\sigma^2$, $\partial_t^{j_2}\Theta$, $D\partial_t^{j_3}\Theta$, $\partial_t^{j_4}R$, $D^2\partial_t^{j_5}\sigma^2$, $D\partial_t^{j_6}W$,  with $\sum j_i\leq k$, and $j_1,\dots,j_5\leq k$ and $j_6\leq k-1$.
\end{lemma}

\subsection{Energy identities}\label{subsec:energy}
In the next few lemmas we prove the basic energy estimates.
\begin{lemma}\label{lem:Venergy1}
	Suppose $u$ satisfies
	\begin{align*}
		\begin{cases}
			\Box u = F\qquad &\mathrm{in~}\Omega\\
			(\partial_t^2+\gamma D_n)u= f\qquad&\mathrm{on~}\partial\Omega
		\end{cases}.
	\end{align*}
	Then
	\begin{align}\label{eq:Venergy1}
		\begin{split}
			&\int_{\Omega_T} (-(g^{-1})^{0\nu}\partial_tu\partial_\nu u +\frac{1}{2}(g^{-1})^{\mu\nu}\partial_\mu u \partial_\nu u) \sqrt{|g|}\ud x +\frac{1}{2}\int_{\partial\Omega_T}\gamma^{-1}(\partial_tu)^{2}\sqrt{|g|}\ud S\\
			&=\int_{\Omega_0} (-(g^{-1})^{0\nu}\partial_tu\partial_\nu u +\frac{1}{2}(g^{-1})^{\mu\nu}\partial_\mu u \partial_\nu u) \sqrt{|g|}\ud x +\frac{1}{2}\int_{\partial\Omega_0}\gamma^{-1}(\partial_tu)^{2}\sqrt{|g|}\ud S\\
			&\quad-\int_0^T\int_{\Omega_t}F \partial_t u\sqrt{|g|}\ud x \ud t+\int_0^T\int_{\partial\Omega_t}\gamma^{-1}f\partial_t u \sqrt{|g|} \ud S \ud t\\
			&\quad+\frac{1}{2}\int_0^T\int_{\Omega_t}(\partial_t(\sqrt{|g|}(g^{-1})^{\mu\nu}))\partial_\mu u \partial_\nu u\,\ud x \ud t+\frac{1}{2}\int_0^T\int_{\partial\Omega_t}(\partial_t(\gamma^{-1}\sqrt{|g|}))(\partial_tu)^2\ud S \ud t.
		\end{split}
	\end{align}
\end{lemma}
\begin{proof}
	Multiplying the interior and boundary equations by $\partial_tu\sqrt{|g|}$ and $\gamma^{-1}\partial_t u\sqrt{|g|}$, respectively, we get
	\begin{align*}
		\begin{split}
			-F\partial_tu = -\partial_\mu ((g^{-1})^{\mu\nu}\sqrt{|g|}\partial_\nu u \partial_t u)+\frac{1}{2}\partial_t((g^{-1})^{\mu\nu}\sqrt{|g|}\partial_\mu  u \partial_\nu u)-\frac{1}{2}(\partial_t((g^{-1})^{\mu\nu}\sqrt{|g|}))\partial_\mu u \partial_\nu u,
		\end{split}
	\end{align*}
	and
	\begin{align*}
		\begin{split}
			\gamma^{-1}f\partial_tu\sqrt{|g|} = \partial_{t}\left(\frac{1}{2\gamma}(\partial_tu)^2\sqrt{|g|}\right)+D_nu \partial_t u\sqrt{|g|} -\frac{1}{2}(\partial_t(\gamma^{-1}\sqrt{|g|}))(\partial_tu)^2.
		\end{split}
	\end{align*}
	The desired result follows from integrating these identities. 
\end{proof}
The following lemma will be used to control one derivative of next to top order $\partial_t$ derivatives of $\Theta$ on the boundary. It will be used when $D_nu$ and $\partial_tu$ can already be controlled, so the term $\alpha (g^{-1})^{\alpha\beta}\partial_\alpha u \partial_\beta u$  on $\partial\Omega$ gives control of other tangential derivatives in terms of $\partial_tu$ and $D_nu$.
\begin{lemma}\label{lem:nablaVGamma}
	Suppose $u$ satisfies
	\begin{align*}
		\begin{split}
			\Box u = F,\qquad \mathrm{in~}\Omega.
		\end{split}
	\end{align*}
	Then there is a future directed timelike vectorfield $Q=\partial_t+\alpha n$, for some constant $\alpha>0$, such that 
	\begin{align}\label{eq:nablaVGamma1}
		\begin{split}
			&\int_{\Omega_T}(-Qu (g^{-1})^{0\alpha}\partial_\alpha u +\frac{Q^0}{2}(g^{-1})^{\alpha\beta}\partial_\alpha u \partial_\beta u) \sqrt{|g|}\ud x- \int_{0}^T\int_{\partial\Omega_t}(Qu D_nu -\frac{\alpha}{2}(g^{-1})^{\alpha\beta}\partial_\alpha u \partial_\beta u) \sqrt{|g|}\,\ud S \ud t\\
			&=\int_{\Omega_0}(-Qu (g^{-1})^{0\alpha}\partial_\alpha u +\frac{Q^0}{2}(g^{-1})^{\alpha\beta}\partial_\alpha u \partial_\beta u) \sqrt{|g|}\ud x-\int_0^T\int_{\Omega_t}F Qu\sqrt{|g|}\ud x \ud t\\
			&\quad-\int_0^T\int_{\Omega_t}\Big((g^{-1})^{\alpha\beta}(\partial_\alpha Q^\mu)\partial_\mu u \partial_\beta u -\frac{1}{2}\partial_\mu(\sqrt{|g|}(g^{-1})^{\alpha\beta}Q^\mu)\partial_\alpha u \partial_\beta u\Big)\ud x \ud t.
		\end{split}
	\end{align}
\end{lemma}
\begin{proof}
	This follows from integrating the identity
	\begin{align}
		\Box u Q u \sqrt{|g|}&=\partial_\alpha\big(Q^\mu \partial_\mu u \sqrt{|g|}(g^{-1})^{\alpha\beta}\partial_\beta u -\frac{1}{2}\sqrt{|g|}(g^{-1})^{\mu\nu}Q^\alpha \partial_\mu u \partial_\nu u\big)\label{eq:Qmult1}\\
		&\quad-(g^{-1})^{\alpha\beta}(\partial_\alpha Q^\mu)\partial_\mu u \partial_\beta u\sqrt{|g|} +\frac{1}{2}\partial_\mu(\sqrt{|g|}(g^{-1})^{\alpha\beta}Q^\mu)\partial_\alpha u \partial_\beta u.\qedhere
	\end{align}
\end{proof}
The next lemma contains the basic energy estimate for $\partial_t\sigma^2$.
\begin{lemma}\label{lem:sigmaenergy}
	Suppose $u$ satisfies
	\begin{align*}
		\begin{cases}
			\Box u = H\qquad&\mathrm{in~}\Omega\\
			u\equiv \mathrm{constant}\qquad&\mathrm{on~}\partial\Omega
		\end{cases}.
	\end{align*}
	Then there is a future directed timelike vectorfield $Q=\partial_t-\alpha n$, for some constant $\alpha>0$, such that
	\begin{align*}
		\begin{split}
			&\int_{\Omega_T}(-Qu (g^{-1})^{0\alpha}\partial_\alpha u +\frac{Q^0}{2}(g^{-1})^{\alpha\beta}\partial_\alpha u \partial_\beta u) \sqrt{|g|}\ud x+\frac{\alpha}{2} \int_{0}^T\int_{\partial\Omega_t}(D_nu)^2 \sqrt{|g|}\,\ud S \ud t\\
			&=\int_{\Omega_0}(-Qu (g^{-1})^{0\alpha}\partial_\alpha u +\frac{Q^0}{2}(g^{-1})^{\alpha\beta}\partial_\alpha u \partial_\beta u) \sqrt{|g|}\ud x-\int_0^T\int_{\Omega_t}H Qu\sqrt{|g|}\ud x \ud t\\
			&\quad-\int_0^T\int_{\Omega_t}\Big((g^{-1})^{\alpha\beta}(\partial_\alpha Q^\mu)\partial_\mu u \partial_\beta u -\frac{1}{2}\partial_\mu(\sqrt{|g|}(g^{-1})^{\alpha\beta}Q^\mu)\partial_\alpha u \partial_\beta u\Big)\ud x \ud t.
		\end{split}
	\end{align*}
\end{lemma}
\begin{proof}
	This follows again from integrating \eqref{eq:Qmult1} and noting that since $u$ is constant on the boundary, $(g^{-1})^{\alpha\beta} \partial_\alpha u\partial_\beta u = (D_n u)^2$ there.
\end{proof}
Finally, for the curvature we use the following simple estimate (which is essentially the definition of the weak equation satisfied by the curvature).
\begin{lemma}\label{lem:Wenergy}
	Suppose $w$ satisfies
	\begin{align*}
		\begin{split}
			\sum_{\mu=0}^3\calB^\mu\partial_\mu w=\calH.
		\end{split}
	\end{align*}
	Then 
	\begin{align*}
		\begin{split}
			\sup_{0\leq t\leq T}\|w(t)\|_{L^2(\bbR^3)}\lesssim \|w(0)\|_{L^2(\bbR^3)}+\int_0^T\int_{\bbR^3}(|\calH|^2+\sum_{j=1}^3|\partial\calB^j|^2)\ud x\ud t.
		\end{split}
	\end{align*}
\end{lemma}
\begin{proof}
	This follows by multiplying the equation by $w$ and integrating.
\end{proof}

\subsection{Elliptic estimates}\label{subsec:elliptic}
Our goal in this section is to prove elliptic estimates. These will be needed both for $L^\infty$ estimate of the lower order terms and for bounding derivatives of some next to top order terms in $L^2$ (for instance $\partial \partial_t^{\ell-1} R$ in $L^2(\Omega_t)$). The main result is contained in the following proposition.
\begin{proposition}\label{prop:elliptic}
	Under the hypotheses of Proposition~\ref{prop:apriori} the following estimates hold for all $t\in[0,T)$ and with implicit constants independent of $C_1$:
	\begin{align*}
		\begin{split}
			\sum_{2p+k\leq \ell+2}\Big(\|\partial^p\partial_t^k\Theta\|_{L^2(\Omega_t)}^2+\|\partial^p\partial_t^{k+1}\sigma^2\|_{L^2(\Omega_t)}^2\Big)\lesssim\calE_{\ell}(t).
		\end{split}
	\end{align*}
	Moreover, 
	\begin{align*}
		\begin{split}
			\|\partial \partial_t^{k}R\|_{L^2(\Omega_t)}^2+\|\partial \partial_t^{k}R\|_{L^2(\Omega_t^c)}^2\lesssim \calE_\ell(t),\qquad k\leq \ell-1.
		\end{split}
	\end{align*}
	and
	\begin{align*}
		\begin{split}
			\sum_{2p+k\leq \ell}\Big(\|\partial^p\partial_t^{k}R\|_{L^2(\Omega_t)}^2+\|\partial^p\partial_t^{k}R\|_{L^2(\Omega_t^c)}^2\Big)\lesssim\calE_{\ell}(t).
		\end{split}
	\end{align*}
\end{proposition}
The rest of this section is devoted to the proof of Proposition~\ref{prop:elliptic}. Starting with the estimates on $\partial\partial_t^kR$, we first investigate the Maxwell system \eqref{Maxwell-inhomo1}, \eqref{Maxwell-inhomo2} which we write as:
\begin{align*}
	\nabla^{\mu}F_{\mu\nu}=J_{\nu},\quad \nabla_{[\mu}F_{\nu\lambda]}=I_{\mu\nu\lambda}.
\end{align*}
Working in the coordinates $(\partial_{t},\partial_{i}), i=1,2,3$ we define $\barE^{AB}\equiv \barE$ and $\barH^{AB}\equiv \barH$ as
\begin{align}\label{EM decom}
	\begin{split}
		\barE_{i}:=F_{it},\quad \barH^{k}:=-\frac{1}{2}\epsilon^{ijk}F_{ij}.
	\end{split}
\end{align}
Our elliptic estimates for the curvature will be obtained from divergence-curl systems for $\barE$ and $\barH$ which we will now derive. Choosing $\nu=\partial_{t}$  in the equation $\nabla^{\mu}F_{\mu\nu}=J_{\nu}$ gives
\begin{align*}
	\nabla^{i}F_{it}=J_{t},
\end{align*}
which, using $\nabla_t\partial_i=\nabla_i\partial_t$, 
can be written in terms of $\barE$ as
\begin{align}\label{eq: div barE 1}
	\begin{split}
		\nabla^{i}\barE_{i}-(g^{-1})^{it}F(\partial_{i},\nabla_{t}\partial_{t})-(g^{-1})^{ij}F(\partial_{i},\nabla_{t}\partial_{j})=J_{t}.
	\end{split}
\end{align}
For the last two terms on the left-hand side of \eqref{eq: div barE 1}, we proceed as follows. First we write $\partial_{t}$ and $\partial_{j}$ as
\begin{align*}
	\partial_{t}=a_{0}^{I}e_{I},\quad \partial_{j}=a_{j}^{J}e_{J},
\end{align*}
where $a_{0}^{I}, a_{j}^{J}$ depend on $e_{K}^{\mu}$ algebraically. So we have
\begin{align*}
	\nabla_{t}\partial_{t}=\partial_{t}a_{0}^{I}\,e_{I},\quad \nabla_{t}\partial_{j}=\partial_{t}a_{j}^{J}\,e_{J},
\end{align*}
which in turn gives
\begin{align*}
	F(\partial_{i}\nabla_{t}\partial_{t})=\partial_{t}a_{0}^{I}\,F_{i I}=e_{I}^{\mu}\,\partial_{t}a_{0}^{I}\,F_{i\mu},\quad F(\partial_{i},\nabla_{t}\partial_{j})=\partial_{t}a_{j}^{J}\,F_{iJ}=e_{J}^{\nu}\,\partial_{t}a_{j}^{J}\,F_{i\nu}.
\end{align*}
Therefore, equation \eqref{eq: div barE 1} gives
\begin{align}\label{eq: div barE 2}
	\begin{split}
		(g^{-1})^{ij}\nabla_{j}\barE_{i}=-(g^{-1})^{it}\nabla_{t}\barE_{i}+J_{t}+C_{\div,\barE}^{j}\barE_{j}+C_{\div,\barH}^{k}\barH_{k},
	\end{split}
\end{align}
where the coefficients $C_{\div,\barE}^{j},C_{\div,H}^{k}$ depend on $\Gamma_{IJ}^{K}$ and $\hatTheta$, $D\hatTheta$. Note that arguing as above, $\nabla_t E_i$ can be expressed in terms of $F_{IJ}$, $\partial_t F_{IJ}$, $e_I$, and $\partial_t e_I$. Now we need to convert the connection $\nabla$ over the entire spacetime to the connection $\barnabla$ on $\Sigma_{t}$. By its definition, the induced connection $\barnabla$ is the restriction of $\nabla$ on $\Sigma_{t}$, so $\nabla_{i}\partial_{j}-\barnabla_{i}\partial_{j}$ is orthogonal to $\Sigma_{t}$. More precisely, with $\hatT=\frac{-(g^{-1})^{0\nu}}{\sqrt{-(g^{-1})^{00}}}\partial_\nu$ be the future-directed unit normal vectorfield of $\Sigma_{t}$. Then
\begin{align*}
	\nabla_{i}\partial_{j}-\barnabla_{i}\partial_{j}=-g\left(\nabla_{i}\partial_{j}-\barnabla_{i}\partial_{j},\hatT\right)\hatT=-g\left(\nabla_{i}\partial_{j},\hatT\right)\hatT=g(\partial_{j},\nabla_{i}\hatT)\hatT,
\end{align*}
which depends on $e_{I}^{\mu}$ and $D e_{I}^{\mu}$. Therefore the equation \eqref{eq: div barE 2} can be rewritten as
\begin{align}\label{eq: div barE 3}
	\begin{split}
		(g^{-1})^{ij}\barnabla_{j}\barE_{i}=\rho_\barE,
	\end{split}
\end{align}
where $\rho_\barE$ depends on $e_{I}^{\mu}$, $D e_{I}^{\mu}$, $\Gamma_{IJ}^{K}$, $\hatTheta$, $D\hatTheta$, $F_{IJ}$, $\partial_tF_{IJ}$ and $J_{t}$.
		Similarly, we can derive a divergence equation for $\barH$ as follows. In the equation $\nabla_{[\alpha}F_{\beta\gamma]}=I_{\alpha\beta\gamma}$ we choose $\alpha=i,\beta=j,\gamma=k$:
		\begin{align}\label{eq: div H pre}
			\nabla_{i}F_{jk}+\nabla_{j}F_{ki}+\nabla_{k}F_{ij}=I_{ijk}.
		\end{align}
		Using the relation $F_{ij}=-\epsilon_{ijk}\barH^{k}$, we can rewrite the equation \eqref{eq: div H pre} as
		\begin{align}\label{eq: div H 1}
			-I_{ijk}=&\nabla_{i}\left(\epsilon_{jk\ell}\barH^{\ell}\right)+\nabla_{j}\left(\epsilon_{ki\ell}\barH^{\ell}\right)+\nabla_{k}\left(\epsilon_{ij\ell}\barH^{\ell}\right).
		\end{align}
		In the equations \eqref{eq: div H pre}-\eqref{eq: div H 1}, the index $ijk$ can only be $123$ or its permutation so choosing $ijk=123$ and arguing as in the derivation of \eqref{eq: div barE 3} we arrive at 
		\begin{align}\label{eq: div H 2}
			\begin{split}
				\barnabla_{i}\barH^{i}=\rho_\barH,
			\end{split}
		\end{align}
		where $\rho_\barH$ depends on $e_{I}^{\mu}$, $D e_{I}^{\mu}$, $\Gamma_{IJ}^{K}$, $\hatTheta$, $D\hatTheta$, $F_{IJ}$, and $\partial_tF_{IJ}$.
		
		Next we turn to the curl equations for $\barE$ and $\barH$. 
		Choosing $\nu=i$ in the equation $\nabla^{\mu}F_{\mu\nu}=J_{\nu}$ gives\begin{align}\label{eq: curl H pre}
			\begin{split}
				(g^{-1})^{tt}\nabla_{t}F_{ti}+(g^{-1})^{tk}\nabla_{k}F_{ti}+(g^{-1})^{jt}\nabla_{t}F_{ji}+(g^{-1})^{jk}\nabla_{k}F_{ji}=J_{i}
			\end{split}
		\end{align}
		The first and third terms on the left-hand side of \eqref{eq: curl H pre} can be expressed in terms of $F_{IJ}$, $\partial_tF_{IJ}$, $e_I$, and $\partial_te_I$ as before.
		For the second term on the left-hand side of \eqref{eq: curl H pre}, for coefficients depending on $e_{I}, De_{I},\Gamma_{IJ}^{K}$ and $\hatTheta$, $D\hatTheta$, we have
		\begin{align*}
			(g^{-1})^{tk}\nabla_{k}F_{ti}=&-(g^{-1})^{tk}\barnabla_{k}\barE_{i}+C_{\barE}^{k}\barE_{k}+C_{\barH}^{j}\barH_{j}.
		\end{align*}
		For the fourth term on the LHS of \eqref{eq: curl H pre}, and with coefficients depending on $\hatTheta$, $D\hatTheta$, $e_I$, and $De_I$, we have
		\begin{align*}
			(g^{-1})^{jk}\nabla_{k}F_{ji}=-(g^{-1})^{jk}\epsilon_{ji\ell}\barnabla_{k}\barH^{\ell}+D_{\curl,\barE}^{k}\barE_{k}+D_{\curl,\barH}^{j}\barH_{j}.
		\end{align*}
		Putting everything together, equation \eqref{eq: curl H pre} gives
		\begin{align}\label{eq: curl H 1}
			\begin{split}
				(g^{-1})^{jk}\epsilon_{ji\ell}\barnabla_{k}\barH^{\ell}+(g^{-1})^{tk}\barnabla_{k}\barE_{i}=\sigma^\barH_i,
			\end{split}
		\end{align}
		where $\sigma^\barH_i$ depends on $e_{I}^{\mu}$, $D e_{I}^{\mu}$, $\Gamma_{IJ}^{K}$, $\hatTheta$, $D\hatTheta$, $F_{IJ}$, and $\partial_tF_{IJ}$. Finally using equation $\nabla_{[\alpha}F_{\beta\gamma]}=I_{\alpha\beta\gamma}$  with $\alpha\beta\gamma=tij$, similar considerations as above yield
		\begin{align}\label{eq: curl barE pre}
			\begin{split}
				\barnabla_{i}\barE_{j}-\barnabla_{j}\barE_{i}=\sigma^\barE_{ij},
			\end{split}
		\end{align}
		where $\sigma^\barE_{ij}$ depends on $e_{I}^{\mu}$, $D e_{I}^{\mu}$, $\Gamma_{IJ}^{K}$, $\hatTheta$, $D\hatTheta$, $F_{IJ}$, and $\partial_tF_{IJ}$.
		
		In order to use equations \eqref{eq: div barE 3}, \eqref{eq: div H 2}, \eqref{eq: curl H 1}, and \eqref{eq: curl barE pre}  to prove elliptic estimates for the curvature, we first need some preliminary linear estimates for these systems. First we prove a positivity property of the tensor $\gamma^{ij}:=(g^{-1})^{ij}$.
		\begin{lemma}\label{lem: gamma positivity}
			As a symmetric operator acting on the cotangent bundle of the Riemannian manifold $(M:=\{x^{0}=\textrm{constant}\},\barg)$, $\gamma$ is positively definite.
		\end{lemma}
		\begin{proof}
			We only need to check the positive definiteness at any given point $p\in M$. Since $\gamma$ is a tensor we can work in normal coordinates $(z^1,z^2,z^3)$ near $p$, so that at $p$
			\begin{align*}
				\begin{split}
					\barg_{ab}=\delta_{ab},\qquad (\barg^{-1})^{ab}=\delta^{ab}.
				\end{split}
			\end{align*} 
			In particular if $\xi=(\xi_1,\xi_2,\xi_3)$ is any cotangent vector at $p$ then
			\begin{align*}
				\begin{split}
					\barg^{-1}(\xi,\xi)=|\xi|^2:=\sum\xi_j^2.
				\end{split}
			\end{align*}
			Therefor in coordinates $(x^0,z^1,z^2,z^3)$ the metric $g$ looks like
			\begin{align*}
				\begin{split}
					g=\pmat{a_0&a_1&a_2&a_3\\a_1&1&0&0\\a_2&0&1&0\\a_3&0&0&1},
				\end{split}
			\end{align*}
			where $a_0<0$ and in fact by continuity we must have $a_0<-c<0$ during the evolution for some $c>0$. By computation, the inverse matrix is
			\begin{align*}
				\begin{split}
					g^{-1}=\frac{1}{a_0-\sum a_j^2}\pmat{1&-a_1&-a_2&-a_3\\-a_1&a_0-\sum_{i\neq1}a_i^2&a_1a_2&a_1a_3\\-a_2&a_2a_1&a_0-\sum_{i\neq2}a_i^2&a_2a_3\\-a_3&a_3a_1&a_3a_2&a_0-\sum_{i\neq3}a_i^2}.
				\end{split}
			\end{align*}
			It follows that
			\begin{align*}
				\begin{split}
					\gamma=I+\frac{1}{a_{0}-\sum a_j^2}\underbrace{\pmat{a_1\\a_2\\a_3}\pmat{a_1&a_2&a_3}}_{K}.
				\end{split}
			\end{align*}
			The matrix $K$ has two zero eigenvalues and a non-zero eigenvalue $\sum_{i}a_{i}^{2}$, so the result follows.
		\end{proof}
		\begin{lemma}\label{prop: H1 div curl sys 1}
			Let $\Theta$ be a $1$-form on $\Sigma_{t}$ satisfying (weakly)
			\begin{align*}
				\barnabla_{i}\Theta_{j}-\barnabla_{j}\Theta_{i}=\sigma_{ij},\quad \gamma^{ij}\barnabla_{i}\Theta_{j}=\rho.
			\end{align*}
			Then
			\begin{align*}
				\left\|\barnabla\Theta\right\|_{L^{2}(\Sigma_{t})}\lesssim \left\|\Theta\right\|_{L^{2}(\Sigma_{t})}+ \left\|\rho\right\|_{L^{2}(\Sigma_{t})}+\left\|\sigma\right\|_{L^{2}(\Sigma_{t})}.
			\end{align*}
			Here the $L^{2}$-norm is with respect to the tensor $\gamma^{ij}=(g^{-1})^{ij}$, namely,
			\begin{align*}
				\left|\barnabla\Theta\right|^{2}:=\gamma^{aa'}\gamma^{bb'}\barnabla_{a}\Theta_{b}\barnabla_{a'}\Theta_{b'},\quad \left|\Theta\right|^{2}:=\gamma^{ab}\Theta_{a}\Theta_{b},\quad \left|\sigma\right|^{2}:=\gamma^{aa'}\gamma^{bb'}\sigma_{ab}\sigma_{a'b'}.
			\end{align*}
			The implicit constant in the estimate depends on the $L^{\infty}$-norm of $e_{I}^{\mu}, D e_{I}^{\mu}$ and the curvature tensor $R$.
		\end{lemma}
		\begin{proof}
			We carry out the proof assuming $\Theta$ is more regular, and the general case follows from the resulting estimate and a standard approximation procedure.	We start with
			\begin{align}\label{barnabla Theta magnitude pre}
				\begin{split}
					\left|\barnabla\Theta\right|^{2}=&\gamma^{aa'}\gamma^{bb'}\barnabla_{a}\Theta_{b}\barnabla_{a'}\Theta_{b'}\\
					=&\gamma^{aa'}\gamma^{bb'}\barnabla_{a}\Theta_{b}\barnabla_{b'}\Theta_{a'}+\gamma^{aa'}\gamma^{bb'}\barnabla_{a}\Theta_{b}\,\sigma_{a'b'}\\
					=&\barnabla_{a}\left(\gamma^{aa'}\gamma^{bb'}\Theta_{b}\barnabla_{b'}\Theta_{a'}\right)-\gamma^{aa'}\gamma^{bb'}\Theta_{b}\barnabla_{a}\barnabla_{b'}\Theta_{a'}\\
					&-\barnabla_{a}\left(\gamma^{aa'}\gamma^{bb'}\right)\Theta_{b}\barnabla_{b'}\Theta_{a'}+\gamma^{aa'}\gamma^{bb'}\barnabla_{a}\Theta_{b}\,\sigma_{a'b'}
				\end{split}
			\end{align}
			The second term on the right-hand side in \eqref{barnabla Theta magnitude pre} can be written as
			\begin{align*}
				-\gamma^{aa'}\gamma^{bb'}\Theta_{b}\barnabla_{a}\barnabla_{b'}\Theta_{a'}=&-\gamma^{aa'}\gamma^{bb'}\Theta_{b}\barnabla_{b'}\barnabla_{a}\Theta_{a'}-\gamma^{aa'}\gamma^{bb'}\Theta_{b}\barR_{ab'a'}{}^{c}\Theta_{c}\\
				=&-\barnabla_{b'}\left(\gamma^{aa'}\gamma^{bb'}\Theta_{b}\barnabla_{a}\Theta_{a'}\right)+\barnabla_{b'}\left(\gamma^{aa'}\gamma^{bb'}\right)\Theta_{b}\barnabla_{a}\Theta_{a'}\\
				&+\rho^{2}-\gamma^{aa'}\gamma^{bb'}\Theta_{b}\barR_{ab'a'}{}^{c}\Theta_{c}.
			\end{align*}
			Here $\barR$ is the Riemannian curvature tensor of the induced metric on $\Sigma_{t}$. So the identity \eqref{barnabla Theta magnitude pre} is rewritten as
			\begin{align}\label{barnabla Theta magnitude 1}
				\begin{split}
					\left|\barnabla\Theta\right|^{2}=&\barnabla_{a}\left(\gamma^{aa'}\gamma^{bb'}\Theta_{b}\barnabla_{b'}\Theta_{a'}\right)-\barnabla_{b'}\left(\gamma^{aa'}\gamma^{bb'}\Theta_{b}\barnabla_{a}\Theta_{a'}\right)\\&+\barnabla_{b'}\left(\gamma^{aa'}\gamma^{bb'}\right)\Theta_{b}\barnabla_{a}\Theta_{a'}+\rho^{2}-\gamma^{aa'}\gamma^{bb'}\Theta_{b}\barR_{ab'a'}{}^{c}\Theta_{c}\\
					&-\barnabla_{a}\left(\gamma^{aa'}\gamma^{bb'}\right)\Theta_{b}\barnabla_{b'}\Theta_{a'}+\gamma^{aa'}\gamma^{bb'}\barnabla_{a}\Theta_{b}\,\sigma_{a'b'}
				\end{split}
			\end{align}
			Then the desired estimate follows by integrating the above identity on $\Sigma_{t}$ and Cauchy-Schwarz.
		\end{proof}
		\begin{lemma}\label{prop: H1 div curl sys 2}
			Let $X$ be vectorfield on $\Sigma_{t}$ satisfying (weakly)
			\begin{align*}
				\barnabla_{i}X^{i}=\rho,\quad \gamma^{jk}\epsilon_{ij\ell}\barnabla_{k}X^{\ell}=\sigma_{i}.
			\end{align*}
			Then
			\begin{align*}
				\left\|\barnabla X\right\|_{L^{2}(\Sigma_{t})}\lesssim \left\|X\right\|_{L^{2}(\Sigma_{t})}+\left\|\rho\right\|_{L^{2}(\Sigma_{t})}+\left\|\sigma\right\|_{L^{2}(\Sigma_{t})}.
			\end{align*}
			Here we define
			\begin{align*}
				\left|\barnabla X\right|^{2}:=\gamma^{aa'}\gamma^{bb'}\barnabla_{a}X_{b}\barnabla_{a'}X_{b'},\quad \left|X\right|^{2}:=\gamma^{ab}X_{a}X_{b},\quad \left|\sigma\right|^{2}:=g_{aa'}g_{bb'}\sigma^{ab}\sigma^{a'b'},
			\end{align*}
			where $X_{a}$ is defined such that $X^{b}=\gamma^{ab}X_{a}$, and $\sigma^{ab}$ is anti-symmetric and defined such that $\sigma_{i}=\epsilon_{iab}\sigma^{ab}$. The implicit constant in the estimate depends on the same quantities as the one in Lemma \ref{prop: H1 div curl sys 1}.
		\end{lemma}
		\begin{proof}
			As in the proof of Proposition \ref{prop: H1 div curl sys 1} we work under the assumption that $X$ is smooth. We rewrite the system as
			\begin{align*}
				\gamma^{ij}\barnabla_{i}X_{j}=\rho-X_{j}\barnabla_{i}\gamma^{ij},\quad \gamma^{jk}\gamma^{\ell m}\epsilon_{ij\ell}\barnabla_{k}X_{m}=\sigma_{i}-\gamma^{jk}\epsilon_{ij\ell}X_{m}\barnabla_{k}\gamma^{\ell m}.
			\end{align*}
			The second equation above can be further manipulated into
			\begin{align*}
				\gamma^{jk}\gamma^{\ell m}\barnabla_{k}X_{m}-\gamma^{\ell k}\gamma^{jm}\barnabla_{k}X_{m}=\sigma^{j\ell}-\gamma^{jk}\barnabla_{k}\gamma^{\ell m}\,X_{m}+\gamma^{\ell k}\barnabla_{k}\gamma^{jm}\,X_{m}=:\sigma^{\prime j\ell}.
			\end{align*}
			Now we are in a position to mimic the proof of Proposition \ref{prop: H1 div curl sys 1}. We consider
			\begin{align}\label{barnabla X magnitude pre}
				\left|\barnabla X\right|^{2}=&\gamma^{aa'}\gamma^{bb'}\barnabla_{a}X_{b}\barnabla_{a'}X_{b'}
				=\gamma^{ba'}\gamma^{ab'}\barnabla_{a}X_{b}\barnabla_{a'}X_{b'}+\barnabla_{a}X_{b}\,\sigma^{\prime ab}.
			\end{align}
			The first term on the right-hand side of \eqref{barnabla X magnitude pre} can be rewritten as
			\begin{align}\label{barnabla X magnitude 1}
				\gamma^{ba'}\gamma^{ab'}\barnabla_{a}X_{b}\barnabla_{a'}X_{b'}=&\nabla_{a}\left(\gamma^{ba'}\gamma^{ab'}X_{b}\barnabla_{a'}X_{b'}\right)-\barnabla_{a}\left(\gamma^{ba'}\gamma^{ab'}\right)X_{b}\barnabla_{a'}X_{b'}-\gamma^{ba'}\gamma^{ab'}X_{b}\barnabla_{a}\barnabla_{a'}X_{b'}.
			\end{align}
			The last term on the right-hand side of \eqref{barnabla X magnitude 1} can be manipulated into
			\begin{align*}
				-\gamma^{ba'}\gamma^{ab'}X_{b}\barnabla_{a}\barnabla_{a'}X_{b'}=&-\gamma^{ba'}\gamma^{ab'}X_{b}\barnabla_{a'}\barnabla_{a}X_{b'}-\gamma^{ba'}\gamma^{ab'}X_{b}\barR_{aa'b'}{}^{c}X_{c}\\
				=&-\barnabla_{a'}\left(\gamma^{ba'}\gamma^{ab'}X_{b}\barnabla_{a}X_{b'}\right)+\barnabla_{a'}\left(\gamma^{ba'}\gamma^{ab'}\right)X_{b}\barnabla_{a}X_{b'}\\
				&+\gamma^{ba'}\gamma^{ab'}\barnabla_{a'}X_{b}\barnabla_{a}X_{b'}-\gamma^{ba'}\gamma^{ab'}X_{b}\barR_{aa'b'}{}^{c}X_{c}.
			\end{align*}
			It follows that
			\begin{align}\label{barnabla X magnitude 2}
				\begin{split}
					\left|\barnabla X\right|^{2}=&\nabla_{a}\left(\gamma^{ba'}\gamma^{ab'}X_{b}\barnabla_{a'}X_{b'}\right)-\barnabla_{a}\left(\gamma^{ba'}\gamma^{ab'}\right)X_{b}\barnabla_{a'}X_{b'}\\
					&-\barnabla_{a'}\left(\gamma^{ba'}\gamma^{ab'}X_{b}\barnabla_{a}X_{b'}\right)+\barnabla_{a'}\left(\gamma^{ba'}\gamma^{ab'}\right)X_{b}\barnabla_{a}X_{b'}\\
					&+\rho^{2}-\gamma^{ba'}\gamma^{ab'}X_{b}\barR_{aa'b'}{}^{c}X_{c}+\barnabla_{a}X_{b}\,\sigma^{\prime ab}.
				\end{split}
			\end{align}
			The desired estimate now follows from integrating this identity and using Cauchy-Schwarz.
		\end{proof}
We next state the linear estimates which are needed for proving elliptic estimates for the fluid quantities. A basic ingredient in deriving elliptic estimates for $\partial_t^k\Theta$ and $\partial_t^{k+1}\sigma^2$ is the decomposition of the scalar wave operator into an elliptic part and a part involving $\partial_t$ derivatives. This is accomplished in the next lemma. Recall that (see \eqref{eq:g1} and \eqref{eq:e0mu1})
\begin{align*}
				\begin{split}
					(g^{-1})^{ij}= (m^{-1})^{IJ}e_I^ie_J^j=\big(\delta^{\tilI\tilJ}-\frac{\hatTheta^\tilI}{\hatTheta^0}\frac{\hatTheta^\tilJ}{\hatTheta^0}\big)e_\tilI^ie_\tilJ^j,
				\end{split}
			\end{align*}
			which is positive definite (see Lemma~\ref{lem: gamma positivity}). Let $\barA$ denote the associated second order elliptic operator:
			\begin{align}\label{eq:barA1}
				\begin{split}
					\barA:=(g^{-1})^{ij}\partial^2_{ij}.
				\end{split}
			\end{align}
			\begin{lemma}\label{lem:BoxA}
				For any function $u$,
				\begin{align*}
					\begin{split}
						\barA u&=\Box u+\bigg(\Big(\frac{1-\hatTheta^\tilI e_\tilI^0}{\hatTheta^0}\Big)^2-\sum_{\tilI}(e_{\tilI}^{0})^{2}\bigg)\partial_t^2u-2\Big(e_\tilI^je_\tilI^0+\frac{(1-\hatTheta^\tilI e_\tilI^0)\hatTheta^\tilJ e_\tilJ^j}{(\hatTheta^0)^2}\Big)\partial_j\partial_tu\\
						&-\epsilon_I(D_Ie_I^\mu)\partial_\mu u+\epsilon_I\Gamma^K_{II}e_K^\mu\partial_\mu u.
					\end{split}
				\end{align*}
			\end{lemma}
			\begin{proof}
				This follows from \eqref{eq:Boxframe1}, and by expanding $e_I^\mu= e_I^\mu \partial_\mu$ where for $I=0$ we use \eqref{eq:e0mu1}. 
			\end{proof}
			We also recall the following the following standard results from elliptic theory without proof (see \cite{Taylor-book1}).
			\begin{lemma}\label{lem:2nd order elliptic esti}
				With $\barA$ as in \eqref{eq:barA1} and any scalar function $u$,
				\begin{align}\label{eq:ellipticlemDir1}
					\|\partial_x^2u\|_{L^{2}(\Omega_{t})}\lesssim \|\barA u\|_{L^{2}(\Omega_{t})}+\|u\|_{H^{\frac32}(\partial\Omega_{t})},
				\end{align}
				and 
				\begin{align}\label{eq:ellipticlemNeu1}
					\|\partial_x^2u\|_{L^{2}(\Omega_{t})}\lesssim \|\barA u\|_{L^{2}(\Omega_{t})}+\|D_N u\|_{H^{\frac12}(\partial\Omega_{t})},
				\end{align}
				where $N$ is a transversal vectorfield to $\partial\Omega_{t}\subset\Omega_{t}$ and the implicit constants depend on $\Omega_{t}$.
			\end{lemma}
			\begin{proof}[Proof of Proposition \ref{prop:elliptic}] Let $M\leq \ell$. 
				We use Lemmas~\ref{prop: H1 div curl sys 1},~\ref{prop: H1 div curl sys 2}, and~\ref{lem:2nd order elliptic esti}. The proof is by induction on the order $p$ of $\partial_{x}^p$, starting with the fluid quantities $\Theta$ and $\sigma^2$. When $p=1$ the estimates for $\partial \partial_t^k\Theta$ and $\partial \partial_t^k\sigma^2$ follow from the definition of the energies.  Now we assume that the estimate holds for orders less or equal to $1\leq p\leq\frac{M+2}{2}-1$, that is,
				\begin{align}\label{eq:JIH}
					\begin{split}
						\sum_{q\leq p}\sum_{k+2q\leq M+2}\|\partial^q \partial_t^{k}\Theta\|_{L^2(\Omega_t)}^2+\sum_{q\leq p}\sum_{k+2q\leq M+2}\|\partial^{q}\partial_t^{k+1}\sigma^{2}\|_{L^{2}(\Omega_{t})}^{2}\lesssim \calE_\ell(t),
					\end{split}
				\end{align}	
				and prove the estimates for $p+1$, that is,
				\begin{align}\label{eq:JIG}
					\begin{split}
						\sum_{k\leq M-2p}\|\partial^{p+1} \partial_t^{k}\Theta\|_{L^2(\Omega_t)}^2+\sum_{k\leq M-2p}\|\partial^{p+1}\partial_t^{k+1}\sigma^{2}\|_{L^{2}(\Omega_{t})}^{2}\lesssim \calE_\ell(t).
					\end{split}
				\end{align}
				To prove this we apply estimates \eqref{eq:ellipticlemNeu1} and \eqref{eq:ellipticlemDir1} in Lemma~\ref{lem:2nd order elliptic esti} to the equations for $\partial^{p-1}\partial_t^k\Theta$ and $\partial^{p-1}\partial_t^{k+1}\sigma^2$ respectively (more precisely we first commute tangential derivatives so that the boundary equation for $\Theta$ and boundary condition for $\partial_t\sigma^2$ can be used, and then use the equation algebraically to recover the estimates for normal derivatives). The details are similar to the proof of Proposition~2.11 in \cite{MSW1}, except that now we also need to treat the contribution of the curvature on the right-hand side of these equations. In view of Lemmas~\ref{lem:Vinthigh},~\ref{lem:Vbdryhigh} and ~\ref{lem:sigmahigh} and the trace theorem, the highest order curvature term that we need to control is $\|\partial^p\partial_t^{k-1}R\|_{L^2(\Omega_t)}$. But since $2p+(k-1)\leq M-1\leq \ell-1$, we are able to use the estimate on $\partial^{p} \partial_t^kR$ in the definition of $\scE_\ell$ and fundamental theorem of calculus to estimate this term by a small multiple of $\scE_\ell(t)$, provided $T$ is sufficiently small, which can be absorbed after combining with the elliptic estimates with for the curvature. Turning to the estimates for the curvature, we first note that the case $p=1$ follows from Lemmas~\ref{prop: H1 div curl sys 1} and~\ref{prop: H1 div curl sys 2}. Arguing inductively as before, we assume that for some $1\leq p \leq \frac{M}{2}-1$,
				\begin{align*}
					\begin{split}
						\sum_{q\leq p}\sum_{2q+k\leq M}\Big(\|\partial^p\partial_t^{k}R\|_{L^2(\Omega_t)}^2+\|\partial^p\partial_t^{k}R\|_{L^2(\Omega_t^c)}^2\Big)\lesssim\calE_{\ell}(t).
					\end{split}
				\end{align*}
				and prove the estimate for $p+1$, that is,
				\begin{align}\label{eq:JIHR}
					\begin{split}
						\sum_{k\leq M-2p-2}\Big(\|\partial^{p+1}\partial_t^{k}R\|_{L^2(\Omega_t)}^2+\|\partial^{p+1}\partial_t^{k}R\|_{L^2(\Omega_t^c)}^2\Big)\lesssim\calE_{\ell}(t).
					\end{split}
				\end{align}
				Note that since $2p+2+k\leq M\leq\ell$, and since we have already treated the estimates for $\Theta$ and $\sigma^2$, we may equivalently use the components $R_{IJKL}$ or the coordinate expressions $R_{\mu\nu\lambda\kappa}$ to prove the desired estimate. Indeed, the difference amounts to estimating $\partial^{p+1}\partial_t^ke_I^\mu$, which can be handled using the elliptic estimates for $\Theta$ and equations~\eqref{eq:etransport1} and~\eqref{eq:Gammatransport1}.  The proof of \eqref{eq:JIHR} is similar to that of \eqref{eq:JIH}, where now we use Lemmas~\ref{prop: H1 div curl sys 1} and~\ref{prop: H1 div curl sys 2} instead of Lemma~\ref{lem:2nd order elliptic esti}. Indeed, we first commute tangential covariant derivatives (see the proof of Proposition~2.11 in \cite{MSW1} for a similar argument for commuting tangential derivatives) with the Maxwell system \eqref{Maxwell-inhomo1}, \eqref{Maxwell-inhomo2}. Then  defining $E$ and $H$ in terms of $\snabla^p\nabla_t^k F$, Lemmas~\ref{prop: H1 div curl sys 1} and~\ref{prop: H1 div curl sys 2}, the induction hypothesis, and \eqref{eq:JIH} allow us to estimate the tangential derivatives in \eqref{eq:JIHR}. Note that as argued above for the components of $R$, the difference between the covariant and coordinate derivatives can be controlled in terms of derivatives of $e_I$ using \eqref{eq:JIH}. Note that in the highest derivative of $\Theta$ that comes up after commuting derivatives comes from the term $\partial^{p+1}\partial_t^ke_I$. When $k\geq 1$, using \eqref{eq:etransport1}, this amounts to estimating $\partial^{p+2}\partial_t^{k-1}\Theta$, which can be done because $2(p+2)+k-1\leq M+1$ in \eqref{eq:JIH}. When $k=0$, we can still use \eqref{eq:etransport1} and the fundamental theorem of calculus to estimate $\partial^{p+1}e_I$ by the same numerology as before, which gives $2(p+2)\leq M+2$. Finally to estimate the remaining normal derivatives, we use equations~\eqref{eq: div barE 3},~\eqref{eq: div H 2},~\eqref{eq: curl H 1}, and~\eqref{eq: curl barE pre}, and the positivity of $\gamma$ in Lemma~\ref{lem: gamma positivity}, to express the normal derivatives in terms of tangential ones. Note that here the equations are satisfied almost everywhere on $\Omega_t$ and $\Omega_t^c$, and we can apply normal derivatives on each side independently. 
			\end{proof}
			\subsection{Proof of Proposition~\ref{prop:apriori}}\label{subsec:apriori}
			This section contains the proof of Proposition~\ref{prop:apriori}. We start with two auxiliary lemmas. The first is used to estimate two derivatives of $\partial_t^{\ell}\sigma^2$ in terms of the top order energy.
			\begin{lemma}\label{lem:d2sigma}
				Under the hypotheses of Proposition~\ref{prop:apriori}, if $T>0$ is sufficiently small then for any $k\leq \ell$ and $t\in[0,T]$
				\begin{align*}
					\begin{split}
						\|\partial^2\partial_t^k\sigma^2\|_{L^2(\Omega_t)}^2\lesssim \calE_\ell(T).
					\end{split}
				\end{align*} 
				Here the implicit constant depends polynomially on $\calE_{k-1}(T)$ if $k$ is sufficiently large, and on $C_1$ for small $k$.
			\end{lemma}
			\begin{proof}
				We apply Lemma~\ref{lem:BoxA}. For $k$ small we can use the fundamental theorem of calculus by integrating $\partial_t\|\partial^2\partial_t^k\sigma^2\|_{L^2(\Omega_t)}^2$ in time, and using the bootstrap assumptions. For $k$ large, proceeding inductively, we need to show that $\|H\|_{L^2(\Omega_t)}$, with $H$ as in Lemma~\ref{lem:sigmahigh}, is bounded by $\calE_\ell$. But this follows from the structure of $H$ given in Lemma~\ref{lem:sigmahigh}, the elliptic estimates in Proposition~\ref{prop:elliptic}, and the induction hypothesis (that is, the desired estimate on $\|\partial^2\partial_t^{k-1}\sigma^2\|_{L^2(\Omega_t)}$).  See the proof of Lemma 2.17 in \cite{MSW1} for a related argument in the flat case.
			\end{proof}
			The next lemma will be used to estimate the $L^2(\partial\Omega_0^T)$ norm of $\partial \partial_t^{\ell-1}\Theta$.
			\begin{lemma}\label{lem:dVboundary}
				Suppose the hypotheses of Proposition~\ref{prop:apriori} hold. Given $\eta>0$ (small), if $T>0$ is sufficiently small then for any $k\leq \ell$ and $t\in[0,T]$ 
				\begin{align}\label{eq:nablaVGamma2}
					\begin{split}
						\|\partial\partial_t^{j}\Theta\|_{L^2(\partial\Omega_0^T)}^2\lesssim \calE_\ell(0)+R_{j}(\calE_{k-1}(T))+\eta\calE_{k}(T),\qquad 0\leq j\leq k-1,
					\end{split}
				\end{align} 
				where the implicit constant is  independent of $C_1$, and $R_{j}$ is some polynomial function for each $j$.
			\end{lemma}
			\begin{proof}
				For small $j$ this follows for instance from the trace theorem and Proposition~\ref{prop:elliptic}. For large $j$, inductively we assume that we have proved the statement for $j\leq i-1\leq k-2$ and prove it for $j=i$. We apply Lemma~\ref{lem:nablaVGamma} with $u=\partial_t^{i}\Theta$. We will show that the desired control comes from the integral on $\partial\Omega$ on the left-hand side of \eqref{eq:nablaVGamma1}. The integrals on $\Omega_{T}$ and $\Omega_0$, as well as the last line on the right-hand side of \eqref{eq:nablaVGamma1} can be bounded by the right-hand side of \eqref{eq:nablaVGamma2} using the definition of $\calE$ (the integral on $\Omega_T$ in \eqref{eq:nablaVGamma1} can be seen to have a favorable sign to leading order, but we do not need to use this).   For the integral on $\partial\Omega$ on the left-hand side of \eqref{eq:nablaVGamma1} first note that the contribution of $\|\partial_t\partial_t^i\Theta\|_{L^2(\partial\Omega_0^T)}$ is of the desired form by the definition of $\calE$. For the contribution of $\|D_n\partial_t^i\Theta\|_{L^2(\partial\Omega_0^T)}$ we use the equation for $\partial_t^i\Theta$ to express $D_n\partial_t^i\Theta$ in terms of $\partial_t^2\partial_t^i\Theta$ and $f$, where $f$ is as in Lemma~\ref{lem:Vbdryhigh}. The contribution of  $\partial_t^2\partial_t^i\Theta$ is of the desired form by taking $T$ small. The contribution of $f$ is bounded using the induction hypothesis. The integral on $\partial\Omega$ on the left-hand side of \eqref{eq:nablaVGamma1} then gives the desired control on the remaining tangential derivatives of $\partial_t^i\Theta$.  
				Finally we need to consider the contribution of $F$ in \eqref{eq:nablaVGamma1}, where $F$ is as in Lemma~\ref{lem:Vinthigh} with $k$ replaced by $i$. But since $i\leq k-1$, in view of the structure of $F$ given in Lemma~\ref{lem:Vinthigh} the contributions of these terms can be seen to be of the desired form by using Proposition~\ref{prop:elliptic} and taking $T$ small.
			\end{proof}
			We now turn to the proof of Proposition~\ref{prop:apriori}.
			\begin{proof}[Proof of Proposition~\ref{prop:apriori}]
				We proceed inductively on $k$ to show that $\calE_k(T)\leq \calP_{k}(\calE_\ell(0))$. Using  Lemmas~\ref{lem:Venergy1},~\ref{lem:sigmaenergy},~\ref{lem:Wenergy}, and elliptic and Sobolev estimates to bound lower order terms in $L^\infty$, we see that $\calE_0(T)\leq \calP_{0}(\calE_\ell(0))$. Now we assume that $\calE_{j}(T)\leq \calP_{j}(\calE_\ell(0))$ for all $j\leq k-1$ and some $k\leq \ell$, and use this to prove that $\calE_k(T)\leq \calP_{k}(\calE_\ell(0))$.
				
				{\bf{Step 1:}} We show that given $\delta>0$ (small) if $T>0$ is sufficiently small then
				\begin{align}\label{eq:aprioriclaim1}
					\begin{split}
						\sup_{0\leq t<T}\|\partial \partial_t^{k+1}\sigma^2\|_{L^2(\Omega_t)}^2+\|\partial \partial_t^{k+1}\sigma^2\|_{L^2(\partial\Omega_0^T)}^2\leq \tilP_{k}(\calE_\ell(0))+\delta \calE_k(T).
					\end{split}
				\end{align}
				This follows from applying Lemma~\ref{lem:sigmaenergy} to the wave equation satisfied by $\partial_t^{k+1}\sigma^2$ as given in Lemma~\ref{lem:sigmahigh}. Using elliptic estimates from Proposition~\ref{prop:elliptic} and Sobolev to bound lower order terms in $L^\infty$, the contributions of all the error terms in Lemma~\ref{lem:sigmaenergy}, except the integral over $\Omega$ of $H Q\partial_t^{k+1}\sigma^2$, are of the desired form by taking $T$ small. For the contribution of  $H Q\partial_t^{k+1}\sigma^2$ we refer to Lemma~\ref{lem:sigmahigh} for the structure of $H$. Unless $H$ contains $\partial^2\partial_t^k\sigma^2$ or $\partial^2\partial_t^{k-1}\Theta$, in view of Proposition~\ref{prop:elliptic}, the corresponding contribution can be estimated using Cauchy-Schwarz, the induction hypothesis, and by taking $T$ small. The contribution of $\partial^2\partial_t^{k}\sigma^2$ can also be handled by the same argument, where now in addition we use Lemma~\ref{lem:d2sigma}. Finally for terms involving $\partial^2\partial_t^{k-1}\Theta$ we need to perform a few integration by parts. The details are similar to the corresponding argument in \cite{MSW1} (see Step 1 of proof of Proposition 2.1) so we carry this out schematically. Writing the corresponding term as $F\cdot\nabla^{(2)}\partial_t^{k-1}\Theta Q\partial_t^{k+1}\sigma^2$ with $F$ consisting of lower order terms, we have
				\begin{align*}
					\begin{split}
						F\cdot\nabla^{(2)}\partial_t^{k-1}\Theta Q\partial_t^{k+1}\sigma^2&=\nabla\cdot\big(F\cdot\nabla\partial_t^{k-1}\Theta Q\partial_t^{k+1}\sigma^2\big)-\big(\nabla\cdot(F Q\partial_t^{k+1}\sigma^2)\big)\cdot\nabla \partial_t^{k-1}\Theta\\
						&=\nabla\cdot\big(F\cdot\nabla\partial_t^{k-1}\Theta Q\partial_t^{k+1}\sigma^2\big)-\partial_t\big((\nabla Q\partial_t^k\sigma^2)F\cdot\nabla\partial_t^{k-1}\Theta\big)\\
						&\quad-Q\partial_t^{k+1}\sigma^2\nabla F\cdot\nabla\partial_t^{k-1}\Theta-[\nabla Q,\nabla_\hatV]\partial_t^{k}\sigma^2 F \cdot \nabla\partial_t^{k-1}\Theta\\
						&\quad+(\nabla Q\partial_t^k\sigma^2)\nabla_\hatV F\cdot\nabla\partial_t^{k-1}\Theta+(\nabla Q\partial_t^k\sigma^2)F\cdot\nabla\partial_t^{k}\Theta\\
						&\quad+(\nabla Q\partial_t^k\sigma^2)F\cdot[\nabla_\hatV,\nabla]\partial_t^{k-1}\Theta.
					\end{split}
				\end{align*}
				Except for the first line, the other terms can be handled by the same arguments as above and taking $T$ small. For the first line we integrate by parts. The resulting terms on $\Omega_T$ contain at most one top order term, so these can be handled using Cauchy-Schwarz with a small constant and using the induction hypothesis. Moreover, since $\partial_t$ is tangential to $\partial\Omega$ only the first term contributes a boundary term on $\partial\Omega$ and this term can be handled using Cauchy-Schwarz with a small constant and Lemma~\ref{lem:dVboundary}. This proves~\eqref{eq:aprioriclaim1}. 
				
				{\bf{Step 2:}} We show that given $\delta>0$ (small) if $T>0$ is sufficiently small then
				\begin{align*}
					\begin{split}
						\Big|\int_0^T\int_{\Omega_t}(\Box \partial_t^{k}\Theta)(\partial_t^{k+1}\Theta)\sqrt{|g|}\ud x\ud t\Big|\leq \hatP_{k}(\calE_\ell(0))+\delta \calE_{k}(T).
					\end{split}
				\end{align*}
				The proof is similar to the argument just produced to handle the contribution of $HQ\partial_t^{k+1}\sigma^2$ in Step 1 above, where now we use Lemma~\ref{lem:Vinthigh} for the structure of $\Box\partial_t^k\Theta$. Except for terms containing $\partial^2\partial_t^{k-1}\Theta$ all other terms can be bounded using Cauchy-Schwarz, the elliptic estimates from Proposition~\ref{prop:elliptic}, and by taking $T$ small. For terms involving $\partial^2\partial_t^{k-1}\Theta$ we argue as in the previous step using integration by parts and appealing to Lemma~\ref{lem:dVboundary} to estimate the boundary terms on $\partial\Omega$. The details are again similar to the corresponding argument in \cite{MSW1} (see Step 2 of proof of Proposition 2.1).
				
				{\bf{Step 3:}} We show that given $\delta>0$ (small) if $T>0$ is sufficiently small then
				\begin{align*}
					\begin{split}
						\sup_{0\leq t< T}\|\partial_t^{k}R\|_{L^2(\Sigma_t)}^{2}\leq \check{P}_{k}(\calE_\ell(0))+\delta\calE_{k}(T).
					\end{split}
				\end{align*}
				Note that since $R$ can be expressed algebraically in terms of $W$ and $\Ric_{\mu\nu}=V_\mu V_\nu+\frac{1}{2}g_{\mu\nu}$, it suffices to prove this estimate with $R$ replaced by $W$. The estimate for $W$ then follows from applying Lemma~\ref{lem:Wenergy} to the higher order equation for $W$ given in Lemma~\ref{lem:Whigh}, and taking $T$ small. Note that the term $D\partial_t^{j_6}W$ on the right-hand side can be bounded in terms of $\partial_t^{j_6+1}W$ using Proposition~\ref{prop:elliptic}. 
				
				{\bf{Step 4:}} We complete the proof of Proposition~\ref{prop:apriori} by applying Lemma~\ref{lem:Venergy1} to $u=\partial_t^k\Theta$. Except for the terms involving $F$ and $f$, all other terms on the right-hand side of \eqref{eq:Venergy1} are of the desired size by taking $T$ sufficiently small. The contribution of $F$ was already handled in Step 2 above. Finally in view of Lemma~\ref{lem:Vbdryhigh}, the contribution of $f$ can also be bounded using Cauchy-Schwarz, the induction hypothesis, and Proposition~\ref{prop:elliptic}. Note that here to bound  $\|\partial_t^{k-1}R\|_{L^2(\partial\Omega_0^T)}$ we use the trace theorem as well as the elliptic estimates from Proposition~\ref{prop:elliptic}. Summarizing, we have shown that given $\delta>0$, if $T$ is sufficiently small,
				\begin{align*}
					\begin{split}
						\sup_{0\leq t <T}\big(\|\partial\partial_t^k\Theta\|_{L^2(\Omega_t)}^2+\|\partial_t^{k+1}\Theta\|_{L^2(\partial\Omega_t)}^2\big)\leq P^\sharp(\calE_\ell(0))+\delta \calE_k(T).
					\end{split}
				\end{align*}
				The induction claim $\calE_k(T)\leq \calP_k(\calE_\ell(0))$ now follows by choosing $\delta$ sufficiently small in Steps 1 -- 4 sufficiently small, and this concludes the proof of Proposition~\ref{prop:apriori}.
			\end{proof}
			
\section{Linear theory and the iteration}\label{sec:iteration}
In this section we set up an iteration based on the a priori estimates established in the previous section, in Proposition~\ref{prop:apriori} and its proof. We start by recalling the linear existence theory for the equations we use for the iteration. The linear equations for the geometric quantities are of first order hyperbolic or transport types, and the non-standard results we use are for the system satisfied by the fluid equations. The linear theory for the latter was established in \cite{MSW1}, and we now briefly recall this.

Starting with the equation for $\Theta$, let us recall from \cite{MSW1} the weak formulation of the equation
\begin{align}\label{eq:Thetaweaktemp1}
	\begin{split}
		\Box \Theta = F,\mathrm{~in~} \Omega,\qquad (\partial_t^2+\gamma D_n)\Theta=f, \mathrm{~on~}\partial\Omega.
	\end{split}
\end{align}
We use  $\angles{\cdot}{\cdot}$ to denote the inner product in $L^2(\Omega_0)$ with respect to $\ud x$, and $\bangles{\cdot}{\cdot}$ to denote the inner product in $L^2(\partial \Omega_0)$ with respect to the induced Euclidean measure $\ud S$. The pairing between $(H^1(\Omega_0))^\ast$ and $H^1(\Omega_0)$ is denoted by $(\cdot,\cdot)$. We define the bounded linear map $\Phi:H^1(\Omega_0)\to(H^1(\Omega_0))^\ast$ by
\begin{align*}
	\begin{split}
		(\Phi(u),v):=\angles{-(g^{-1})^{tt}u}{v}+\bangles{\gamma^{-1}\tr\,u}{\tr\,v},\qquad v\in H^1(\Omega_0).
	\end{split}
\end{align*}
Here, and in what follows, $\tr$ denotes the Sobolev trace operator. The following bilinear forms will be used in the formulation of the weak problem (since $\Omega_0$ is diffeomorphic to a ball, here we have parameterized the initial domain using the usual polar coordiantes $(r,\theta)$):
\begin{align}\label{eq:defbilin}
	\begin{split}
		&B:H^1(\Omega_0)\times H^1(\Omega_0)\to \bbR,\qquad C:L^2(\Omega_0)\times H^1(\Omega_0)\to\bbR,\qquad D,E:L^2(\partial \Omega_0)\times H^1(\Omega_0)\to \bbR,\\
		&B(u,v):=\angles{{(g^{-1})}^{ij}\partial_iu}{\partial_jv} -\frac{1}{2}\angles{\partial_iu}{ v{(g^{-1})}^{i\alpha}\partial_\alpha\log |g|}-\angles{\partial_iu}{v\partial_t{(g^{-1})}^{ti}}+\angles{u}{v\partial_t^2(g^{-1})^{tt}},\\
		&C(u,v):=2\angles{u}{{(g^{-1})}^{ti} \partial_iv}-\frac{1}{2}\angles{u}{v {(g^{-1})}^{t\alpha}\partial_\alpha\log |g|}+\angles{u}{v\partial_i {(g^{-1})}^{ti}}+\angles{u}{v\partial_t(g^{-1})^{tt}},\\
		&D(u,v):=-\bangles{u}{{(g^{-1})}^{tr}\tr\,v}-2\bangles{u}{(\gamma^{-1})'\tr\,v},\\
		&E(u,v):=-\bangles{u}{(\gamma^{-1})''\tr\,v}.
	\end{split}
\end{align}
For any $\Theta:[0,T]\to H^1(\Omega_0)$ satisfying  
\begin{align}\label{eq:Thetaspaces1}
	\begin{split}
		&\Theta\in L^2([0,T];H^1(\Omega_0)),\quad \Theta'\in L^2([0,T],L^2(\Omega_0)),\quad (\tr\,\Theta)'\in L^2([0,T];L^2(\partial \Omega_0)),\\
		&\Phi(\Theta), \Phi(\Theta)', \Phi(\Theta)''\in L^2([0,T];(H^1(\Omega_0))^\ast),
	\end{split}
\end{align}
let
\begin{align}\label{eq:defcalL}
	\begin{split}
		\calL(\Theta,v):=B(\Theta,v)+C(\Theta',v)+D((\tr\,\Theta)',v)+E(\tr\,\Theta,v).
	\end{split}
\end{align}
The weak formulation of \eqref{eq:Thetaweaktemp1} is then given by
\begin{align}\label{eq:weak1}
	\begin{split}
		(\Phi(\Theta)'',v)+\calL(\Theta,v)=\bangles{\gamma^{-1}f}{\tr\,v}-\angles{F}{v},\qquad \forall v\in H^1(\Omega_0),
	\end{split}
\end{align}
for almost every $t\in [0,T].$  The initial conditions are 
\begin{align}\label{eq:weakdata1}
	\begin{split}
		&\Theta(0)=\theta_0\quad\mathrm{in~} L^2(\Omega_0),\\
		&(\Phi(\Theta)'(0),v)=\angles{\theta_1}{v}+\bangles{\tiltheta_1}{\tr\,v},\qquad \forall v\in H^1(\Omega_0),
	\end{split}
\end{align}
for given initial data
\begin{align*}
	\begin{split}
		\theta_0\in H^1(\Omega_0),\quad \theta_1\in L^2(\Omega_0), \quad \tiltheta_1\in L^2(\partial \Omega_0).
	\end{split}
\end{align*}
This formulation is obtained by multiplying the equation by a test function $\varphi$ and integrating with respect to $\ud x \ud t$. It is based on the energy estimates in Lemma~\ref{lem:Venergy1} which are the main ingredient in estimating $\Theta^I$.  We refer the reader to equations (3.6)--(3.11) in \cite{MSW1} for more details. We also need to consider the equations after commuting $\partial_{t}$ derivatives. For this purpose let $F_{0}=F$ and $f_{0}=\gamma^{-1}f$ denote the right-hand sides of the interior and boundary equations respectively. Similarly, let $F_{k}$ and $f_{k}$  denote the right-hand sides of these equations after commuting $\partial_{t}^k$, and let $\Theta_k$ be the $k$-times differentiated unknown\footnote{Note that as part of Proposition~\ref{prop:weak2} below we will show that indeed $\Theta_k=\partial_t^k\Theta$, so for now equation \eqref{eq:weak2} for $\Theta_k$ should be taken as the defining equation for $\Theta_k$.}. For $k\geq 1$, let the higher order source terms be defined by
\begin{align}\label{higher order source}
	\begin{split}
		F_{k}=&\partial_{t}^{k}F_{0}-\sum_{\ell=0}^{k-1}\partial_{t}^{\ell}\calC(\Theta_{k-\ell-1})-\sum_{\ell=0}^{k-2}(k-\ell-1)\partial_t^\ell\tiltilcalC(\Theta_{k-\ell-1}),\\
		\calF^{a}_{k}=&=-\sum_{\ell=0}^{k-1}\partial_{t}^{\ell}\calC^{a}(\Theta_{k-\ell-1}),\\
		f_{k}=&\partial_{t}^{k}f_{0}-\sum_{\ell=0}^{k-1}\partial_{t}^{\ell}\calC_{\calB}(\Theta_{k-\ell-1})-\sum_{\ell=0}^{k-2}(k-\ell-1)\partial_t^\ell\tiltilcalC_\calB(\Theta_{k-\ell-1}),
	\end{split}
\end{align}
where we have used the notation
\begin{align}\label{def:commutator weak}
	\begin{split}
		\angles{\calC(\Theta)}{v}&:=-\frac{1}{2}\angles{\partial_a\Theta}{ v\partial_{t}(g^{a\alpha}\partial_\alpha\log |g|)}-\angles{\partial_a\Theta}{v\partial^{2}_{t}g^{ta}}\\&\quad-\frac12\angles{\Theta'}{v\partial_{t}(g^{t\alpha}\partial_{\alpha}\log|g|)}+\angles{\Theta'}{v\partial_{t}\partial_{a}g^{ta}},\\
		\angles{\calC^{a}(\Theta)}{\partial_{a}v}&:=\angles{\partial_{t}g^{ab}\partial_{b}\Theta}{\partial_{a}v}+2\angles{\Theta'}{\left(\partial_{t}g^{ta}\right)\partial_{a}v},\\
		\angles{\tilcalC(\Theta)}{v}&:=\angles{\Theta'}{-v\partial_t (g^{-1})^{tt} },\\
		\angles{\tiltilcalC(\Theta)}{v}&:=\angles{\Theta'}{-v\partial_t^2(g^{-1})^{tt}},\\
		\bangles{\calC_{\calB}(\Theta)}{\tr\,v}&:=-\bangles{(\tr\,\Theta)'}{(\partial_{t}g^{tr})\tr\,v},\\
		\bangles{\tilcalC_{\calB}(\Theta)}{\tr\,v}&:=\bangles{(\tr\,\Theta)'}{(\gamma^{-1})'\tr\,v},\\
		\bangles{\tiltilcalC_\calB(\Theta)}{v}&:=\bangles{(\tr\,\Theta)'}{(\gamma^{-1})''\tr\,v}.
	\end{split}
\end{align}
The defining equation for $\Theta_k$ is then
\begin{align}\label{eq:weak2}
	\begin{split}
		&(\Phi(\Theta_{k})'',v)+\calL(\Theta_{k},v)+k\angles{\tilcalC(\Theta_k)}{v}+k\bangles{\tilcalC_\calB(\Theta_k)}{v}\\
		&=\angles{F_{k}}{v}+\angles{\calF^{a}_{k}}{\partial_{a}v}+\bangles{f_{k}}{\tr\,v},\qquad \forall v\in H^1(\Omega_0).
	\end{split}
\end{align}
This equation is derived by $k$ times differentiating the weak equation \eqref{eq:weak1}, under the assumption that $\Theta$ is sufficiently regular. The next result, which is the main linear theorem for the coupled system satisfied by $\Theta^I$, was proved\footnote{Even though \cite{MSW1} considered the case of the Minkowski background, since the linear theory was worked out in Lagrangian coordinates, the results derived there apply equally to the variable coefficient setting of the current work.} in \cite{MSW1}, in Propositions~3.2,~3.7 and~3.9. We will assume the following regularity assumptions on the coefficients, where $M$ is a sufficiently large integer:
\begin{align}\label{eq:g-assumption1}
	\begin{split}
		\partial^a\partial_t^k g\in L^\infty([0,T];L^2(\Omega_0)),\qquad &k\leq M+1,~ \begin{cases}2a\leq M+1-k,\quad &k\geq1\\ 2a\leq M,\quad &k=0\end{cases},\\
		\partial_t^k (\tr\,g)\in L^2([0,T];L^2(\partial \Omega_0)),\qquad &k\leq M,\\
		\partial_t^k \gamma^{-1},\partial_t^k\gamma\in L^2([0,T];L^2(\partial \Omega_0)),\qquad &k\leq M,\\
		\partial_t^k\gamma^{-1},\partial_t^k\gamma\in L^\infty([0,T];L^\infty(\partial \Omega_0)),\qquad &k\leq M-5,\\
		\partial_t^kf_{0}\in L^2([0,T];L^2(\partial \Omega_0)),\qquad&k\leq M,\\
		\partial_t^k F_{\sigma,0}\in L^{2}([0,T];L^2( \Omega_0)),\qquad&k\leq M.\\
	\end{split}
\end{align} 
Note that in our applications $g$ is related to the frame $\{e_I\}$ by the relation
\begin{align*}
	\begin{split}
		(g^{-1})^{\alpha\beta}= \sum_I \epsilon_I e_I^\alpha e_I^\beta.
	\end{split}
\end{align*}
\begin{proposition}\label{prop:weak2}
	Suppose \eqref{eq:g-assumption1} holds and that there are
	\begin{align*}
		\begin{split}
			\theta_k\in H^1(\Omega_0),\quad \theta_{k+1}\in L^2(\Omega_0), \quad \tiltheta_{k+1}\in L^2(\partial \Omega_0), \quad k=0,...,M
		\end{split}
	\end{align*}
	such that:
	\begin{itemize}
		\item  For $k=0,\dots,M-1$,
		\begin{align*}
			\begin{split}
				\angles{-(g^{-1})^{tt}\theta_{k+2}}{v}+\bangles{\tiltheta_{k+2}}{\gamma^{-1}\tr\,v}+\calL(\theta_{k},v)+k\bangles{\tilcalC_\calB(\theta_k)}{v}
				=\bangles{f_{k}(0)}{\tr\,v}+\angles{F_{k}(0)}{v}+\angles{\calF^{a}_{k}(0)}{\partial_av}.
			\end{split}
		\end{align*}
		\item$\tiltheta_k=\tr\,\theta_k$ for $k=1,\dots, M$.
	\end{itemize}
	Then there is a unique $\Theta_{k}$ satisfying \eqref{eq:Thetaspaces1} and \eqref{eq:weakdata1}, such that for all $v\in H^1(\Omega_0)$ equation \eqref{eq:weak2} holds for almost every $t\in[0,T]$. The solution $\Theta_k$ satisfies
	\begin{align}\label{eq:energy high}
		\begin{split}
			&\sup_{t\in [0,T]}\big(\|\Theta'_{k}\|_{L^2(\Omega_0)}+\|\Theta_{k}\|_{H^1(\Omega_0)}+\|\tr\,\Theta'_{k}\|_{L^2(\partial \Omega_0)}\big)\\
			&\leq C_1e^{C_2T}\Big(\|\theta_k\|_{H^1(\Omega_0)}+\|\theta_{k+1}\|_{L^2(\Omega_0)}+\|\tiltheta_{k+1}\|_{L^2(\partial \Omega_0)}+\|f_{k}\|_{L^2([0,T];L^2(\partial \Omega_0))}\\
			&\phantom{\leq C_1e^{C_2T}\Big(}+\|F_{k}\|_{L^2([0,T];L^2(\Omega_0))}+\|\calF^{a}_{k}\|_{L^{\infty}([0,T];L^2(\Omega_0))}\Big).
		\end{split}
	\end{align}
	Here $C_1$, $C_2$, and $C_3$ depend on the norms of $g,\tr\,g, \gamma, \gamma^{-1}$ appearing in \eqref{eq:g-assumption1}.  Moreover, we have $\Theta'_{k-1}=\Theta_{k}$ for $k=1,...,M$, and there are functions $P_k$ depending polynomially on their arguments such that for $k\leq M$ and $2a+k\leq M+2$ and 
	for $\tau\leq T$ 
	\begin{align}\label{eq:Sobolev-estimate1}
		&\|\partial^a\Theta_{k}(\tau)\|_{L^\infty([0,\tau];L^2(\Omega_0))}\leq P_k\bigg(\sup_{t\leq\tau}\sum_{\ell\leq 2a+k-2}(\|\nabla \Theta_{\ell}(t)\|_{L^{2}(\Omega_0)}+\| \Theta_{\ell+1}(t)\|_{L^2(\Omega_0)}+\| \Theta_{\ell+1}(t)\|_{L^2(\partial \Omega_0)}),\nonumber\\
		&\phantom{\|\partial^a\Theta_{k}(\tau)\|_{L^\infty([0,\tau];L^2(\Omega_0))}\leq P_k\bigg(}\|g\|_{L^\infty([0,\tau];H^{\max\{a-2,5\}}(\Omega_0))},\sum_{\ell\leq k}\|\partial_t^{\ell}f\|_{L^{\infty}([0,T];H^{a-\frac32}(\Omega_0))}\bigg).
	\end{align}
\end{proposition}
Before turning to the equation for $\partial_t\sigma^2$ we also record the analogue of Lemma~\ref{lem:nablaVGamma}, which is proved in Lemma~3.12 of \cite{MSW1}.
\begin{lemma}\label{lem:voblique}
	Under the assumptions of Proposition~\ref{prop:weak2}, and with the same notation, for any $k \leq M-1$
	\begin{align*}
		\begin{split}
			\|\nabla\Theta_k\|_{L^2([0,T];L^2(\partial \Omega_0))}^2&\lesssim \sum_{j\leq k}\Big(\|F_j\|_{L^2([0,T];L^2(\Omega_0))}^2+\|\calF_j\|_{L^2([0,T];L^2(\Omega_0))}^2+\|f_j\|_{L^2([0,T];L^2(\partial \Omega_0))}^2\Big)\\
			&\quad+\|\theta_k\|_{H^1(\Omega_0)}^2+\|\theta_{k+1}\|_{L^2(\Omega_0)}^2+\|\tiltheta_{k+1}\|_{L^2(\partial \Omega_0)}^2\\
			&\quad+\|(\tr\,\Theta_k)'\|_{L^2([0,T];L^2(\partial \Omega_0))}^2+\|(\tr\,\Theta_k)''\|_{L^2([0,T];L^2(\partial \Omega_0))}^2,
		\end{split}
	\end{align*}
	where the implicit constant depends only on $g$, $\gamma$, and their first three derivatives.
\end{lemma}
The discussion for the linearized equation for $\partial_t\sigma^2$ is similar. For this we use $\Lambda$ to denote the linearized unknown. Consider the same bilinear forms $B$ and $C$ as above but with restricted domains:
\begin{align*}
	\begin{split}
		B:H^1_0(B)\times H^1_0(\Omega_0)\to \bbR\qquad \mathrm{and} \qquad C:L^2(\Omega_0)\times H^1_0(\Omega_0)\to \bbR.
	\end{split}
\end{align*}
Let $\iota$ be the standard embedding of $H^1_0(\Omega_0)$ in $H^{-1}(\Omega_0):=(H^1_0(\Omega_0))^\ast$:
\begin{align*}
	\begin{split}
		(\iota(u),v):=\angles{u}{v},\qquad u\in H^1_0(\Omega_0),\quad v\in H^{1}_0(\Omega_0).
	\end{split}
\end{align*}
We will simply write $u$ for $\iota(u)$ from now on. For any $\Lambda$ with
\begin{align}\label{eq:Lambdaspaces1}
	\begin{split}
		&\Lambda\in L^2([0,T];H^1_0(\Omega_0)),\quad \Lambda'\in L^2([0,T],L^2(\Omega_0)),\quad\Lambda''\in L^2([0,T];H^{-1}(\Omega_0)),
	\end{split}
\end{align}
and $v\in H^1_0(\Omega_0)$ let
\begin{align*}
	\begin{split}
		\calL_\sigma(\Lambda,v):=B(\Lambda,v)+C(\Lambda',v).
	\end{split}
\end{align*}
The weak equation for $\Lambda$ is
\begin{align}\label{eq:weakLambda1}
	\begin{split}
		(-(g^{-1})^{tt}\Lambda'',v)+\calL_\sigma(\Lambda,v)=\angles{F_\sigma}{v},\qquad \forall v\in H^1_0(\Omega_0),
	\end{split}
\end{align}
where $F_\sigma\in L^2([0,T];L^2(\Omega_0))$ corresponds to the linearized source term. The initial conditions are
\begin{align}\label{eq:Lambdaid}
	\begin{split}
		\Lambda(0)=\lambda_0,\qquad \quad ((\Lambda)'(0),v)=\angles{\lambda_1}{v},\quad\forall v\in H^1_0(\Omega_0),
	\end{split}
\end{align}
where  $\lambda_0\in H^1_0(\Omega_0)$ and $\lambda_1\in L^2(\Omega_0)$ are given initial data. The time differentiated equations for the higher order unknown  $\Lambda_k$ are also derived as above.  Let
\begin{align}\label{eq:Lambdalinsource}
	\begin{split}
		F_{\sigma,0}&:=F_\sigma,\qquad \calF^{a}_{\sigma,0}:=0,\\
		F_{\sigma,k}=&\partial_{t}^{k}F_{\sigma,0}-\sum_{\ell=0}^{k-1}\partial_{t}^{\ell}\calC(\Lambda_{k-\ell-1})-\sum_{\ell=0}^{k-2}(k-\ell-1)\partial_t^\ell\tiltilcalC(\Lambda_{k-\ell-1}),\\
		\calF^{a}_{\sigma,k}=&-\sum_{\ell=0}^{k-1}\partial_{t}^{\ell}\calC^{a}(\Lambda_{k-\ell-1}),\\
	\end{split}
\end{align}
The defining equation for $\Lambda_k$ is then
\begin{align}\label{eq:Lambdalinweak2}
	\begin{split}
		(-(g^{-1})^{tt}\Lambda_k'',v)+\calL_\sigma(\Lambda_{k},v)=\angles{F_{\sigma,k}}{v}+\angles{\calF^{a}_{\sigma,k}}{\partial_{a}v},\qquad \forall v\in H^1_0(\Omega_0).
	\end{split}
\end{align}
The next result which is the analogue of Proposition~\ref{prop:weak2} for $\Lambda$ is proved in Lemma~3.10 and Propositions~3.11 in \cite{MSW1}.
\begin{proposition}\label{prop:weaksigma}
	Suppose \eqref{eq:g-assumption1} holds and that there are
	\begin{align*}
		\begin{split}
			\lambda_k\in H^1_0(\Omega_0),\quad \lambda_{k+1}\in L^2(\Omega_0),  \quad k=0,...,M
		\end{split}
	\end{align*}
	such that
	\begin{align*}
		\begin{split}
			\angles{-(g^{-1})^{tt}\lambda_{k+2}}{v}+\calL_\sigma(\lambda_{k},v)	
			=\angles{F_{\sigma,k}(0)}{v}+\angles{\calF^{a}_{\sigma,k}(0)}{\partial_av}.
		\end{split}
	\end{align*}
	Then there is a unique $\Lambda_{k}$ satisfying \eqref{eq:Lambdaspaces1} and \eqref{eq:Lambdaid}, such that for all $v\in H^1_0(\Omega_0)$ equation \eqref{eq:Lambdalinweak2} holds for almost every $t\in[0,T]$. The solution satisfies
	\begin{align}\label{eq:lambda-energy}
		\begin{split}
			&\sup_{t\in [0,T]}\big(\|\Lambda'_{k}\|_{L^2(\Omega_0)}+\|\Lambda_{k}\|_{H^1(\Omega_0)}\big)+\|\nabla \Lambda_k\|_{L^2([0,T];L^2(\partial \Omega_0))}^2\\
			&\leq C_1e^{C_2T}\Big(\|\lambda_k\|_{H^1(\Omega_0)}+\|\lambda_{k+1}\|_{L^2(\Omega_0)}+\|F_{\sigma,k}\|_{L^2([0,T];L^2(\Omega_0))}+\|\calF_{\sigma,k}\|_{L^{\infty}([0,T];L^2(\Omega_0))}\Big).
		\end{split}
	\end{align}
	Here the constants $C_1$, $C_2$, and $C_3$ depend on norms of $g$ appearing in \eqref{eq:g-assumption1}.  Moreover,  $\Lambda'_{k-1}=\Lambda_{k}$ for $k=1,...,M$, and there are functions $P_k$ depending polynomially on their arguments such that for $k\leq M$ and $2a+k\leq M+2$, and for $\tau\leq T$,  
	\begin{align}\label{eq:LambdaSobolev-estimate1}
		\begin{split}
			&\|\partial^a\Lambda_k(\tau)\|_{L^\infty([0,\tau];L^2(\Omega_0))}\leq P_k\bigg(\sup_{t\leq\tau}\sum_{\ell\leq 2a+k-2}(\|\nabla \Lambda_{\ell}(t)\|_{L^{2}(\Omega_0)}+\| \Lambda_{\ell+1}(t)\|_{L^2(\Omega_0)}),\\
			&\phantom{\|\partial^a\Lambda_k(\tau)\|_{L^\infty([0,\tau];L^2(\Omega_0))}\leq P_k\bigg(}\|g\|_{L^\infty([0,\tau];H^{\max\{a-2,5\}}(\Omega_0))},\sum_{\ell\leq k}\|\partial_{t}^{\ell}F_{\sigma}\|_{L^{2}([0,T];H^a(\Omega_0))}\bigg).
		\end{split}
	\end{align}
\end{proposition}
The last equation we use for the iteration is the Maxwell system \eqref{Maxwell-inhomo1}, \eqref{Maxwell-inhomo2}  derived from the Bianchi equations. During the iteration we decompose $F$ in a slightly modified frame compared with \eqref{eq:Maxwell1}--\eqref{eq:Maxwell2}, given as follows. $\hate_1$ is by definition $\che_0=\hatV=\partial_t$, while $\che_\tilI$, $\tilI=1,2,3$, are given by applying Gram-Schmidt to $\{e_1,e_2,e_3\}$, that is (here for a spacelike vector $v$ we use the notation $\|v\|=\sqrt{g(v,v)}$),
\begin{align}\label{eq:chedef1}
	\begin{split}
		&\che_3 = (e_3+g(e_3,\che_0)\che_0)/\|e_3+g(e_3,\che_0)\che_0 \|,\\
		&\che_2=(e_2+g(e_2,\che_0)\che_0-g(e_2,\che_3)\che_3)/\|e_2+g(e_2,\che_0)\che_0-g(e_2,\che_3)\che_3\|,\\
		&\che_1=(e_1+g(e_1,\che_0)\che_0-g(e_1,\che_3)\hate_3-g(e_1,\che_2)\che_2)/\|e_1+g(e_1,\che_0)\che_0-g(e_1,\che_3)\che_3-g(e_1,\hate_2)\che_2\|.
	\end{split}
\end{align}
By a slight abuse of notation we denote contractions with $\che_I$ with a check, so for instance we write (more precisely we should write $\chF^{AB}_{IJ}=F(\che_I,\che_J,X_A,X_B)$)
\begin{align*}
	\begin{split}
		\chF_{IJ}= F(\che_I,\che_J), \quad  \chGamma_{IJ}^K\che_K=\nabla_{\che_I}\che_J, \quad \chD_I= \che_I^\mu \partial_\mu,\quad \mathrm{etc}.
	\end{split}
\end{align*}
Note that since $g(e_J,\che_0)= \epsilon_J\hatTheta^J$, the frame $\{\che_I\}$ is a linear combination of $\{e_I\}$ with coefficients which depend algebraically on $\hatTheta$, and therefore have the same regularity properties. The electric and magnetic parts of $\chF$ are denoted by $\chE$ and $\chH$ and are given by
\begin{align*}
	\begin{split}
		\chE_I=\chF_{0I},\qquad \chH^I=-\frac{1}{2}\sum_{J,K=1}^3\epsilon^{IJK}\chF_{JK},\quad I=1,2,3.
	\end{split}
\end{align*}
For $\chE$ and $\chH$ we formally define the following first order differential operators (where $\chnabla = \chD\pm\chGamma$)
\begin{align}\label{def chdiv chcurl}
	\begin{split}
		&\chdiv \chE:=\sum_{\tilI=1}^{3}\chnabla_{\tilI}\chE_{\tilI},\quad \chdiv\chH:=\sum_{\tilI=1}^{3}\chnabla_{\tilI}\chH^{\tilI},\\ &(\chcurl\chE)_{\tilI}:=\sum_{\tilJ,\tilK=1}^{3}\epsilon_{\tilI\tilJ\tilK}\chnabla_{\tilJ}\chE_{\tilK},\quad (\chcurl\chH)^{\tilI}:=\sum_{\tilJ,\tilK=1}^{3}\epsilon_{\tilI\tilJ\tilK}\chnabla_{\tilJ}\chH^{\tilK}.
	\end{split}
\end{align}
Here $\epsilon_{\tilI\tilJ\tilK}=1$ if $(\tilI\tilJ\tilK)$ is an even permutation of $(123)$, $\epsilon_{\tilI\tilJ\tilK}=-1$ if $(\tilI\tilJ\tilK)$ is an odd permutation of $(123)$, and $\epsilon_{\tilI\tilJ\tilK}=0$ otherwise. Even though these are not geometric divergence and curl operators by a slight abuse of notation we continue to refer to them as such.
The reason we decompose $F$ using $\{\che_I\}$ rather than $\{e_I\}$ is that in proving elliptic estimates for $\chE$ and $\chH$ we need to use the divergence equations satisfied by them, which is derived by taking the divergence of the evolution equations. For this purpose, in order to show that the divergences of $\chE$ and $\chH$ are lower order during the iteration, we find it more convenient if $\che_0$ coincides with $\partial_t$. We turn to the details. The equations satisfied by $\chE$ and $\chH$ are
\begin{align}\label{eq:checkEH1}
	\begin{split}
		\partial_t\chE+\chcurl \chH= \chcalI,\qquad \partial_t\chH-\chcurl \chE=\chcalJ^\ast,
	\end{split}
\end{align}
where $\chcalI$ and $\chcalJ$ are defined as in \eqref{eq:calIint1}, \eqref{eq:calIext1}, \eqref{eq:calJ1} with $D$, $\Theta$, $\Gamma$, $F$, $X$, $R$ replaced by $\chD$, $\chTheta$, $\chGamma$, $\chF$, $\chX$, $\chR$ respectively. 
		In terms of $\chW=(\chE,\chH)$ and $\chK=(\chcalI,\chcalJ^\ast)$, equation \eqref{eq:checkEH1} becomes the first order symmetric hyperbolic system
		\begin{align}\label{eq:checkW1}
			\begin{split}
				\sum_{\nu=0}^3 \chcalB^\mu \partial_\mu \chW=\chcalK, \quad \chcalB^0:=1+\sum_{\tilI=1}^3 \che_\tilI^0\calA^\tilI, \quad \chcalB^j:=\sum_{\tilI=1}^3\che_\tilI^j \calA^\tilI,~j=1,2,3.
			\end{split}
		\end{align}
		We will use the following standard result for the iteration for $\chW$.
		\begin{proposition}\label{prop:chW1}
			Suppose $\che_\tilI$, $\tilI=1,2,3$, and $\chcalK$ satisfy the following conditions:
			\begin{align}\label{eq:che-assumption1}
				\begin{split}
					1-\big(\sum_{\tilI=1}^3(\che_\tilI^0)^2\big)^{\frac{1}{2}}\geq \kappa >0,\qquad&\\
					\partial_t^k \che_I\in L^\infty([0,T];L^2(\bbR^3)),\qquad &k\leq M,\\
					\partial^a\partial_t^k\che_I\in L^\infty([0,T];L^2(\Omega_0)\cap L^2(\Omega_0^c)),\qquad &2a+k\leq M,\\
					\partial_t^k\chcalK\in L^2([0,T]\times \bbR^3),\qquad &k\leq M,\\
					\partial_t^k\chcalK\in L^{\infty}([0,T];H^{a-1}(\Omega_0)\cap H^{a-1}(\Omega_0^c)),\qquad &2a+k\leq M,\\
					\partial_{t}^{\ell}(\chdiv\chcalK_\chE,\chdiv\chcalK_\chH)\in L^{\infty}([0,T];H^{a-1}(\Omega_0)\cap H^{a-1}(\Omega_0^c)),\qquad &2a+k\leq M.
				\end{split}
			\end{align}
			Let $\chcalB^\mu$ be defined as in \eqref{eq:checkW1}. Then there is a unique solution $\chW=(\chE,\chH)$ to 
			\begin{align}\label{eq:chWlin1}
				\begin{split}
					\sum_{\nu=0}^3 \chcalB^\mu \partial_\mu \chW=\chcalK=(\chcalK_\chE,\chcalK_\chH),
				\end{split}
			\end{align} 
			which in addition satisfies (below the term $\|\partial\partial_t^{k-1}\chW\|_{L^2(\bbR^3)}$ on the left is present only for $k\geq1$)
			\begin{align}\label{eq:chWlinbound1}
				\begin{split}
					&\sup_{t\in[0,T]}(\|\partial_t^k\chW(t)\|_{L^2(\bbR^3)}+\|\partial\partial_t^{k-1}\chW(t)\|_{L^2(\bbR^3)})\\
					&\leq C_1e^{C_2T}(\|\partial_t^k\chW(0
					)\|_{L^2(\bbR^3)}+\|\partial\partial_t^{k-1}\chW(0)\|_{L^2(\bbR^3)}+\|\partial_t^k\chcalK\|_{L^2([0,T]\times\bbR^3)}),
				\end{split}
			\end{align}
			for $k\leq M$. Moreover, there are functions $P_k$ depending polynomially on their arguments such that for $k\leq M$ and $2a+k\leq M$, and $\tau\leq T$, 
			\begin{align}\label{eq:chWlinbound2}
				\begin{split}
					&\sup_{t\in[0,\tau]}(\|\partial^a\partial_t^k\chW(t)\|_{L^2(\Omega_0)}+\|\partial^a\partial_t^k\chW(t)\|_{L^2(\Omega_0^c)})\\
					&\leq P_k\bigg(\sup_{t\leq\tau}(\sum_{\ell\leq 2a+k}\|\partial_t^\ell \chW(t)\|_{L^{2}(\bbR^3)}+\sum_{\ell\leq 2a+k-1}\|\partial\partial_t^\ell \chW(t)\|_{L^{2}(\bbR^3)}),\\
					&\phantom{P_k\bigg(}\quad\sum_{\ell\leq k}\|\partial_{t}^{\ell}\chcalK\|_{L^{\infty}([0,T];H^{a-1}(\Omega_0^c))},\,\sum_{\ell\leq k}\|\partial_{t}^{\ell}\chcalK\|_{L^{\infty}([0,T];H^{a-1}(\Omega_0^c))},\\
					&\phantom{P_k\bigg(}\quad\sum_{\ell\leq k}\|\partial_{t}^{\ell}(\chdiv\chcalK_\chE,\chdiv\chcalK_\chH)\|_{L^{\infty}([0,T];H^{a-1}(\Omega_0^c))},\,\sum_{\ell\leq k}\|\partial_{t}^{\ell}(\chdiv\chcalK_\chE,\chdiv\chcalK_\chH)\|_{L^{\infty}([0,T];H^{a-1}(\Omega_0^c))}\bigg).
				\end{split}
			\end{align}
			Here the constants in \eqref{eq:chWlinbound1} and \eqref{eq:chWlinbound2} depend on $\kappa$ and the norms of $\che_I$ appearing in~\eqref{eq:che-assumption1}.
		\end{proposition}
		\begin{proof}
			Since $\chcalB^j$, $j=1,2,3$, are symmetric, existence of a unique solution is standard once we show that $\chcalB^0$ is positive definite (with constant depending on $\kappa$). But this follows by observing that the eigenvalues of $\chcalB^0-I$ are $\lambda=0,\pm\sqrt{\sum_{\tilI=1}^3 (\che_\tilI^0)^2}$. The energy estimate~\eqref{eq:chWlinbound1} in the case $k=0$ is then also the standard energy estimate for the system \eqref{eq:chWlin1} as in Lemma~\ref{lem:Wenergy}. The higher order estimates in \eqref{eq:chWlinbound1} and \eqref{eq:chWlinbound2} are then proved by differentiating the equation and arguing as in Lemmas~\ref{prop: H1 div curl sys 1} and~\ref{prop: H1 div curl sys 2} and the proof of Proposition~\ref{prop:elliptic}. However, for this we first need to show that $\chE$ and $\chH$ satisfy appropriate elliptic equations, similar to \eqref{eq: div barE 3}, \eqref{eq: div H 2}, \eqref{eq: curl H 1}, and \eqref{eq: curl barE pre}, which in turn follow if we can prove the analogue of \eqref{Maxwell-inhomo1}, \eqref{Maxwell-inhomo2} (see the derivation of \eqref{eq: div barE 3}, \eqref{eq: div H 2}, \eqref{eq: curl H 1}, and \eqref{eq: curl barE pre}). For this we first write equation \eqref{eq:chWlin1} as (by definition)
			\begin{align*}
				\begin{split}
					\partial_t\chE+\chcurl \chH= \chcalK_\chE,\qquad \partial_t\chH-\chcurl \chE=\chcalK_\chH.
				\end{split}
			\end{align*}
			Applying $\chdiv$ to these equations and using the fact that $\chdiv\chcurl$ vanishes to leading order (in fact it can be shown to vanish completely, but since it will take extra effort to show this during the iteration, we do not use this stronger fact), we get
			\begin{align}\label{eq:chdivtransporttemp1}
				\begin{split}
					\partial_t\chdiv\chE= \chdiv\chcalK_\chE+[\partial_t,\chdiv]\chE+\calF_1(\chH,\chD \chH, \chGamma), \qquad \partial_t\chdiv\chH= \chdiv\chcalK_\chH+[\partial_t,\chdiv]\chH+\calF_2(\chE,\chD \chE, \chGamma),
				\end{split}
			\end{align}
			for appropriate functions $\calF_1$ and $\calF_2$. Next, let $\chF_{0I}=-\chF_{I0}=\chE_I$ and $\epsilon^{0IJK}\chF_{IJ}=\chH^K$, and define $F_{\mu\nu}=-F_{\nu\mu}$ by $F_{\mu\nu}=\cha_\mu^I\cha_\nu^J\chF_{IJ}$, where $\partial_\mu=\cha_\mu^I\che_I$. Let $\barE_i=F_{0i}$ and $\barH^k=-\frac{1}{2}\epsilon^{ijk}F_{ij}$. Then (recall that $\che_0=\partial_t$)
			\begin{align*}
				\begin{split}
					\chdiv\chE&=\sum_{\tilI=1}^3 (\chD_{\tilI} \chE_{\tilI}-\chGamma_{\tilI\tilI}^J\chE_J)=\sum_{\tilI,j=1}^3\sum_{\mu=0}^3\che^\mu_{\tilI}\partial_\mu(\che_{\tilI}^jF_{0j})-\sum_{\tilI=1}^3\chGamma_{\tilI\tilI}^J\chE_J\\
					&=\sum_{j,i=1}^3\chgamma^{ij}\partial_{i}\barE_{j}+\sum_{\tilI,j,i=1}^3\che_{\tilI}^i\barE_{j}\partial_i\che_{\tilI}^j+\sum_{\tilI=1}^3\che_{\tilI}^0\partial_t\chE_{\tilI}-\sum_{\tilI=1}^3\chGamma_{\tilI\tilI}^J\chE_J,
				\end{split}
			\end{align*}
			where we have defined $\chgamma^{ij}=\sum_{\tilI=1}^3\che_\tilI^i\che_\tilI^j$. Rearranging we get
			\begin{align}\label{eq:chdivtemp1}
				\begin{split}
					\sum_{j,i=1}^3\chgamma^{ij}\partial_{i}\barE_j=\chdiv\chE-\sum_{\tilI,j,i=1}^3\che_\tilI^i\barE_j\partial_i\che_\tilI^j-\sum_{\tilI=1}^3\che_\tilI^0\partial_t\chE_\tilI+\sum_{\tilI=1}^3\chGamma_{\tilI\tilI}^J\chE_J.
				\end{split}
			\end{align}
			To derive a divergence equation for $\barH$ note that for $\tilI,\tilJ,\tilK\in\{1,2,3\}$ we have $$\chD_{[\tilI}\chF_{\tilJ\tilK]}=-\chdiv\chH+\calF(\chGamma,\chF),$$ for an appropriate function $\calF$, and $$\chD_{[0}\chF_{IJ]}=\frac{1}{2}\epsilon_{0IJ K}\chcalK_\chH^K+\calF(\chGamma,\chF)$$ for a different function $\calF$. It follows that (here $i,j,k$ take values in $\{1,2,3\}$, while $I,J,K$ take values in $\{0,1,2,3\}$)
			\begin{align}\label{eq:chdivtemp2}
				\begin{split}
					\partial_k\barH^k&=-\frac{1}{2}\epsilon^{ijk}\partial_k F_{ij}=-\frac{1}{2}\epsilon^{ijk}\cha_k^K\chD_K(\cha_i^I\cha_j^J\chF_{IJ})=-\frac{1}{6}\epsilon^{ijk}\cha_i^I\cha_j^J\cha_k^K\chD_{[K}\chF_{IJ]}-\frac{1}{2}\epsilon^{ijk}\chF_{IJ}\partial_k(\cha_i^I\cha_j^J)\\
					&=\calF(\cha,\chdiv\chH, \chcalK_\chH)-\frac{1}{2}\epsilon^{ijk}\chF_{IJ}\partial_k(\cha_i^I\cha_j^J),
				\end{split}
			\end{align}
			for some function $\calF$. To derive curl equations for $\barH$ and $\barE$, note that using $\sum_{\tilI=1}^{3}\chD_\tilI\chF_{\tilI0}=\chdiv\chE$ and $\sum_{I=0}^{3}\epsilon_I\chD_I\chF_{IJ}=-\chcalK_\chE^J$, for $J=1,2,3$, we get
			\begin{align}\label{eq:chcurltemp1}
				\begin{split}
					\epsilon_J\che_J^\mu\che_J^\nu\partial_\nu F_{\mu i}=\epsilon_J\cha_i^I\chD_J\chF_{JI}+\epsilon_J\che_J^\mu\che_J^\nu\chF_{NI}\partial_\mu(\cha_\nu^N\cha_i^I)=\calG(\che, \chdiv\chE,\chcalK_\chE)+\epsilon_J\che_J^\mu\che_J^\nu\chF_{NI}\partial_\mu(\cha_\nu^N\cha_i^I),
				\end{split}
			\end{align}
			for some function $\calG$, while using $\chD_{[0}\chF_{IJ]}=\frac{1}{2}\epsilon_{0IJK}\chcalK_\chH^K+\calF(\chGamma, \chF)$,
			\begin{align}\label{eq:chcurltemp2}
					\begin{split}
						\partial_{[0}\chF_{ij]}=\frac{1}{2}\epsilon_{0IJK}\chcalK_\chH^K\cha_i^I\cha_j^J+\chF_{IJ}\partial_t(\cha_i^I\cha_j^J)+\chF_{0I}\partial_j\cha_i^I+\chF_{J0}\partial_i\cha^J_j+\cha_{i}^{I}\cha_{j}^{J}\calF(\chGamma,\chF).
					\end{split}
			\end{align}
			Equations~\eqref{eq:chcurltemp1} and~\eqref{eq:chcurltemp2} can be used exactly as in the derivation of~\eqref{eq: curl H 1} and~\eqref{eq: curl barE pre} to derive analogous curl equations for $\chE$ and $\chH$. Combining the resulting equations with the divergence equations~\eqref{eq:chdivtemp1} and~\eqref{eq:chdivtemp2}, we can then use Lemmas~\ref{prop: H1 div curl sys 1} and~\ref{prop: H1 div curl sys 2} and the standard energy estimate of Lemma~\ref{lem:Wenergy} to prove~\eqref{eq:chWlinbound2} and the cases $k\geq1$ in~\eqref{eq:chWlinbound1} by the same argument as in the proof of Proposition~\ref{prop:elliptic}. Here to estimate $\chdiv\chE$ and $\chdiv\chH$ on the right-hand side of \eqref{eq:chdivtemp1},~\eqref{eq:chdivtemp2},~\eqref{eq:chcurltemp1},~\eqref{eq:chcurltemp2}, we couple the elliptic system \eqref{eq:chdivtemp1}-\eqref{eq:chcurltemp2} with the transport equations \eqref{eq:chdivtransporttemp1}.
		\end{proof}
		The a priori estimates in Proposition~\ref{prop:apriori} and the existence results in Propositions~\ref{prop:weak2},~\ref{prop:weaksigma}, and~\ref{prop:chW1} form the basis of our  proof of Theorem~\ref{thm:main1}, which is by standard Picard iteration modeled on the proof of Theorem~1.1 in \cite[Section~4]{MSW1}.
		\begin{proof}[Proof of Theorem~\ref{thm:main1}]
			We start by setting up an iteration for the system defined by \eqref{eq:e0mu1}, ~\eqref{eq:etransport1}, \eqref{eq:Gamma01}, \eqref{eq:Gammatransport1}, \eqref{eq:g1}, \eqref{eq:BoxTheta1}, \eqref{eq:Thetabdry1}, \eqref{eq:Lambda1}, \eqref{eq:checkW1}, and showing the existence of a solution to the system \eqref{eq:geomit1}, \eqref{eq:Thetait1}, \eqref{eq:Lambdait1} defined below.  Uniqueness follows by similar estimates. After closing the iteration, we will describe how to go back and show that our solution solves the original equations \eqref{eq:relEuler4} and \eqref{eq:Einstein2}.
			
			The derived system for the iteration is described as follows. The geometric quantities will be defined using the equations
			\begin{align}\label{eq:geomit1}
				\begin{cases}
					&e_0^\mu=\frac{1}{\hatTheta^0}(\delta_0^\mu-\hatTheta^{\tilI}e_{\tilI}^\mu),\\
					&\partial_te_\tilI^\mu=-e_J^\mu D_\tilI\hatTheta^J-\hatTheta^J\Gamma_{\tilI J}^Ke_K^\mu,\\
					&\Gamma_{0J}^K=-\frac{\hatTheta^\tilI}{\hatTheta^0}\Gamma_{\tilI J}^K,\\
					&\partial_t\Gamma_{\tilI J}^K=\hatTheta^I\big(R^K_{\phantom{K}JI\tilI}+\Gamma_{MJ}^K(\Gamma_{I\tilI}^M-\Gamma_{I\tilI}^M)+\Gamma^K_{\tilI M}\Gamma^M_{IJ}-\Gamma^K_{IM}\Gamma^M_{\tilI J}\big)-\Gamma_{I J}^KD_\tilI\hatTheta^I,\\
					&\sum_{\nu=0}^3 \chcalB^\mu \partial_\mu \chW_{AB}=\chcalK_{AB},
				\end{cases}
			\end{align}
			where $\chcalK$ is defined as in \eqref{eq:checkW1} (which is in turn defined as in \eqref{eq:calIint1}, \eqref{eq:calIext1}, \eqref{eq:calJ1} using the frame $\{\che_I\}$ instead of $\{e_I\}$) and the fluid quantities using the equations 
			\begin{align}\label{eq:Thetait1}
				\begin{cases}
					\Box \Theta^I=(\frac{1}{2}-\sigma^2)\Theta^I+\epsilon_I\epsilon_J\epsilon_K\big(2\Gamma_{KI}^JD_K\Theta^J+2\Gamma_{KL}^J\Gamma_{KI}^J\Theta^L+\epsilon_J\Theta^J D_K\Gamma_{KI}^J+\epsilon_J\Theta^J\Gamma_{KI}^L\Gamma_{KL}^J),\qquad& \mathrm{in~}\Omega\\
					(\partial_t^2+\gamma D_n)\Theta^I=\frac{\epsilon_I\epsilon_K}{2\sigma^2}\Gamma_{KI}^J\Theta^JD_K\sigma^2-\frac{1}{2\sigma^2}D_I(\sqrt{\sigma^2}\partial_t\sigma^2)-\frac{1}{2\sigma^2}\partial_t\sigma^2\partial_t\Theta^I,\qquad&\mathrm{on~}\partial\Omega
				\end{cases},
			\end{align}
			for the $\Theta^I$, and 
			\begin{align}\label{eq:Lambdait1}
				\begin{cases}
					\Box \partial_t\sigma^2=2\partial_t\sigma^2-6\sigma^2\partial_t\sigma^2+\frac{6}{\sqrt{\sigma^2}}\epsilon_I\epsilon_J(D_I\Theta^J+\Gamma_{IK}^J\Theta^K)(D_ID_J\sigma^2+\epsilon_K\epsilon_J\Gamma_{IK}^J D_K\sigma^2)\\
					\phantom{\Box \partial_t\sigma^2=}+\frac{4}{\sqrt{\sigma^2}}\epsilon_I\epsilon_J\epsilon_K(D_I\Theta^J+\Gamma_{IM}^J\Theta^M)(D_I\Theta^J+\Gamma_{IN}^K\Theta^N)(D_K\Theta^J+\Gamma_{KP}^J\Theta^P)\\
					\phantom{\Box \partial_t\sigma^2=}-\frac{4}{\sqrt{\sigma^2}} \epsilon_I\epsilon_J\epsilon_K\epsilon_LR_{L I J K}(D_I\Theta^J+\Gamma_{IM}^J\Theta^M)(D_K\Theta^L+\Gamma_{KN}^L\Theta^N)+\frac{\epsilon_I}{\sigma^2} D_I\sigma^2 D_I\partial_t\sigma^2\\
					\phantom{\Box \partial_t\sigma^2=}-\frac{\epsilon_I\partial_t\sigma^2}{4\sigma^2}D_I\sigma^2 D_I\sigma^2-\frac{\partial_t\sigma^2}{2}-\sigma^2\partial_t\sigma^2-\frac{\partial_t\sigma^2}{\sigma^2}\epsilon_I\epsilon_J(D_I\Theta^J+\Gamma_{IK}^J\Theta^K)^2),~&\mathrm{in~}\Omega\\
					\partial_t\sigma^2=0,~ &\mathrm{on~}\partial\Omega
				\end{cases},
			\end{align}
			for $\partial_t\sigma^2$. Note that in \eqref{eq:geomit1} the first and third equations are algebraic definitions, while the second and fourth equations are transport equations. In the fluid equations $\gamma$ is defined as
			\begin{align*}
				\begin{split}
					\gamma=\frac{\sqrt{\epsilon_I(D_I\sigma^2)^2}}{2\sigma^2},
				\end{split}
			\end{align*}
			and $\Box$ is defined with respect to the metric $g$ defined by the relations
			\begin{align*}
				\begin{split}
					m_{IJ}= g_{\mu\nu}e_{I}^\mu e_J^\nu,\qquad (g^{-1})^{\mu\nu}=(m^{-1})^{IJ}e_I^\mu e_J^\nu.
				\end{split}
			\end{align*}
			The modified frame $\{\che_I\}$ is defined by the relations $\che_0=\partial_t$ and \eqref{eq:chedef1}, and we have set (the matrices $\calA^\tilI$ are as in \eqref{eq:Amatdef1})
			\begin{align*}
				\begin{split}
					\chcalB^0:=1+\sum_{\tilI=1}^3 \che_\tilI^0\calA^\tilI, \quad \chcalB^j:=\sum_{\tilI=1}^3\che_\tilI^j \calA^\tilI,~j=1,2,3.
				\end{split}
			\end{align*}
			The curvature $R$ is defined in terms of $\chW\equiv\chW_{AB}$ as follows. First, for fixed $A,B$, we write $\chW=(\chE,\chH)$ and for $\tilI=1,2,3$, define 
			\begin{align*}
				\begin{split}
					\chF_{0\tilI}=-\chF_{\tilI0}= \chE_\tilI,\quad-\chF_{12}=\chF_{21}=\chH^3,\quad -\chF_{23}=\chF_{32}=\chH^1,\quad \chF_{13}=-\chF_{31}=\chH^2.
				\end{split}
			\end{align*}
			To define $R$ in terms of $\chF$, for each $A,B,=0,1,2$ we define $\chX_A$ as in \eqref{eq:XAI1}, but in terms of $\{\che_I\}$ rather than $\{e_I\}$, that is,
			\begin{align*}
				\begin{split}
					\chX_A= \epsilon_I \chX_A^I \che_I,\quad \chX_A^I:=\frac{\delta_A^I-(\sum_K\epsilon_K(\chD_K\sigma^2)^{2})^{-1}\chD_A\sigma^2 \chD_I\sigma^2}{\sqrt{1-\epsilon_A(\sum_J \epsilon_J (\chD_J\sigma^2)^2)^{-1}(\chD_A\sigma^2)^2}},
				\end{split}
			\end{align*}
			and let
			\begin{align*}
				\begin{split}
					R(\che_I,\che_J,\chX_A,\chX_B)=-R(\che_I,\che_J,\chX_B,\chX_A)= \chF_{IJ}^{AB}.
				\end{split}
			\end{align*}
			By linearity, this defines $R(X,Y,X_A,X_B)$ for any $X,Y$. The components $R(\chX_A,\chX_B,\chX_C,n)$ are then defined by the requirement
			\begin{align*}
				\begin{split}
					R(\chX_A,\chX_B,\chX_C,n)=-R(\chX_A,\chX_B,n,\chX_C)=R(\chX_C,n,\chX_A,\chX_B),
				\end{split}
			\end{align*}
			and the components $R(\chX_A,n,\chX_B,n)$, $A\leq B$, by the requirement that
			\begin{align*}
				\begin{split}
					R(\chX_A,n,\chX_B,n)&=-R(n,\chX_A,\chX_B,n)=\chi_{\Omega}\left(\frac{1}{2}\delta_{AB}+g(\Theta^Ie_I,\chX_A)g(\Theta^Ie_I,\chX_B)\right)\\
					&\quad-\sum_{C=0}^2\epsilon_CR(\chX_A,\chX_C,\chX_B,\chX_C),
				\end{split}
			\end{align*}
			and finally the components $R(\chX_{B},n,\chX_{A},n)$, $A<B$, by the requirement that
				\begin{align*}
					R(\chX_{B},n,\chX_{A},n)=R(\chX_{A},n,\chX_{B},n).
				\end{align*}
			By linearity, this now defines $R(X,Y,Z,W)$ for any vectors $X,Y,W,Z$, and in particular $$R_{IJKL}=R(e_I,e_J,e_K,e_L).$$ 
			
			We denote by $\Theta^{(m)}$ (more precisely $\Theta^{I,(m)}$), $\Lambda^{(m)}$, and $\Sigma^{(m)}$ the iterates of $\Theta$ (more precisely $\Theta^I$), $\partial_t\sigma^2$, and $\sigma^2$ respectively. The iterates of the geometric quantities (frame, metric, connection coefficients, and curvature) are denotes by $e^{(m)}_I$, $g^{(m)}$, $\Gamma_{IJ}^{K,(m)}$, $\chW^{(m)}$ (more precisely $\chW_{AB}^{(m)}$). The zeroth iterates are defined simply by extending the initial data. Inductively, we define the $m$th iterate using equations \eqref{eq:geomit1}, \eqref{eq:Thetait1}, and \eqref{eq:Lambdait1}, where on the left-hand side we use the $m$th iterate of the unknowns, and on the right-hand side, as well as in the coefficients of $\Box$, $\gamma$, and $\chcalB$, we use the $m-1$ iterate. Here $\Sigma^{(m)}$ (corresponding to $\sigma^2$) is defined by the relation $\partial_t\Sigma^{(m)}=\Lambda^{(m)}$, and $\chcalB^{(m)}$, $g^{(m)}$, $R^{(m)}$, $\gamma^{(m)}$ are defined algebraically in terms of $\Theta^{(m)}, \Lambda^{(m)}, \Sigma^{(m)}, \Gamma^{(m)}, \chW^{(m)}$ as described above. Our first task is to show that the iterates have uniformly bounded energy. For this we define
			\begin{align*}
				\begin{split}
					\calE_k^{m}(T)&=\sup_{0\leq t \leq T}\sum_{\ell\leq k}(\|\partial\partial_t^\ell\Theta^{(m)}(t)\|_{L^2(\Omega_0)}^2+\|\partial\partial_t^\ell\Lambda^{(m)}(t)\|_{L^2(\Omega_0)}^2+\|\partial_t^\ell R^{(m)}\|_{L^2(\Omega_0)}^2+\|\partial_t^{\ell+1}\Theta^{(m)}(t)\|_{L^2(\partial\Omega_0)}^2)\\
					&\quad+ \int_0^T\|\partial\partial_t^\ell\Lambda^{(m)}\|_{L^2(\Omega_0)}^2\ud t.
				\end{split}
			\end{align*}
			We claim that if $T$ is sufficiently small then there are constants $A_k$, $k=0,\dots,M$ such that for all $m$
			\begin{align}\label{eq:claimittemp1}
				\begin{split}
					\calE_k^m(T)\leq A_k.
				\end{split}
			\end{align}
			In view of Propositions~\ref{prop:weak2},~\ref{prop:weaksigma}, and~\ref{prop:chW1}, the proof of this claim is the same as in Proposition~\ref{prop:apriori}. See also \cite[Proof of Theorem 1.1]{MSW1} for a similar argument in the case of the Minkowski background. Compared with the argument in \cite{MSW1}, the difference is that we now also need to show that the hypotheses of Proposition~\ref{prop:chW1} are satisfied. The subtle point here is that we need to make sure that $\chdiv\chcalK^{(m)}_{\chE^{(m)}}$  and $\chdiv\chcalK^{(m)}_{\chH^{(m)}}$ have the right regularity. Comparing with Propositions~\ref{prop:elliptic} and~\ref{prop:apriori} (or the linear estimates in Propositions~\ref{prop:weak2},~\ref{prop:weaksigma}, and~\ref{prop:chW1}), we see that the difficulty would be if two derivatives fall on $\chX^{(m)}_A$ after applying $\chdiv$ to the last equation in \eqref{eq:geomit1} (written in terms of $\chE^{(m)},\chH^{(m)}$ as in \eqref{eq:checkEH1}), so we need to show that no such terms appear. Comparing with \eqref{eq:calIint1} and \eqref{eq:calIext1}, the term to be considered for $\chdiv\chcalK_{\chE}$ is (here to simplify notation we drop the index $(m)$ denoting the iterate)
			\begin{align*}
				\begin{split}
					\epsilon_I\epsilon_J\epsilon_K\chR_{JKLI}(\chX_A^I\chD_K\chD_J\chX_B^L+\chX_B^I\chD_K\chD_J\chX_A^L).
				\end{split}
			\end{align*}
			But since, by construction, $\chR_{JKLI}=-\chR_{KJLI}$, this can be written as
			\begin{align*}
				\begin{split}
					\frac{1}{2}\epsilon_I\epsilon_J\epsilon_K\chR_{JKLI}(\chX_A^I[\chD_K,\chD_J]\chX_B^L+\chX_B^I[\chD_K,\chD_J]\chX_A^L),
				\end{split}
			\end{align*}
			which contains only one derivative of $\chX_A,\chX_B$. Similarly, for $\chdiv\chcalK_\chH$, comparing with \eqref{eq:calJ1} we need to consider
			\begin{align*}
					\begin{split}
						\epsilon^{LIJK}\chD_L(\chX_B^N\chR_{MN[JK}\chD_{I]}\chX_A^M)+\epsilon^{LIJK}\chD_L(\chX_A^N\chR_{MN[JK}\chD_{I]}\chX_B^M).
					\end{split}
			\end{align*}
			Again, using the antisymmetry of $\epsilon^{LIJK}$ we can replace every appearance of two derivatives on $X_A$, $X_B$ by a commutator. For instance, since $\epsilon^{LIJK}=-\epsilon^{KIJL}$,
			\begin{align*}
					\begin{split}
						\epsilon^{LIJK}\chX_B^N\chR_{MNIJ}\chD_L\chD_{K}\chX_A^M=\frac{1}{2}\epsilon^{LIJK}\chX_B^N\chR_{MNIJ}[\chD_L,\chD_{K}]\chX_A^M.
					\end{split}
			\end{align*}
			With these observations, we can now apply Propositions~\ref{prop:weak2},~\ref{prop:weaksigma},~\ref{prop:chW1},  the elliptic estimates \eqref{eq:Sobolev-estimate1}, \eqref{eq:lambda-energy}, \eqref{eq:chWlinbound2} contained in them (to bound lower order terms in $L^\infty$), and Lemma~\ref{lem:voblique} (to bound $\partial\partial_t^k\Theta$ in $L^2([0,T]\times\partial\Omega_0)$ as in the proof of Proposition~\ref{prop:apriori}), to inductively deduce \eqref{eq:claimittemp1}. Once the uniform bound \eqref{eq:claimittemp1} is available, we can consider the equations for the difference of the iterates $\Theta^{(m+1)}-\Theta^{(m)}$, $\Lambda^{(m+1)}-\Lambda^{(m)}$, $e_I^{(m+1)}-e_I^{(m)}$, $\Gamma^{(m+1)}-\Gamma^{(m)}$, and $\chW^{(m+1)}-\chW^{(m)}$ (which have zero initial data) to show the convergence of the iterates. The routine details are similar to the proof of \cite[Theorem~1.1]{MSW1}.
			
			Finally, we need to show that our derived solution satisfied the original equations \eqref{eq:relEuler4} and \eqref{eq:Einstein2}. For this we will derive homogeneous equations, of the forms already appearing in our a priori estimates, for geometric and fluid quantities which are expected to vanish. We refer to such quantities, to be written out below, as vanishing quantities. Having closed the iteration we now use the frame $\{e_A\}$ for the geometric quantities, but the calculations can be carried out the same way in $\{\che_I\}$ and the differences in the equations are lower order. The vanishing fluid quantities in our calculation are
			\begin{align}\label{def vanishing fluid}
				\begin{split}
					&\tilDelta:=\nabla_{I}\Theta^{I}\\& \omega_{IJ}:=\nabla_{I}\Theta_{J}-\nabla_{J}\Theta_{I},\\& S:=\Sigma+\Theta^{I}\Theta_{I}\\
					&X_{I}:=\Theta^{J}\nabla_{J}\Theta_{I}+\frac12\nabla_{I}\Sigma,\\ &Y_{I}:=\Theta^{J}\nabla_{J}\left(\Theta^{K}\nabla_{K}\Theta_{I}\right)-\frac12\left(\nabla^{J}\Sigma\right)\nabla_{J}\Theta_{I}+\frac12\nabla_{I}\Lambda,\\
					&A:=\Box\Sigma+2(\nabla^{I}\Theta^{J})(\nabla_{I}\Theta_{J})+(2\Sigma-1)\Sigma.
				\end{split}
			\end{align}
			These are expected to vanish in $\Omega$, with the vanishing of the first two quantities being precisely \eqref{eq:relEuler4}. The vanishing of the other quantities was used in arriving at our derived system \eqref{eq:collected1}, so it is natural that we need to treat them as unknowns at this point. Also note that we already know the vanishing of $Y_I$ on the boundary. Similarly, the vanishing geometric quantities whose vanishing is not already guaranteed by our construction are
			\begin{align}\label{def vanishing geom}
				\begin{split}
					& Q_{IJ}^{K}e_K:=-\Gamma_{IJ}^{K}e_K+\Gamma_{JI}^{K}e_k+[e_{I},e_{J}],\\
					& B_{IJKL}:=R_{[IJK]L},\\
					& \tilY_{ABCD}:=R_{ABCD}-R_{CDAB},\\
					&\Delta_{I}^{An}:=\nabla^{J}R_{JIAn}-\chi_{\Omega}\left(\Theta_{n}\nabla_{A}\Theta_{I}-\Theta_{A}\nabla_{n}\Theta_{I}\right),\\& Z^{An}_{IJK}:=\nabla_{[I}R_{JK]An},\\
					&\Delta_{I}^{AB}:=\nabla^{J}R_{JIAB}-\chi_{\Omega}\left(\Theta_{B}\nabla_{A}\Theta_{I}-\Theta_{A}\nabla_{B}\Theta_{I}\right),\\
					& Z^{An}_{IJK}:=\nabla_{[I}R_{JK]An},\\
					& Z^{AB}_{IJK}:=\nabla_{[I}R_{JK]AB},\\
					&\tilR_{AB}:=R_{AB}-\chi_{\Omega}\left(\Theta_{A}\Theta_{B}+\frac12g_{AB}\right),\quad \textrm{for}\quad A>B,\\
					&\tilR_{An}:=R_{An},\\
					&\tilR_{nn}:=R_{nn}-\frac12\chi_{\Omega}g_{nn},\\
					&\tilS^{K}{}_{M}{}_{IJ}:=R^{K}{}_{M}{}_{IJ}-\left(D_{I}\Gamma_{JM}^{K}-D_{J}\Gamma_{IM}^{K}\right)-\left(\Gamma_{IL}^{K}\Gamma_{JM}^{L}-\Gamma_{JL}^{K}\Gamma_{IM}^{L}\right)-\left(\Gamma_{JI}^{L}-\Gamma_{IJ}^{L}\right)\Gamma_{LM}^{K}.
				\end{split}
			\end{align} 
			Besides the Einstein equations \eqref{eq:Einstein2}, these vanishing quantities capture the symmetries of the curvature tensor, the differential Bianchi identities for the entire curvature tensor (not just when the last two components are tangential), the torsion free property of the connection, and the Ricci identity. As with the fluid quantities, the vanishing of these quantities was used in arriving at our derived system \eqref{eq:collected1}. We also allow arbitrary components instead of only tangential ones for these quantities by taking appropriate linear combinations, for instance to write $\tilY_{KMIJ}=R_{KMIJ}-R_{IJKM}$. When the precise choice of components is not important we simply write the defining letter without indices, for instance $S$, $Z$, or $\tilS$. We also use $Z^{An}$, $Z^{AB}$, $\Delta^{An}$, $\Delta^{AB}$ when the precise choice of lower components is not important. Note that using the definition of the curvature in our iteration, it can be checked directly that $\tilY$ and all components of $B$ except $B_{nACD}$ (and its permutations),  can be written as linear combinations of $\tilR$, but since $\tilY$ and $B$ come up naturally as  symmetries of the curvature tensor we have kept them as a separate vanishing quantities. Similarly, $\tilR_{nn}$ can be written algebraically in terms of the other Ricci components, but we have kept it as a separate unknown.
			
			By definition, the fluid and geometric vanishing quantities are zero initially, so to show that they vanish during the evolution we need to derive a system of differential equations relating them. For the fluid quantities, the system in similar to the corresponding one in \cite{MSW1} with the difference that now the geometric vanishing quantities also appear as source terms. A cumbersome but routine calculation (where it is easier to first use a covariant formulation using $\nabla= D+\Gamma$ as usual) gives the following equations, where we have used the notation $D_V=\Theta^JD_J$:
			\begin{align}\label{eq:vanishing1}
				\begin{cases}
					\Box \omega_{IJ}=F_{\omega,IJ}\qquad&\mathrm{in~}\Omega\\
					(D_V^2-\frac{1}{2}(\nabla^K\Sigma) D_K)\omega_{IJ}=F_{\omega,bdry, IJ}\qquad&\mathrm{on~}\partial\Omega
				\end{cases},
			\end{align}
			and
			\begin{align}\label{eq:vanishing2}
				\begin{cases}
					\Box Y_I = F_{Y,I}\qquad&\mathrm{in~}\Omega\\
					Y_I=0\qquad&\mathrm{on~}\partial\Omega
				\end{cases}.
			\end{align}
			Here the source term $F_{\omega}$ depends algebraically on
			\begin{align*}
				\begin{split}
					\Delta,\quad \omega,\quad \nabla\omega,\quad Q,\quad \nabla Q,\quad  \nabla X,
				\end{split}
			\end{align*}
			$F_{\omega,bdry}$ depends algebraically on
			\begin{align*}
				\begin{split}
					\omega, \quad D_V\omega,\quad X, \quad \nabla X, \quad \nabla Y,
				\end{split}
			\end{align*}
			and $F_Y$ depends algebraically on
			\begin{align*}
				\begin{split}
					\tilS,\quad  \nabla \partial_t\tilS, \quad \nabla \partial_t\tilY, \quad \Delta,\quad \nabla Q, \quad \omega,\quad \nabla\omega,\quad X,\quad\nabla X,\quad \nabla^2X,\quad A,\quad \nabla A.
				\end{split}
			\end{align*}
			The remaining vanishing fluid quantities can be related to $\omega$ and $Y$ using the transport equations
			\begin{align}\label{eq:vanishing3}
				\begin{split}
					\partial_t\tilDelta = F_\tilDelta,
					\quad\partial_t X_I=F_{X,I},
					\quad\partial_t S = F_S,
					\quad\partial_t A = F_A,
				\end{split}
			\end{align}
			where $F_\tilDelta$ depends algebraically on 
			\begin{align*}
				\begin{split}
					\nabla\omega,\quad\tilR,\quad Q,
				\end{split}
			\end{align*}
			$F_X$ depends algebraically on 
			\begin{align*}
				\begin{split}
					Y,\quad \omega,
				\end{split}
			\end{align*}
			$F_S$ depends algebraically on $X$, and $F_A$ depends algebraically on
			\begin{align*}
				\begin{split}
					\tilR,\quad X,\quad\nabla X,\quad B,\quad \omega.
				\end{split}
			\end{align*}
			For the geometric vanishing quantities, the main system of differential equations is for the unknowns $\Delta^{An}$ and $Z^{An}$. One can show that for each choice of $A$, the corresponding eight unknowns $\Delta^{An}_B$, $\Delta^{An}_n$, $Z^{An}_{012}$, $Z^{An}_{BVn}$, $B< C$, satisfy a coupled first order hyperbolic system, which can be used to prove $L^2$ and elliptic estimates. Similarly, since these quantities vanish across the boundary, this system can be differentiated to yield, after tedious calculations using similar considerations as in Section~\ref{subsec:frameeqs}, the following wave equations valid on $\cup_{0\leq t\leq T}\Sigma_t$:
			\begin{align}\label{eq:vanishing4}
				\begin{split}
					\Box Z^{An}_{IJK}= F_{Z,IJK}^{An},\qquad \Box \Delta^{An}_{I}= F_{\Delta,I}^{An},
				\end{split}
			\end{align}
			where $F_Z^{An}$ depends algebraically on
			\begin{align*}
				\begin{split}
					Z, \quad \nabla Z,\quad \chi_\Omega\omega, \quad \chi_{\Omega}\nabla\omega,\quad B, \quad \nabla B, \quad \Delta,\quad \nabla \Delta,
				\end{split}
			\end{align*}
			and $F_\Delta^{An}$ depends algebraically on
			\begin{align*}
				\begin{split}
					\Delta, \quad \nabla \Delta, \quad Z,\quad\nabla Z,\quad \chi_\Omega \omega,\quad \chi_\Omega\nabla\omega,\quad \tilY.
				\end{split}
			\end{align*}
			Finally, the remaining vanishing geometric quantities can be related to $Z$ and $\Delta$ using the following transport and first order hyperbolic systems:
			\begin{align}\label{eq:vanishing5}
				\begin{split}
					\partial_tQ_{IJ}^K=F_{Q,IJ}^K,\quad \partial_t\tilS_{IJKL}=F_{\tilS,IJKL},
				\end{split}
			\end{align}
			and (here $0, 1, 2$ are tangential components)
			\begin{align}\label{eq:vanishing6}
				\begin{split}
					&D_{0}B_{nACD}=F_{B,nACD}, \quad D_0 Z^{AB}_{IJK}= F^{AB}_{Z,IJK},\quad D_0 \Delta^{AB}_I= F_{\Delta,I}^{AB},\\
					&D_{0}\tilR_{nA}=F_{\tilR,nA},\quad D_{0}\tilR_{10}=F_{\tilR,10},\\
					&D_{0}\tilR_{20}-D_{1}\tilR_{21}=F_{\tilR,20},\quad D_{0}\tilR_{21}-D_{1}\tilR_{20}=F_{\tilR,21}.
				\end{split}
			\end{align}
			Here $F_Q$ depends algebraically on 
			\begin{align*}
				\begin{split}
					B,\quad \tilY,\quad \tilS,\quad Q,
				\end{split}
			\end{align*}
			$F_\tilS$ depends algebraically on 
			\begin{align*}
				\begin{split}
					Z,\quad \tilS,\quad \tilY,\quad \nabla \tilY \quad B,\quad Q,
				\end{split}
			\end{align*}
			$F_\tilB$ depends algebraically on $\Delta$, $F_\tilR$ depends algebraically on $\tilR$ and $Z$, $F_Z^{AB}$ and $F_\Delta^{AB}$ depend algebraically on $\tilY, \Delta, B, Z$. Equations \eqref{eq:vanishing1}, \eqref{eq:vanishing2}, \eqref{eq:vanishing3}, \eqref{eq:vanishing4}, \eqref{eq:vanishing5}, \eqref{eq:vanishing6} are the desired system of equations for the vanishing unknowns. By commuting $\partial_t$ derivatives with these equations and arguing as in  Section~\ref{sec:apriori}, it follows that the quantities appearing in \eqref{def vanishing fluid} and \eqref{def vanishing geom} vanish for all $t\leq T$ completing the proof of the theorem.
		\end{proof}
\bibliographystyle{plain}
\bibliography{ref}

\bigskip

\centerline{\scshape Shuang Miao}
\smallskip
{\footnotesize
	\centerline{School of Mathematics and Statistics, Wuhan University}
	\centerline{Wuhan, Hubei, 430072, China}
	\centerline{\email{shuang.m@whu.edu.cn}}
} 

\medskip

\centerline{\scshape Sohrab Shahshahani}
\medskip
{\footnotesize
	\centerline{Department of Mathematics and Statistics, University of Massachusetts}
	\centerline{Lederle Graduate Research Tower, 710 N. Pleasant Street,
		Amherst, MA 01003-9305, U.S.A.}
	\centerline{\email{sohrab@math.umass.edu}}
}
\end{document}